\documentclass[11pt,twoside]{amsart}

\usepackage{xspace}
\usepackage{amssymb}
\usepackage{bbm}
\usepackage{graphicx}
\usepackage{url}
\usepackage{pifont}

\numberwithin{equation}{section}

\newcommand{\boundary}[1]{{\del #1}}

\newcommand{\abs}[1]{{\lvert#1\rvert}}

\newcommand{\half}{\tfrac{1}{2}}

\newcommand{\ol}[1]{\overline{#1}}

\newcommand{\mi}{\bbi\xspace}

\DeclareMathSymbol{\varnothing}{\mathord}{AMSb}{"3F}

\DeclareMathOperator{\Imag}{Im}
\DeclareMathOperator{\Order}{O}

\DeclareMathOperator{\Real}{Re}

\DeclareMathOperator{\sign}{sign}

\DeclareMathOperator{\cross}{\times}

\DeclareMathOperator{\del}{\partial\!}

\DeclareMathOperator{\divides}{|}

\DeclareMathOperator{\id}{\mathbbm{1}}

\DeclareMathOperator*{\ord}{ord}

\DeclareMathOperator{\sech}{sech}
\DeclareMathOperator{\suchthat}{|}

\DeclareMathOperator{\trace}{tr}

\DeclareMathOperator{\curve}{Y}
\DeclareMathOperator{\cc}{C}

\DeclareMathOperator{\jx}{\mathbf{x}}
\DeclareMathOperator{\jy}{\mathbf{y}}
\DeclareMathOperator{\jk}{\mathrm{q}}
\DeclareMathOperator{\jq}{\mathrm{k}}
\DeclareMathOperator{\jh}{\mathrm{h}}
\newcommand{\sql}[1]{\lambda_{#1}^{1/2}}

\newcommand{\bbC}{\mathbb{C}}

\newcommand{\bbE}{\mathbb{E}}

\newcommand{\bbP}{\mathbb{P}}
\newcommand{\bbQ}{\mathbb{Q}}
\newcommand{\bbR}{\mathbb{R}}
\newcommand{\bbS}{\mathbb{S}}
\newcommand{\bbT}{\mathbb{T}}

\newcommand{\bbZ}{\mathbb{Z}}

\newcommand{\calB}{\mathcal{B}}
\newcommand{\calC}{\mathcal{C}}
\newcommand{\calD}{\mathcal{D}}

\newcommand{\calI}{\mathcal{I}}

\newcommand{\calL}{\mathcal{L}}

\newcommand{\calR}{\mathcal{R}}

\newcommand{\calT}{\mathcal{T}}

\newcommand{\calW}{\mathcal{W}}

\newcommand{\calZ}{\mathcal{Z}}

\newcommand{\ep}{\epsilon}

\newcommand{\AND}{\quad\text{and}\quad}
\newcommand{\OR}{\quad\text{or}\quad}

\DeclareMathOperator{\jacobiDN}{dn}

\newcommand{\MatrixGroup}[1]{{\mathrm{#1}}}

\newcommand{\matSL}[2]{\MatrixGroup{SL}_{#1}{#2}}

\newcommand{\matSU}[2]{\MatrixGroup{SU}_{#1}{#2}}

\newcommand{\matsl}[2]{\MatrixGroup{sl}_{#1}{#2}}
\newcommand{\matsu}[2]{\MatrixGroup{su}_{#1}{#2}}

\theoremstyle{plain}
\newtheorem{theorem}{Theorem}[section]
\newtheorem*{theorem*}{Theorem}
\newtheorem{corollary}[theorem]{Corollary}
\newtheorem*{corollary*}{Corollary}
\newtheorem{proposition}[theorem]{Proposition}
\newtheorem*{proposition*}{Proposition}
\newtheorem{lemma}[theorem]{Lemma}
\newtheorem*{lemma*}{Lemma}
\newtheorem{example}[theorem]{Example}
\newtheorem*{example*}{Example}
\newtheorem{definition}[theorem]{Definition}
\newtheorem*{definition*}{Definition}

\newtheorem*{notation*}{Notation}
\newtheorem{remark}[theorem]{Remark}
\newtheorem*{remark*}{Remark}

\numberwithin{figure}{section}

\newcommand{\DOTC}[2]{\langle{#1},\,{#2}\rangle}
\newcommand{\DOTR}[2]{\langle\langle{#1},\,{#2}\rangle\rangle}

\newcommand{\JK}{{\mathbf{K}}}
\newcommand{\JE}{{\mathbf{E}}}

\newcommand{\OMEGA}{\omega}
\newcommand{\NU}{\nu}

\newcommand{\Vconst}{\nu}
\newcommand{\Metric}{v}
\newcommand{\Period}{2\JK'}

\newcommand{\Ratio}{\varrho}

\newcommand{\Bpoint}{\jq}
\newcommand{\olBpoint}{\jq}
\newcommand{\Xzero}{\jq}
\newcommand{\Xone}{\jk}
\newcommand{\Xtwo}{\jh}
\newcommand{\XzeroDot}{\dot{\jq}}
\newcommand{\XoneDot}{\dot{\jk}}
\newcommand{\XtwoDot}{\dot{\jh}}
\newcommand{\XzeroDdot}{\ddot{\jq}}

\newcommand{\AngleZero}{{\chi_0}}
\newcommand{\AngleOne}{{\chi_1}}
\newcommand{\AngleTwo}{{\chi_2}}

\newcommand{\FEQ}{F}

\newcommand{\Turn}{\ell}
\newcommand{\TurnZero}{\ell_0}
\newcommand{\TurnOne}{\ell_1}
\newcommand{\TurnTwo}{\ell_2}
\newcommand{\maptau}{\mathbf{s}}

\newcommand{\bbi}{\mathbbm{i}}
\newcommand{\bbj}{\mathbbm{j}}
\newcommand{\bbk}{\mathbbm{k}}

\newcommand{\coloneq}{:=}

\newcommand{\cmc}{{\sc{cmc}}\xspace}

\newcommand{\sym}{{sym}\xspace}

\newcommand{\Rstar}{{\bbR^\times}}
\newcommand{\Cstar}{{\bbC^\times}}

\newlength{\PWIDTH}
\newlength{\PSPACE}
\newcommand{\Diagonal}{\mathfrak{k}}
\newcommand{\OffDiagonal}{\mathfrak{p}}

\newcommand{\GK}{K}

\newcommand{\spacecomma}{\,\,,}
\newcommand{\spaceperiod}{\,\,.}
\newcommand{\interspace}{\qquad}

\title[Flows of equivariant {\sc{cmc}} tori in $\bbS^3$]
{Flows of constant mean curvature tori in the 3-sphere: \\The
equivariant case}
\author{M.~Kilian, M.~U.~Schmidt and N.~Schmitt}
\address{M. Kilian, Department of Mathematics,
University College Cork, Ireland. \newline M. U. Schmidt, Institut
f\"ur Mathematik, Universit\"at Mannheim, Germany.\newline N.
Schmitt, Mathematisches Institut, Universit\"at T\"ubingen,
Germany.}
\email{m.kilian@ucc.ie; schmidt@math.uni-mannheim.de;
nschmitt@mathematik.uni-tuebingen.de\,.}


\begin{document}


\begin{abstract}
We present a deformation for constant mean curvature tori
in the 3-sphere. We show that the moduli space of equivariant constant mean curvature tori in the 3-sphere is connected, and we classify the minimal, the embedded, and the Alexandrov embedded tori therein. We conclude with an instability result.
\end{abstract}


\maketitle

\section*{Introduction}

The euclidean 3-sphere $\bbS^3$ admits compact embedded minimal surfaces of any genus \cite{Law:S3, KarPS, KapY}. For simple examples like a great 2-sphere or the (minimal) Clifford torus there is a smooth deformation through constant mean curvature (\cmc) surfaces with the same topology, which can be expressed in terms of changing radii. For other minimal tori, as well as higher genus compact embedded minimal surfaces in $\bbS^3$ it might be useful to have a general deformation technique. We do this for the case of \cmc tori, and present a smooth topology preserving deformation for \cmc tori in $\bbS^3$ in Theorem~\ref{th:nocommon}. By this theorem the moduli space of \cmc tori in $\bbS^3$ is locally one dimensional. To get an idea of the global structure of the moduli space we turn to equivariant \cmc tori in $\bbS^3$ \cite{HsiL, Uhl:equi, BurK}, which are either flat or a truncation of some member of an associated family of a Delaunay surface \cite{DoCD}. We show that the moduli of equivariant \cmc tori in $\bbS^3$ is a connected infinite graph whose edges are parameterized by the mean curvature, and by  flowing through this moduli space of equivariant \cmc tori we classify the minimal, the embedded, and the Alexandrov embedded tori therein.

Amongst harmonic maps \cite{Poh, Uhl, Uhl:connection} there is an important class consisting of harmonic maps of finite type \cite{BurFPP, hit:tor, PinS}. To a harmonic map of finite type there corresponds an associated algebraic curve whose compactification is called the spectral curve. The genus of the curve is called the spectral genus, and denoted by $g$. The crucial fact that makes it possible to adapt the Whitham deformation technique \cite{MaOs,Kri_whi,GriS1} to the case of \cmc tori in $\bbS^3$, is that a \cmc torus in $\bbS^3$ always has finite spectral genus \cite{hit:tor, PinS}. The spectral curve of a \cmc torus is a double cover of the Riemann sphere with $2g+2$ many branch points. During the deformation the spectral curve changes such that two branch points remain fixed, while the other $2g$ branch points may move around. The closing conditions involve a choice of two double points on the real part of the spectral curve: we call these the \sym points, and to ensure that the topology of the surface stays intact during the deformation, the flow of the \sym points has to be controlled.

In the smooth deformation family of the Clifford torus there is a $\bbZ$-family which allow a deformation into absolute cohomogeneity one rotational embedded \cmc tori. Such a deformation is possible when in addition to the two \sym points there is a further double point on the real part of the spectral curve. By opening up this additional double point and moving the resulting two branch points off the real part, the spectral curve becomes a double cover of the Riemann sphere branched now at four points: It has spectral genus $g=1$ and is the spectral curve of a Delaunay surface, and the corresponding \cmc torus is a truncation of a Delaunay surface in $\bbS^3$. We show that at the end of the flow the new branch points pair-wise coalesce with the two fixed branch points. Hence in the limit the coalescing pairs of branch points disappear, and the limit curve is an unbranched double cover of the sphere: the spectral curve of a bouquet of spheres. In the rotational case, see also \cite{HynPM}, our deformation corresponds to pinching the neck of a Delaunay surface, starting at a flat torus and continuing through to a bouquet of spheres. Thus the connected component of the Clifford torus is an infinite comb: The spine ($g=0$) consists of embedded flat \cmc tori parameterized by the mean curvature, and each tooth ($g=1$) of embedded Delaunay tori ends in a bouquet of spheres. By considering covers of Clifford tori the moduli space of rotational \cmc tori is a $\bbZ^2$--family of such combs. It turns out that each bouquet of spheres occurs exactly twice in this moduli space (Theorem~\ref{thm:sphere-bouquet}), so that we may glue the two families together there. Thus the moduli space of rotational \cmc tori in $\bbS^3$ is an infinite connected graph.

A similar picture transpires in the non-rotational case. In each isogeny class there is a sequence of $g=0$ tori that can be deformed into $g=1$ tori. In the non-rotational case a $g=1$ deformation family stays away from bouquets of spheres, and we prove that every $g=1$ deformation family begins and ends at a $g=0$ torus (Theorem~\ref{thm:moduli-space}). The above results combined give that every deformation family of absolute cohomogeneity one \cmc tori ends at a flat \cmc torus. The classification of absolute cohomogeneity one \cmc tori is thus reduced to that of spectral curves of flat \cmc tori with a double point on the real part; this initial data is classified and interpreted geometrically. We classify the equivariant minimal tori, as well as the embedded and Alexandrov embedded equivariant \cmc tori, and prove that the minimal Clifford torus is the only embedded minimal equivariant torus in the 3-sphere. We also show that the spectral curve of an equivariant \cmc torus has no double points off the real part (Theorem~\ref{th:no_equi_bub}), which implies that there can not be a Bianchi-B\"acklund transform of an equivariant \cmc torus into a \cmc torus. We prove that the moduli space of equivariant \cmc tori in $\bbS^3$ is connected, and conclude the paper by showing that minimal tori of spectral genus one are all unstable extrema of the Willmore energy. Throughout the text we provide graphics of some of the surfaces under discussion. More images and videos of deformation families can be viewed at the website \cite{Sch:gallery}.


\section{Spectral curve} \label{sec:spectral_curve}

We start by recalling the description of {\sc{cmc}} tori in terms of spectral curves and abelian differentials \cite{Bob:tor,hit:tor, McI:tor,KilS_osaka}. We then adapt a deformation technique from \cite{GriS1} and prove that any generic \cmc torus in $\bbS^3$ lies in a smooth family of \cmc tori in $\bbS^3$.

Let $\curve$ be a hyperelliptic Riemann surface with meromorphic function $\lambda$ of degree two and with branch points over $\lambda = 0 \,(y^+)$ and $\lambda = \infty \,(y^-)$. Then $\curve$ is the \emph{spectral curve} of an immersed {\sc{cmc}} torus in $\mathbb{S}^3$ if and only if the following four conditions hold:
\begin{enumerate}
\item Besides the hyperelliptic involution $\sigma$, the surface
$\curve$ has two further anti-holomorphic involutions $\eta$ and $\varrho = \eta \circ \sigma = \sigma \circ \eta$, such that $\eta$ has no fix points and $\eta(y^+) = y^-$.
\item There exist two non-zero holomorphic functions
$\mu_1,\,\mu_2$ on $\curve\setminus\{y^+,\,y^-\}$ such that for $i=1,\,2$
\begin{align*}
  \mu_i \circ \sigma &=\mu_i^{-1}&
  \mu_i \circ \eta &=\bar{\mu}_i&
  \mu_i \circ \varrho &=\bar{\mu}_i^{-1}.
\end{align*}
\item The forms $d \ln \mu_i$ are meromorphic differentials
of the second kind with double poles at $y^\pm$. The singular parts at $y^+$ respectively $y^-$ of these two differentials are linearly independent.
\item There are four fixed points $y_1,\,y_2 =
\sigma(y_1),\,y_3,\,y_4 = \sigma(y_3)$ of $\varrho$, such that the
functions $\mu_1$ and $\mu_2$ are either $1$ or $-1$ there.
\end{enumerate}
\begin{definition}
Given the spectral curve of an immersed {\sc{cmc}} torus in $\mathbb{S}^3$, let $\lambda_1\in\mathbb{S}^1$ denote the value $\lambda(y_1) = \lambda(y_2)$, and $\lambda_2\in\mathbb{S}^1$ the value $\lambda(y_3) = \lambda(y_4)$ at the four fixed points of $\varrho$ where $\mu_1$ and $\mu_2$ are either $1$ or $-1$. We call $\lambda_1$ and $\lambda_2$ the \sym points.
\end{definition}
For $\mi = \sqrt{-1}$, the mean curvature of the corresponding \cmc torus in terms of the \sym points is
\begin{equation} \label{eq:cmc-I-II}
    H = \mi\,\frac{\lambda_2+\lambda_1}{\lambda_2-\lambda_1}\spaceperiod
\end{equation}
We shall describe spectral curves of \cmc tori in $\mathbb{S}^3$
via hyperelliptic surfaces of the form
\begin{align*}
  \NU^2 = \lambda\,a(\lambda)
\end{align*}
where $a \in \bbC^g[\lambda]$ is a polynomial of degree $g$ and %
$$
\bar{\lambda}^{-2g}\bar{a}(\lambda) = a(\bar{\lambda}^{-1})
\quad\mbox{ and }
\lambda^{-g}a(\lambda)\geq 0\quad\mbox{ for }|\lambda|=1.
$$
The involutions of the spectral curve are
\begin{equation} \label{eq:involutions} \begin{split}
  \sigma\,(\lambda,\,\NU) &=(\lambda,\,-\NU)\,,\\
  \eta\,(\lambda,\,\NU) &= (\bar{\lambda}^{-1},\, -\bar{\NU}\bar{\lambda}^{-g-1})\,,\\
  \varrho\,(\lambda,\,\NU) &= (\bar{\lambda}^{-1},\,\bar{\NU}\bar{\lambda}^{-g-1})\,. \end{split}
\end{equation}
Making use of a rotation of $\lambda$ and a rescaling of $\NU$ we may
assume that $a$ is a polynomial with highest coefficient one. The
meromorphic differentials $d\ln\mu_i$ have the form
\begin{equation} \label{eq:b_i}
d\ln\mu_i\coloneq\pi\frac{b_i(\lambda)}{\NU}\frac{d\lambda}{\lambda}
\end{equation}
with polynomials $b_i \in \bbC^{g+1}[\lambda]$ of degree $g+1$ and
$\bar{\lambda}^{-g-1}\bar{b}_i(\lambda)=b_i(\bar{\lambda}^{-1})$.

We next describe a one parameter family of deformations
of the spectral curve, that depends on a deformation parameter $t$.
We view all functions on the corresponding spectral curves as
functions in the variables $\lambda$ and $t$.

Since the path integrals of the differentials $d \ln \mu_i$ along
all cycles in $H_1(\curve,\,\bbZ)$ are integer multiples of
$2\pi\mi$, these differentials do not depend on the deformation
parameter $t$. Further, see \cite{Mir} (ch.3, Prop. 1.10), a meromorphic function $f$ on a hyperelliptic Riemann surface given by $\NU^2 = h(\lambda)$ is of the form $f(\lambda) = r(\lambda) + \NU \,s(\lambda)$ with rational functions $r,\,s$. Hence both $\del_{\,t} \ln \mu_i$ are global meromorphic functions on $\curve$ with only possible poles at the
branch points of $\curve$. More precisely these meromorphic
functions are of the form
\begin{equation} \label{eq:c_i}
  \del_{\,t} \ln \mu_i =\pi\frac{\mi \,c_i(\lambda)}{\NU}
\end{equation}
with polynomials $c_i\in \bbC^{g+1}[\lambda]$ of degree $g+1$ and
$\bar{\lambda}^{-g-1}\bar{c}_i(\lambda)=c_i(\bar{\lambda}^{-1})$.

Integrability $\del^{\,2}_{\,t\lambda} \ln \mu_i = \del^{\,2}_{\,\lambda
t} \ln \mu_i$ reads $\del_{\,t} \left( \lambda^{-1}\NU^{-1}b_i
\right) =
  \del_{\,\lambda} \NU^{-1}\mi c_i$ which yields
\begin{equation}\label{eq:adot}
    2 a(\lambda)\dot{b}_i(\lambda)-\dot{a}(\lambda)b_i(\lambda)=
    \mi(2\lambda a(\lambda)c_i'(\lambda)-
    \lambda a'(\lambda)c_i(\lambda)-a(\lambda)c_i(\lambda)).
\end{equation}
Here dash and dot denote the derivatives with respect to $\lambda$ respectively $t$.

The differential
$$
    \Omega = (\del_{\,t} \ln \mu_1)\,d\ln \mu_2
  - (\del_{\,t} \ln \mu_2 ) \,d\ln \mu_1
$$
is a meromorphic $1$-form on $\curve$ with only poles at most of order three at $\lambda = 0,\,\infty$, and roots at the sym points $\lambda_1,\,\lambda_2$. Further, since $\eta^*\bar{\Omega} =
\Omega$, $\varrho^*\bar{\Omega} = \Omega$ and $\sigma^*\Omega =
\Omega$, we conclude that
\begin{equation*}
  \Omega = \pi^2 \cc_1\frac{(\lambda - \lambda_1)(\lambda -\lambda_2)}
       {\lambda\sqrt{\lambda_1\,\lambda_2}}
   \frac{d\lambda}{\mi\lambda}
\end{equation*}
with a real function $\cc_1 = \cc_1(t)$. Using \eqref{eq:b_i} and \eqref{eq:c_i} we obtain
\begin{equation} \label{eq:c1c2}
  b_1(\lambda)c_2(\lambda)-b_2(\lambda)c_1(\lambda)=
   \cc_1 a(\lambda)\frac{(\lambda-\lambda_1)(\lambda-\lambda_2)}
     {\sqrt{\lambda_1\lambda_2}}\,,
  \end{equation}
and combining the above gives
\begin{equation} \label{eq:diff_ode}
(\del_{\,t} \ln \mu_1)\,d\ln \mu_2
  - (\del_{\,t} \ln \mu_2 ) \,d\ln \mu_1 =
  \pi^2 \cc_1\frac{(\lambda - \lambda_1)(\lambda -\lambda_2)}
       {\lambda \,\sqrt{\lambda_1\,\lambda_2}}
   \frac{d\lambda}{\mi\lambda}\,.
\end{equation}
To preserve the topology during the flow, the values of $\ln\mu_i$ should be fixed at the two sym points $\lambda=\lambda_1$ and $\lambda=\lambda_2$. Consequently $\del_{\,t} \ln\mu_i |_{\lambda = \lambda_j} = \bigl( (\del_\lambda \ln\mu_i) \,\dot{\lambda} + \del_{\,t} \ln\mu_i \bigr)|_{\lambda = \lambda_j} =0$, which gives
\begin{equation} \label{eq:sym-point-deformation}
    \dot{\lambda}_j = -\left.\frac{\del_{\,t}\ln\mu_i}
                             {\del_\lambda\ln\mu_i}
                        \right|_{\lambda=\lambda_j}\,.
\end{equation}
\begin{theorem} \label{th:nocommon}
Let $\curve$ be a genus $g$ spectral curve of a {\sc{cmc}} torus in
$\bbS^3$. If the two differentials $d \ln \mu_i$ for $i=1,\,2$ have
no common roots, then $\curve$ is contained in a unique smooth one
dimensional family of spectral curves of {\sc{cmc}} tori in
$\mathbb{S}^3$.
\end{theorem}
\begin{proof}
If the roots of $b_1$ and $b_2$ are pairwise distinct, then the
$2g+2$ values of equation \eqref{eq:c1c2} at these roots uniquely
determine the values of $c_1$ and $c_2$ there. Therefore the
equation \eqref{eq:c1c2} determines the polynomials $c_1$ and $c_2$
uniquely up to a real multiple of $b_1$ and $b_2$. The choice
$c_1=b_1$ and $c_2=b_2$ in \eqref{eq:c_i} and \eqref{eq:adot}
corresponds to a rotation of $\lambda$. For given $c_1$ and $c_2$
the equations \eqref{eq:adot} determine uniquely the derivatives of
all roots of the polynomials $a$ and $b_i$. The condition
$\bar{\lambda}^{-2g}\bar{a}(\lambda)=a(\bar{\lambda}^{-1})$ together
with the assumption that the highest (and lowest) coefficient of the
polynomial $a$ has absolute value one determines the polynomial $a$
in terms of its roots up to multiplication with $\pm 1$. We conclude
that these conditions on $a$ determine $\dot{a}$, $\dot{b}_1$ and
$\dot{b}_2$ in terms of $c_1$ and $c_2$. We remark that due to
equation \eqref{eq:c1c2} the solutions $\dot{a}$ of both equations
\eqref{eq:adot} coincide. Finally the condition that the highest
(and lowest) coefficient of $a$ is qual to one, fixes the rotations. Therefore there exist unique solutions $c_1$ and $c_2$ of \eqref{eq:c1c2}.
\end{proof}
We expect that one can pass through common zeroes of the $b_i's$, therefore making the deformation global, but this will be considered elsewhere. Below we will restrict to the cases of spectral genera $g=0$, where this is trivially true, and $g=1$ where the roots of $b_1$ and $b_2$ turn out to always be distinct throughout the flow (Corollary~\ref{th:no_common_zeroes}).


\section{Flows of equivariant cmc tori in the 3-sphere}

The double cover of the isometry group of $\bbS^3 \cong \mathrm{SU}_2$ is $\mathrm{SU}_2 \times \mathrm{SU}_2$ via the action $P \mapsto F\,P\,G^{-1}$. A surface is \emph{equivariant} if it is preserved set-wise by a one-parameter family of isometries. To an equivariant surface we associate two \emph{axes}: These are geodesics which are fixed set-wise by the one-parameter family of isometries. An equivariant surface is rotational precisely when one of its axes is fixed point-wise. Equivariant {\sc{cmc}} surfaces have spectral genus zero or one \cite{BurK}. By a theorem of DoCarmo and Dajczer \cite{DoCD}, an equivariant {\sc{cmc}} surface is a member of an associated family of a Delaunay surface, so up to isometry determined by an elliptic modulus and its \sym points. Hence an equivariant \cmc surface in the 3-sphere is parameterized by its elliptic modulus, its mean curvature, and its associated family parameter. We will express the differential equations \eqref{eq:diff_ode} and
\eqref{eq:sym-point-deformation} in terms of the three coordinates
\begin{equation}\label{eq:coor0}
    (\jq,\,\jk,\,\jh) \in [-1,\,1]^3
\end{equation}
where $\jq$ is the elliptic modulus, and $\jk,\,\jh$ are defined in terms of the \sym points as
\begin{equation}\label{eq:coor}
    \jk \coloneq \frac{1}{2} \left(\sqrt{\lambda_1\lambda_2}\, + \, 1/\sqrt{\lambda_1\lambda_2}\,\,\right)
    \quad \mbox{ and } \quad
    \jh \coloneq \frac{1}{2}\left(\sqrt{\frac{\lambda_1}{\lambda_2}}\, + \, \sqrt{\frac{\lambda_2}{\lambda_1}} \,\,\right)\,.
\end{equation}
The mean curvature $H$ in \eqref{eq:cmc-I-II} can then be expressed as
\begin{equation} \label{eq:H_and_h}
    H = \frac{\jh}{\sqrt{1-\jh^2}}\,.
\end{equation}
The following identities will be used below:
\begin{equation} \label{eq:sqrt_q_formula} \begin{split}
    2\sqrt{\jk^2 - 1} &= \sqrt{\lambda_1\lambda_2}\, - \, 1/\sqrt{\lambda_1\lambda_2} \,, \quad 2\sqrt{\jh^2 - 1} = \sqrt{\lambda_1/\lambda_2}\, - \, \sqrt{\lambda_2/\lambda_1}  \,,\\
    4(\jk^2 - \jh^2) &= \left( \lambda_1 -\lambda_1^{-1} \right) \left(\lambda_2 - \lambda_2^{-1} \right)\,, \quad 4\jk\sqrt{\jk^2 -1} = \lambda_1\lambda_2 - \lambda_1^{-1}\lambda_2^{-1} \\
    4\jh\sqrt{\jk^2-1} &= \lambda_1 - \lambda_1^{-1} + \lambda_2 - \lambda_2^{-1}\,, \quad 4\jk\sqrt{\jh^2-1} = \lambda_1 - \lambda_1^{-1}-\lambda_2+\lambda_2^{-1} \\
    4\jh\sqrt{\jh^2-1} &= \lambda_1\lambda_2^{-1} - \lambda_1^{-1}\lambda_2\,.
    \end{split}
\end{equation}
The derivatives of $\jk,\,\jh$ with respect to the flow parameter are
\begin{equation} \label{eq:qdot_hdot}
    \dot{\jk} = \frac{\sqrt{\jk^2-1}}{2}
    \left(\frac{\dot{\lambda}_1}{\lambda_1}+
    \frac{\dot{\lambda}_2}{\lambda_2}\right)\,, \qquad
    \dot{\jh} = \frac{\sqrt{\jh^2-1}}{2}
    \left(\frac{\dot{\lambda}_1}{\lambda_1}-
    \frac{\dot{\lambda}_2}{\lambda_2}\right)\,.
\end{equation}
%
%
%
\subsection{Spectral genus zero.}
In the spectral genus zero case, $a \equiv 1$, the spectral curve is $\nu^2 = \lambda$, and the elliptic modulus $\jq \equiv \pm 1$. The functions $b_i$ in \eqref{eq:b_i} for $i=1,\,2$ are polynomial in $\lambda$ of degree one, and from $\bar{\lambda}^{-1}\bar{b}_i(\lambda)=b_i(\bar{\lambda}^{-1})$ we conclude that $b_i = \beta_i\,\lambda + \bar{\beta}_i$ for some smooth complex valued functions $t \mapsto \beta_i(t)$. Thus $d\ln\mu_i$ integrates to
\begin{equation*}
    \ln \mu_j = \frac{2 \pi (\beta_j\,\lambda - \bar{\beta}_j)}{\nu}\, ,
\end{equation*}
and consequently the functions $c_i,\,i=1,\,2$ in \eqref{eq:c_i} are $c_i = 2\mi (\dot{\bar\beta}_i - \dot\beta_i\lambda)$. We may fix one of the \sym points, and assume without loss of generality that $\lambda_1 \equiv 1$ during the flow. Then $\left.\partial_{\,t}\ln\mu_j\right|_{\lambda =1} =0$, which is equivalent to $c_i(\lambda=1) =0$, or equivalently $\dot{\beta}_i \in \bbR$.
Thus equation \eqref{eq:diff_ode} reads
\begin{equation*}
    2\mi\dot{\beta}_2(\beta_1\lambda+\bar{\beta}_1) - 2\mi\dot{\beta}_1(\beta_2\lambda+\bar{\beta}_2) =
    \frac{C_1}{\sqrt{\lambda_2}}\,(\lambda_2 -\lambda)\,.
\end{equation*}
Evaluating this at $\lambda = -\bar{\beta}_i/\beta_i$ gives
\begin{equation*}
    \dot{\beta}_i = \frac{C_1\,(\beta_i \lambda_2^{1/2} + \bar{\beta}_i \lambda_2^{-1/2})}{2\mi\,(\bar{\beta}_1\beta_2 - \beta_1\bar{\beta}_2)}
\end{equation*}
From \eqref{eq:sym-point-deformation} we get
\begin{equation*}
    \frac{\dot{\lambda}_2}{\lambda_2} = \frac{c_1}{\mi b_1} = \frac{2\dot{\beta}_1(1-\lambda_2)}{\beta_1\lambda_2 +\bar{\beta}_1}
    =  C_1 \frac{\lambda_2^{1/2}-\lambda_2^{-1/2}}{\mi (\beta_1\bar{\beta}_2 - \bar{\beta}_1\beta_2)}
\end{equation*}
Set $C_1 \coloneq \mi (\beta_1\bar{\beta}_2 - \bar{\beta}_1\beta_2)$. Then together with \eqref{eq:sqrt_q_formula} and the fact that $\jk=\jh$ when $\lambda_1 \equiv 1$, the flow equations \eqref{eq:diff_ode} and \eqref{eq:sym-point-deformation} reduce to the single equation
$\dot{\jh} = \jh^2 - 1$. The solution to this equation is given by $\jh(t) = -\tanh(t + C)$ with some constant of integration $C\in\bbR$. Consequently, the argument of the \sym point $\mathrm{arg}[\lambda_2] = 2 \arccos(-\tanh(t))$ can vary over all of the interval $(0,\,2\pi)$, and is strictly monotonic. Hence also the mean curvature $H$ by equation \eqref{eq:H_and_h} is strictly monotonic, and given by $H(t) = -\sinh(t)$. In summary, we have proven the following
\begin{theorem} \label{th:cliff_fam}
Every flat \cmc torus in the 3-sphere lies in a smooth $\bbR$--family of flat \cmc tori. Each family is parameterized by the mean curvature.
\end{theorem}
%


%
\subsection{Spectral genus one.}
In the spectral genus one case, the meromorphic differentials $d\ln\mu_i$ are linear combinations of the derivative of a meromorphic function and an elliptic integral. We write these meromorphic differentials in terms of Jacobi's elliptic functions. We define an elliptic curve by
\begin{equation}\label{eq:nu}
4\NU^2=(\lambda-\jq)(\lambda^{-1}-\jq).
\end{equation}
Here $\jq\in[-1,1]$ is the real modulus. Then
\begin{equation} \label{eq:dnu}
    d\NU = \jq\frac{\lambda^{-1} - \lambda}{8\NU}\frac{d\lambda}{\lambda}\,,
\end{equation}
and the only roots of $d\NU$ are at $\lambda = \jq,\,\jq^{-1}$. In addition to the three involutions in \eqref{eq:involutions}, the elliptic curve \eqref{eq:nu} has a further holomorphic involution
\begin{equation}\label{eq:involtau}
\tau \,(\lambda,\,\NU) = (\lambda^{-1},\NU)\,.
\end{equation}
We decompose $\frac{\ln\mu_i}{\pi\mi}$ into
the symmetric and skew symmetric part with respect to
$\tau$.
The symmetric part is a real multiple of the single valued
meromorphic function $\NU$, and the skew symmetric part is a real multiple of a multi valued function $\OMEGA$ with real periods.
The first homology group of the elliptic curve \eqref{eq:nu} is
generated by a cycle around the two branch points at
$\lambda=\jq^{\pm 1}$ and the cycle
$\bbS^1=\{\lambda\mid|\lambda|=1\}$. The first cycle is symmetric with
respect to $\tau$ and the second skew symmetric. Therefore
the integral of $d\OMEGA$ along the first cycle vanishes. We assume
that the integral of $d\OMEGA$ along the second cycle is equal to
$2$. Together with the first order poles of $\OMEGA$ at $y^{\pm}$ and
the skew symmetry with respect to $\tau$ this normalization
determines $\OMEGA$ uniquely. Further, since $\int_{\bbS^1} d\ln\mu_i \in 2\pi\mi\bbZ$, we conclude that the functions $\ln\mu_i$ are linear combinations of $\NU$ and $\OMEGA$, and given by
\begin{equation}\label{eq:lnmu}
\ln\mu_1 = \pi\mi\,(x_1\,\NU + p_{10}\,\OMEGA)\,,\quad
\ln\mu_2 = \pi\mi\,(x_2\,\NU + p_{20}\,\OMEGA)\,\quad \mbox{for }\, p_{10},\,p_{20} \in \bbZ\,.
\end{equation}
We shall express $\OMEGA$ as a linear combination of complete elliptic integrals of the first and second kind. First we shall relate the curve \eqref{eq:nu} to the elliptic curve $\curve$ in Legendre's form with elliptic modulus $\tfrac{1-\jq}{1+\jq}$ given by
\begin{equation*}
    \jy^2=\left(1-\jx^2\right)\left(1-(\tfrac{1-\jq}{1+\jq})^2\jx^2\right).
\end{equation*}
Let $\JK$ and $\JE$ denote the complete integrals of the first and second kind
\begin{equation*}
\JK(\jq) \coloneq\int_0^1\tfrac{d\jx}{\sqrt{(1-\jx^2)(1-\jq^2\jx^2)}}\,,\quad
\JE(\jq) \coloneq\int_0^1\sqrt{\frac{1-\jq^2\jx^2}{1-\jx^2}}\,d\jx\,.
\end{equation*}
Since $\JE$ has first order poles at the two points over $\jx=\infty$, these two points correspond to $y^{\pm}$. The involution $\tau$ corresponds to $(\jx,\jy)\mapsto(\jx,-\jy)$. Therefore we set
\begin{align*}
\lambda&=
\frac{1+(\tfrac{1-\jq}{1+\jq})^2(1-2\jx^2)-2\frac{1-\jq}{1+\jq}\jy}
     {1-(\tfrac{1-\jq}{1+\jq})^2}\,,&
\lambda^{-1}&=
\frac{1+(\tfrac{1-\jq}{1+\jq})^2(1-2\jx^2)+2\frac{1-\jq}{1+\jq}\jy}
     {1-(\tfrac{1-\jq}{1+\jq})^2}\,,&
\NU&=\frac{\tfrac{1-\jq}{1+\jq}\jx}{1+\tfrac{1-\jq}{1+\jq}}.
\end{align*}
The integrals of the meromorphic differentials $d\ln\mu_i$ along all
cycles of $\curve$ are purely imaginary. Define
\begin{equation} \label{eq:omega-defn}
d\OMEGA\coloneq\left(\JE(\tfrac{1-\jq}{1+\jq})-\JK(\tfrac{1-\jq}{1+\jq})
\left(1-(\tfrac{1-\jq}{1+\jq})^2\jx^2\right)\right)
\frac{2d\jx}{\pi\mi\jy}.
\end{equation}
The cycle around the two branch points
$\lambda=\jq^{\pm 1}$ corresponds to the real period and the cycle
$\bbS^1$ to the imaginary period of Jacobi's elliptic functions.
The integral along the real period vanishes and due to Legendre's
relations \cite[17.3.13]{AbrS} the integral along the imaginary period
is equal to $2$.
In summary, we have found a linear combination of elliptic integrals of the first and second kind which obeys the conditions that characterize
$\OMEGA$.

In terms of the complete elliptic integrals
$\JK'=\JK'(\jq)=\JK(\sqrt{1-\jq^2})$ and
$\JE'=\JE'(\jq)=\JE(\sqrt{1-\jq^2})$ and the
functions $\lambda$ and $\NU$, and
the formulas \cite[17.3.29 and 17.3.30]{AbrS}
\begin{equation*}
    \JE(\tfrac{1-\jq}{1+\jq}) = (1+\tfrac{1-\jq}{1+\jq})\JE'+(1-\tfrac{1-\jq}{1+\jq})\JK'\,,\quad
    \JK(\tfrac{1-\jq}{1+\jq}) = \frac{2}{1+\tfrac{1-\jq}{1+\jq}}\JK'\,,
\end{equation*}
the differential $d\OMEGA$ in \eqref{eq:omega-defn} simplifies to
\begin{equation}\label{eq:omega}
    d\OMEGA=\frac{2\JE'-\jq\JK'(\lambda+\lambda^{-1})}{4\pi\NU}\frac{d\lambda}{\mi\lambda}\,.
\end{equation}
\begin{proposition}\label{th:alpha_less_beta}
The complete elliptic integrals
$\JK'$ and $\JE'$ satisfy
\begin{equation*}
    1\leq\frac{2\JE'}{1+\jq^2}<\JK'<\frac{\JE'}{|\jq|}\quad
    \mbox{for}\quad 0<|\jq|<1.
\end{equation*}
\end{proposition}
\begin{proof}
Assume first that $\jq\in(0,1)$. Since $\tau^*d\OMEGA = -d\OMEGA$, we have that
$$
    \int_{\lambda=\jq}^{\lambda=\jq^{-1}}d\OMEGA =
 2\int_{\lambda=\jq}^{\lambda=1}d\OMEGA = 0\,.
$$
Hence the function $2\JE'-\jq\JK'(\lambda+\lambda^{-1})$ has a real root in $\lambda\in(\jq,\,1)$, say at $r \in (\jq,\,1)$, and thus also at $r^{-1} \in (1,\,\jq^{-1})$. Then $2< r+1/r < k +1/k$, and together with $2\JE'-\jq\JK'(r + r^{-1})=0$ gives $2\JE' > 2\jq\JK'$ and $2\JE'<\JK'(1+\jq^2)$. Finally, from $2>\jq^2 + 1$ and $\JE' \geq 1$ (13.8.11 in \cite{Bat2}), we obtain $2\JE' \geq 1+\jq^2$.

For $\jq\in(-1,\,0)$, the function $2\JE'-\jq\JK'(\lambda+\lambda^{-1})$ has a pair of reciprocal real roots $s,\,s^{-1}$ with $-1<s<\jq^{-1}$, and analogous arguments prove the assertion in this case.
\end{proof}
\begin{corollary}\label{th:no_common_zeroes}
The differentials $d\NU$ and $d\OMEGA$ have no common zeroes.
\end{corollary}
\begin{proof}
While $d\NU$ in \eqref{eq:dnu} has only roots at $\lambda = \jq,\,\jq^{-1}$, we saw in the proof of Proposition \ref{th:alpha_less_beta} that $d\OMEGA$ has only roots $\lambda\in(\jq,\,1) \cup (1,\,\jq^{-1})$ when $\jq\in(0,1)$, and $\lambda\in(-1,\,\jq) \cup (\jq^{-1},\,-1)$ when $\jq\in(-1,\,0)$.
\end{proof}
%
%
\begin{theorem} \label{thm:flow}
Every spectral genus one \cmc torus in $\bbS^3$ lies on an integral curve of the vector field
\begin{equation}\label{eq:torus-flow}
    \left( \begin{array}{c} \dot{\jq} \\ \dot{\jk} \\ \dot{\jh} \end{array} \right) = \left(
    \begin{array}{c} \jq\,(\JE'\,\jk - \jq\,\JK'\,\jh) \\ \frac{1-\jk^2}{1-\jq^2}\,((1 + \jq^2)\JE'-2\jq^2\JK') \\
    \jq\frac{1-\jh^2}{1-\jq^2}\,(2\JE'-(1+\jq^2)\JK')) \end{array} \right)\,.
\end{equation}
The vector field $(\XzeroDot,\,\XoneDot,\,\XtwoDot)$ is analytic on the set $\calD \coloneq \{(\Xzero,\,\Xone,\,\Xtwo)\in\bbR^3\suchthat \Xzero\not=0\}$, and its zero set is $\{\Xzero^2=1\}\cap\{\Xone=\Xzero\Xtwo\}\cap\calD$.
\end{theorem}
\begin{proof}
We shall calculate the differential equations \eqref{eq:diff_ode} and
\eqref{eq:sym-point-deformation} in terms of the coordinates \eqref{eq:coor0}. From \eqref{eq:nu} we compute the derivative of $\NU$ with respect to $t$, and obtain
\begin{equation} \label{eq:nu dot}
    \dot{\NU} =\dot{\jq}\,\frac{2\jq - \lambda - \lambda^{-1}}{8\NU}\,.
\end{equation}
In agreement with the skew-symmetry of $\OMEGA$ with respect to
$\tau$ we set
\begin{equation*}
\dot{\OMEGA}=\cc_2\frac{\lambda-\lambda^{-1}}{4\pi\mi\NU}.
\end{equation*}
From \eqref{eq:lnmu} we obtain $\del_{\,t}\ln\mu_j = \pi\mi\,(\dot{x}_j\,\NU + x_j\dot{\NU} + p_{j0}\,\dot{\OMEGA})$ and $d\ln\mu_j = \pi\mi\,(x_j\,d\NU + p_{j0}\,d\OMEGA)$ for $j = 1,\,2$.
Putting all this together, the differential $\Omega = (\del_{\,t} \ln \mu_1)\,d\ln \mu_2 - (\del_{\,t} \ln \mu_2 ) \,d\ln \mu_1$ reads
\begin{equation*} \begin{split}
\Omega &=
\pi^2\bigl( (\dot{x}_1x_2-\dot{x}_2x_1)\,\NU\,\NU' + (\dot{x}_2 p_{10} - \dot{x}_1 p_{20})\,\NU\,\OMEGA' + (p_{10} x_2 - p_{20} x_1) (\dot{\NU}\,\OMEGA' - \NU'\dot{\OMEGA}) \bigr) \,\,\mbox{ with } \\
\NU\,\NU' &= \frac{\jq\,(\lambda^{-1} - \lambda)}{8\,\lambda}\,,\qquad
\NU\,\OMEGA' = \frac{2\,\JE' - \jq\,(\lambda + \lambda^{-1})\,\JK'}{4\,\pi\,\mi\,\lambda}\,,\\
\dot{\NU}\,\OMEGA' - \NU'\dot{\OMEGA} &= \frac{\cc_2\jq(\lambda - \lambda^{-1})^2 + \dot{\jq}\,(\lambda - 2\,\jq + \lambda^{-1})(\jq\,(\lambda+\lambda^{-1})\,\JK' - 2\,\JE')}{8\,\pi\,\mi\,\lambda\,(\lambda-\jq)(\lambda^{-1} - \jq)}\,.
\end{split}
\end{equation*}
Note that while $\dot{\NU}\,\OMEGA' - \NU'\dot{\OMEGA}$ has simple poles at $\lambda = \jq,\,\jq^{-1}$, but $\Omega$ does not, we conclude that the numerator of $\dot{\NU}\,\OMEGA' - \NU'\dot{\OMEGA}$ must vanish at $\lambda = \jq,\,\jq^{-1}$, giving
\begin{equation*}
    \cc_2 = \dot{\jq}\,\frac{(\jq^2 + 1)\,\JK' - 2\,\JE'}{\jq^2 - 1}\,.
\end{equation*}
The unique solution of equation \eqref{eq:diff_ode} can now be computed, and is given by
\begin{equation*} \begin{split}
&\dot{\jq} = \frac{4\pi \cc_1(\jq^2 - 1)(\jq\,\jh\,\JK' - \jk\,\JE')}
              {(x_1 p_{20} - x_2 p_{10} )(\JE'^2-\jq^2\JK'^2)}\,,\qquad
\dot{x}_1x_2-\dot{x}_2x_1 = \frac{8\cc_1\sqrt{1-\jk^2}}{\jq}\,,\\
&\dot{x}_1p_{20}-\dot{x}_2p_{10} = \frac{4\pi \cc_1(\jq\jk(\JE'-\JK') +  \jh(\JE'-\jq^2\JK'))}{\JE'^2-\jq^2\JK'^2}\,.
     \end{split}
\end{equation*}
Since $\dot{\OMEGA} = \partial_{\jq} \OMEGA\, \dot{\jq}$, the above imply that
\begin{equation} \label{eq:omega'}
    \frac{\del\OMEGA}{\del\jq} =
    \frac{(1+\jq^2)\JK'-2\JE'}{4\pi\mi(1-\jq^2)\NU}\,
    (\lambda-\lambda^{-1})\,.
\end{equation}
Now setting
\begin{equation*}
    \cc_1 \coloneq -\frac{\jq\,(x_1 p_{20} - x_2 p_{10})(\JE'^2-\jq^2\JK'^2)}
    {4\pi(\jq^2 - 1)}
\end{equation*}
gives $\dot{\jq} = \jq(\JE'\jk-\jq\JK'\jh)$ as required.

The deformation equations of the \sym points \eqref{eq:sym-point-deformation} read
\begin{equation*} \begin{split}
    \dot{\lambda}_j\,(x_1\,\NU' + p_{10}\,\OMEGA') &= -(\dot{x}_1\,\NU + x_1\,\dot\NU + p_{10}\,\dot\OMEGA ) \\
    \dot{\lambda}_j\,(x_2\,\NU' + p_{20}\,\OMEGA') &= -(\dot{x}_2\,
    \NU + x_2\,\dot\NU + p_{20}\,\dot\OMEGA )
    \end{split}
\end{equation*}
Multiplying the first equation by $p_{20}$ and the second equation by $p_{10}$ and subtracting gives
\begin{equation*}
    \dot{\lambda}_j =  -\left.\frac{(\dot{x}_1\,p_{20} - \dot{x}_2\,p_{10})\,\NU + (x_1\,p_{20} - x_2\,p_{10})\,\dot{\NU}}{(x_1\,p_{20} - x_2\,p_{10})\,\NU'}
                        \right|_{\lambda = \lambda_j}\,.
\end{equation*}
Using the above formulae as well as \eqref{eq:sqrt_q_formula} we obtain
\begin{equation*} \begin{split}
    \frac{\dot{\lambda}_1}{\lambda_1} + \frac{\dot{\lambda}_2}{\lambda_2}&=
    \frac{2\sqrt{\jk^2-1}}{\jh^2 - \jk^2} \left(
    \frac{\jq\jk(\JE'-\JK')+\jh(\JE'-\jq^2\JK')}{1-\jq^2}(2\jq\jk-\jh(\jq^2+1)) + (\jk\JE'-\jh\jq\JE')(\jk-\jh\jq) \right) \\
    &= \frac{2\sqrt{(\jk^2-1)}}{\jq^2 - 1} \left( (1+\jq^2)\JE' - 2\jq^2\JK' \right)\,, \\
    \frac{\dot{\lambda}_1}{\lambda_1} - \frac{\dot{\lambda}_2}{\lambda_2}&=
    \frac{2\sqrt{\jh^2-1}}{\jh^2 - \jk^2} \left(
    \frac{\jq\jk(\JE'-\JK')+\jh(\JE'-\jq^2\JK')}{1-\jq^2}((\jq^2+1)\jk - 2\jh\jq) + (\jk\JE'-\jh\jq\JE')(\jq\jk-\jh) \right) \\
    &= \frac{2\sqrt{\jh^2-1}}{\jq^2 - 1}\,\jq \left( 2\JE'-(1+ \jq^2) \JK' \right)
\end{split}
\end{equation*}
and putting these into \eqref{eq:qdot_hdot} gives the equations for $\dot\jk$ and $\dot\jh$, and concludes the proof of \eqref{eq:torus-flow}.

The elliptic integrals $\JK'$ and $\JE'$ are analytic on $\jq\in\Rstar$, and at $\jq=\pm1$ equal to $\frac{\pi}{2}$. Therefore the right hand sides of \eqref{eq:torus-flow} extend analytically to $\jq\in\Rstar$. Due to
Proposition \ref{th:alpha_less_beta}, for $\jk,\jh\in(-1,1)$ we have
\begin{align}\label{eq:monotonicity}
    \dot{\jk}>0&\mbox{ for }0<|\jq|<1\,,&
    \dot{\jh}>0&\mbox{ for }-1<\jq<0\,,&
    \dot{\jh}<0&\mbox{ for }0<\jq<1\,.
\end{align}
By the properties of $\JK'$ and $\JE'$ the
vector field $(\XzeroDot,\,\XoneDot,\,\XtwoDot)$ is analytic in
$\Xzero$ on $\Rstar$ and has simple zeros at $\Xzero=\pm 1$. Thus
the vector field is analytic on $\calD$. The zero set statement
follows from the fact that on $\{\Xzero\not=0\}$, the functions
$(1+\jq^2)\JE'-2\jq^2\JK'$ and $2\JE'-(1+\jq^2)\JK'$ have zeros only at
$\jq =\pm 1$, and $\JE'=\jq\JK'$ holds only at $\jq =1$. That $\jq =\pm 1$ is a simple root follows from the series expansions at $\jq =1$ (similarly at $\jq =-1$), given by
\begin{align*}
    \JK'(\jq) &= \pi \left( \tfrac{1}{2} -\tfrac{1}{4}(\jq - 1) + \tfrac{5}{32}(\jq-1)^2 - \tfrac{7}{64}(\jq-1)^3 + \mathrm{O}(\jq-1)^4 \right)\,, \\
    \JE'(\jq) &= \pi \left( \tfrac{1}{2} + \tfrac{1}{4}(\jq - 1) + \tfrac{1}{32}(\jq-1)^2 - \tfrac{1}{64}(\jq-1)^3 + \mathrm{O}(\jq-1)^4 \right)\,.
\end{align*}
\end{proof}

%
%
\begin{figure}[t]
\centering
\includegraphics[width=5.4cm]{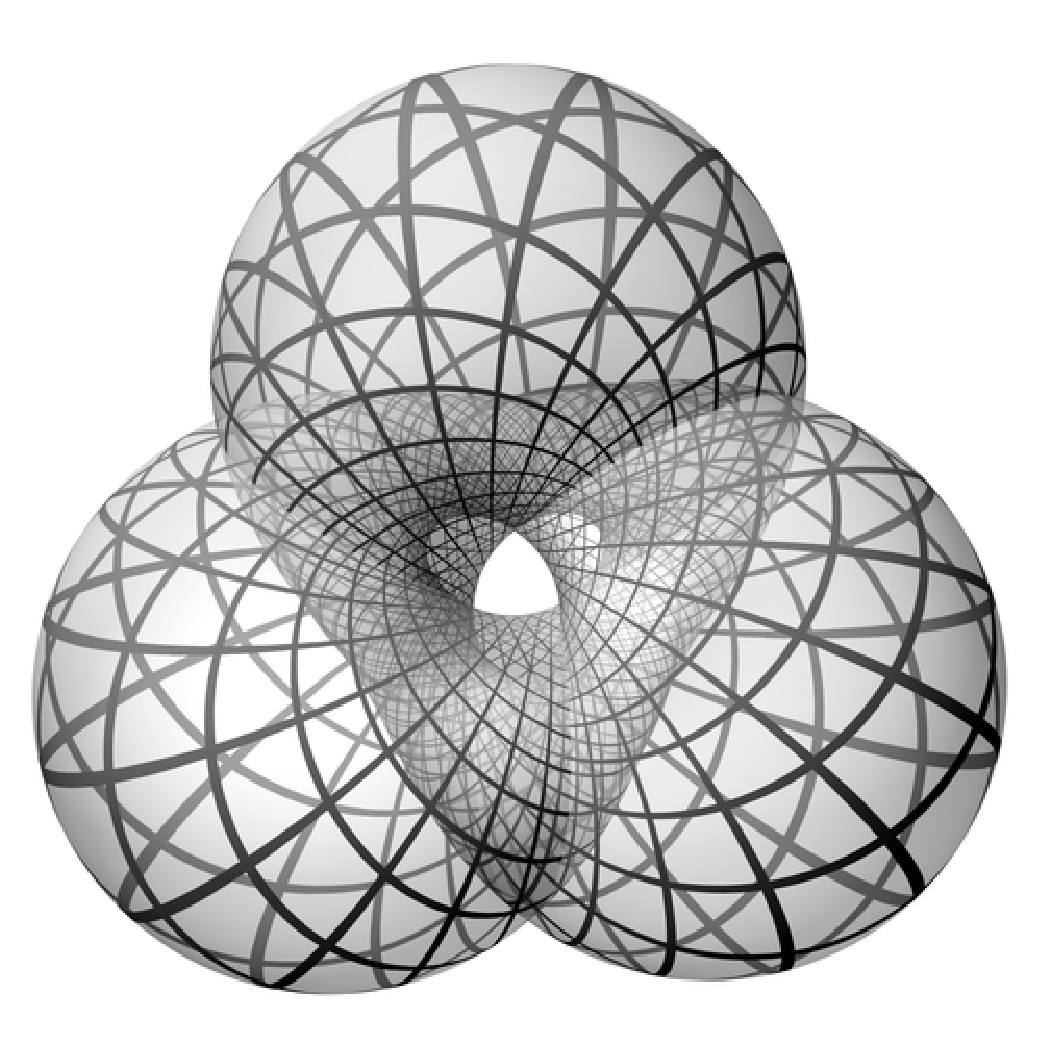}\hspace{0.0cm}
\includegraphics[width=5.4cm]{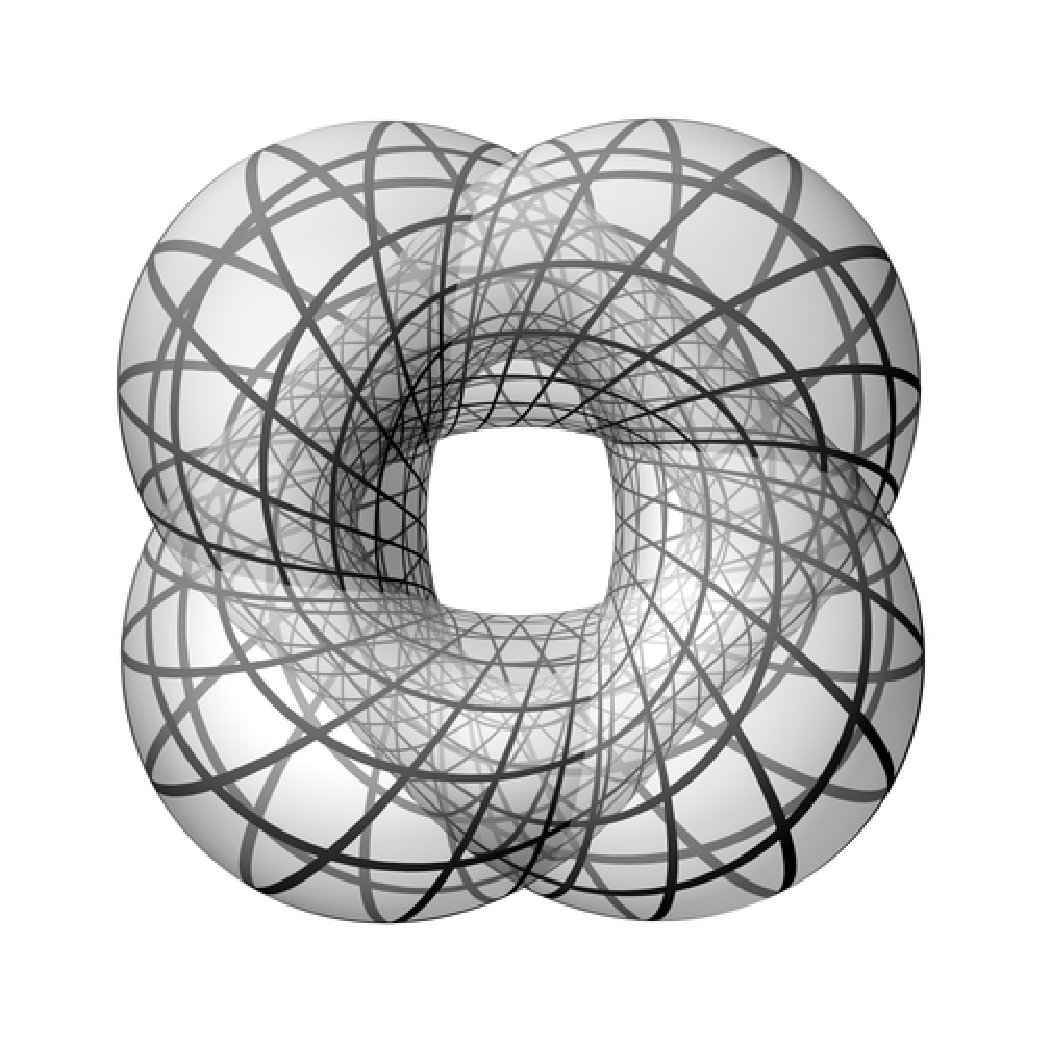}\hspace{0.0cm}
\includegraphics[width=5.4cm]{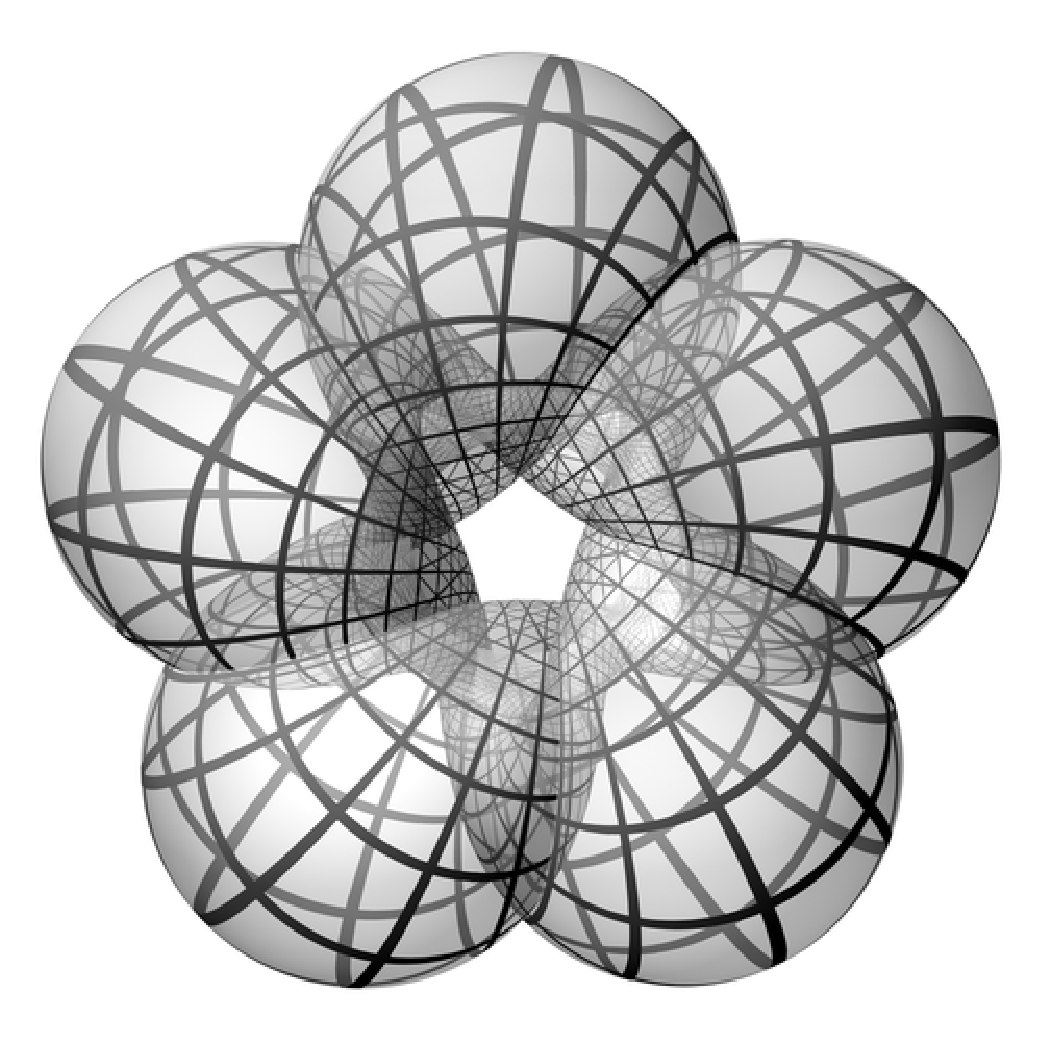}\hspace{0.0cm}
\caption{ \label{fig:twizzled2} Equivariant $(2,\,1,\,n)$ \cmc tori
($n=3,\,4,\,5$). By Proposition~\ref{th:minlobes}, there are no
twizzled tori with one or two major lobes. }
\end{figure}

%
%
\subsection{Global behaviour of the integral curves}

The proof of the following Proposition~\ref{prop:torus-of-revolution} is deferred to section \ref{sec:frame}.
\begin{proposition} \label{prop:torus-of-revolution}
An equivariant \cmc torus in $\mathbb{S}^3$ is rotational if and only if the \sym points are reciprocal.
\end{proposition}
Special values of $\Xone$ and $\Xtwo$ include
\begin{equation*}
\begin{aligned}
    \Xone^2=1 &\iff \lambda_1=\lambda_2^{-1} (rotational)\spacecomma&
    \Xtwo^2=1 &\iff \lambda_1 = \lambda_2\ (H=\infty)\spacecomma\\
    \Xone=0 &\iff \lambda_1 = -\lambda_2^{-1}\spacecomma&
    \Xtwo=0 &\iff \lambda_1 = -\lambda_2\ (H=0)\spacecomma\\
    \Xone=\Xtwo &\iff \lambda_1=1\text{ or }\lambda_2 = 1\spacecomma&
    \Xone=-\Xtwo &\iff \lambda_1=-1\text{ or }\lambda_2 = -1 \spaceperiod
\end{aligned}
\end{equation*}
By the deformation equations \eqref{eq:torus-flow} if $\jk^2(t_0) =1$, for some $t_0$, then $\jk^2 \equiv 1$ throughout the flow. Hence rotational tori stay rotational during the flow. We begin the qualitative analysis of the flow by first considering equivariant \cmc tori which are not rotational ($\Xone^2<1$): we call these \emph{twizzled} tori.
The tori of revolution ($\Xone^2=1$) are treated subsequently in Proposition~\ref{prop:moduli-tori-of-revolution}.

The flow will be investigated in the open solid cuboid
\begin{equation*}
    \calB \coloneq
    \{(\jq,\jk,\jh)\in(-1,1)^3\suchthat\Xzero\not=0\}\spaceperiod
\end{equation*}
Due to \eqref{eq:monotonicity} $\jk-\sign(\jq)\jh$ is strictly monotonic on $\calB$ with the locally constant function $\sign(\jq)$.
%
%
\begin{proposition} \label{prop:levelset}
Define the set
\begin{equation*} 
    \calL_c \coloneq \{
    (\Xzero,\,\Xone,\,\Xtwo)\in\calB\suchthat\quad
    (1-\Xone^2)(1-\Xtwo^2) =c(\tfrac{1+\jq^2}{2\jq}-\Xone\Xtwo)^2\}
    \spacecomma\interspace c\in\bbR_+ \spaceperiod
\end{equation*}
\begin{enumerate}
\item\label{item:levelset1} Then every integral curve in
$\calB$ lies in $\calL_c\cap\calB$ for some $c\in(0,\,1)$.

\item\label{item:levelset2} The following uniform estimates hold
  on the integral curve through $(\jq_0,\jk_0,\jh_0)\in\calL_c$:
\begin{equation*}
|\jq| \geq \frac{\sqrt{c}}{2+2\sqrt{c}} \quad \mbox{ and } \quad
\min\{1-\jk^2,1-\jh^2\} \geq c\,(\min\{1-|\jk_0|,1-|\jh_0|\})^2
\end{equation*}
\item \label{prop:dq-sign} The function $\XzeroDot$ has at most one zero along any integral curve in $\calB$.
\end{enumerate}
\end{proposition}
\begin{proof}
(1) From~\eqref{eq:torus-flow} we compute $\calW_t[
(1-\Xone^2)(1-\Xtwo^2),\,
    (\frac{1+\jq^2}{2\jq}-\Xone \Xtwo)^2 ] = 0$, where
    $\calW_t[X,\,Y] \coloneq X \dot Y- \dot X Y$ is the Wronskian
with respect to the flow parameter $t$.

Since $(1-\Xone^2)(1-\Xtwo^2)$ and $(\frac{1+\jq^2}{2\jq}-\Xone \Xtwo)^2$ are strictly positive in $\calB$, then every integral curve in $\calB$ lies in $\calL_c$ for some $c\in\bbR_+$. In $\calB$, we have
$(1-\Xone^2)(1-\Xtwo^2) \le {(1-\Xone\Xtwo)}^2 <
(\frac{1+\jq^2}{2\jq}- \Xone \Xtwo)^2$,
so $c\in(0,\,1)$.

(2) For $(\jq,\jk,\jh)\in\calL_c$ the first inequality of
    follows from
\begin{equation*}\frac{1}{|\jq|}\leq\frac{1+\jq^2}{|\jq|}
\leq 2\sqrt{\frac{(1-\jk^2)(1-\jh^2)}{c}}+2|\jk\jh|\leq
\frac{2+2\sqrt{c}}{\sqrt{c}}.
\end{equation*}
On each integral curve $\jq$ is either positive or negative.
The second inequality follows from
\begin{equation*}
\min\{1-\jk^2,1-\jh^2\}\geq(1-\jk^2)(1-\jh^2)\geq
c(1-\sign(\jq)\jk\jh)^2\geq c(\max\{1-|\jk|,1-|\jh|\})^2.
\end{equation*}
In fact, due to \eqref{eq:monotonicity} either
$\sign(\jq)\jk\jh\leq0$ or one of the functions
$1-|\jk|$ and $1-|\jh|$ is increasing and the other one
decreasing. Furthermore, if at some point $t$ with
$\sign(\jq)\jk\jh>0$ one of these functions is increasing, it
stays increasing for all $t_0\leq t$, and if it is decreasing, it
stays decreasing for all $t_0\geq t$, since the derivative of these
functions can change sign only at the maximal value $1$.

(3) If $\XzeroDot(t_0)=0$ for some value of
the flow parameter $t=t_0$, then due to \eqref{eq:torus-flow} and
\eqref{eq:monotonicity}
\begin{equation*}
\sign(\Xzero)\XzeroDdot=
\sign(\Xzero)
\tfrac{d}{dt}\Xzero(\JE'\Xone-\Xzero\JK'\Xtwo)=
|\Xzero|(\JE'\XoneDot-\Xzero\JK'\XtwoDot))>0
\end{equation*}
for $t=t_0$. Thus $\sign(\Xzero)\XzeroDot$ is increasing at
each of its zeros, and hence can have at most one zero.
\end{proof}



The next result shows that one endpoint of each integral curve
corresponds to a flat \cmc torus.
\begin{proposition} \label{prop:moduli-space}
Every maximal integral curve of \eqref{eq:torus-flow} in $\calB$ is
defined on a finite interval and passes from a point in
$\{\Xzero=\pm1\}\cap\{\Xone-\Xzero\Xtwo<0\}\cap\boundary\calB$
to a point in
$\{\Xzero=\pm1\}\cap\{\Xone-\Xzero\Xtwo>0\}\cap\boundary\calB$ with the
same sign of $\Xzero$. Furthermore, the mean curvature
$t \mapsto H(t)$ is a diffeomorphism of the maximal interval of
definition onto a finite interval.
\end{proposition}
\begin{proof}
Due to \eqref{eq:monotonicity} both functions $\Xone$ and $\Xtwo$
are strictly monotonic. Since $\XoneDot$ and $\XtwoDot$ have no
roots in $\calB$, the integral curve must hit the boundary of
$\calB$ at the end points (either infinite or finite) of the maximal
interval of definition. Due to
Proposition~\ref{prop:levelset}~\eqref{item:levelset2}, $\Xzero$
takes the same value $\pm 1$ at both endpoints and the vector
field~\eqref{eq:torus-flow} does not vanish there. Hence the maximal
interval of definition is finite. Due to
Proposition~\ref{prop:levelset} (3) the sign of $\XzeroDot$ changes once
and is proportional to $\Xone-\sign(\Xzero)\Xtwo$ at the end points.
Due to \eqref{eq:monotonicity} $\Xone-\sign(\Xzero)\Xtwo$ is
strictly increasing and the first claim follows. Since $\Xtwo\mapsto\Xtwo/\sqrt{1-\Xtwo^2}$ is
strictly increasing on $(-1,\,1)$, the mean curvature \eqref{eq:H_and_h} is strictly monotonic, and by
Proposition~\ref{prop:levelset}~\eqref{item:levelset2} it is bounded.
\end{proof}

\subsection{Global behaviour of the integral curves on the boundary}

The boundary $\boundary\calB$ consists of the three parts with
$\jq=\pm1$ or $\jq=0$, $\jk=\pm 1$ and $\jh=\pm 1$. The first set
contains the flat \cmc tori and the bouquets of spheres, which we consider later. The third set corresponds to the infinite mean curvature limit, that is \cmc tori in $\mathbb{R}^3$. In this limit the two \sym points coalesce and the differentials $d\ln\mu_i$ have zeroes there. Since $d\nu$ and $d\omega$ do not have common zeroes by Corollary \ref{th:no_common_zeroes}, there are no such examples with spectral genus zero or one. Therefore we treat only $\jk^2=1$ and $\jh\in(-1,1)$.

By Proposition~\ref{prop:torus-of-revolution} the tori of revolution
appear in the two-dimensional boundary $\{\Xone^2=1\}$ of the moduli
space of equivariant \cmc surfaces in $\mathbb{S}^3$. Since
$\dot\Xone \equiv 0$ along $\Xone^2 =1$, tori of revolution stay
tori of revolution throughout the flow. The
flow~\eqref{eq:torus-flow} thus describes one parameter
families of tori of revolution. Later in Theorem~\ref{thm:rev-mean-curvature} we describe the range
of mean curvature for the families in the flow, and exhibit those
that contain a minimal torus of revolution. A consequence of the
next result is that spectral genus $g=1$ tori of revolution lie in 1-parameter families with one endpoint at a flat \cmc torus and the other endpoint at a sphere bouquet. In the process we also show that the mean curvature stays bounded during the flow.

Since we often need to evaluate the functions $\NU$ and $\OMEGA$ at the two \sym points we set
\begin{equation} \label{eq:nu-omega-k}
    \NU_k = \NU(\lambda_k)  \mbox{ and }
    \OMEGA_k = \OMEGA(\lambda_k) \mbox{ for } k=1,\,2\,.
\end{equation}
%

\begin{proposition}
\label{prop:moduli-tori-of-revolution}
On the integral curve of \eqref{eq:torus-flow} through
$(\jq_0,\jk_0,\jh_0)\in\boundary\calB$ with $\jk_0=\pm1$ and
$\jq_0,\jh_0\in(-1,1)$ the function $\jk$ is equal to $\pm1$ and
$1-\jh^2$ is bounded away from zero. The maximal interval of
definition is of the form $(-\infty,\,t_{\max})$ with
$$
    \lim_{t\downarrow -\infty}\jq=0 \quad \mbox{ and } \quad
    \lim_{t\uparrow t_{\max}}\jq=\pm1 \spaceperiod
$$
The mean curvature $t\mapsto H(t)$ is a diffeomorphism from
$(-\infty,\,t_{\max})$ onto a finite interval.
\end{proposition}
\begin{proof} By \eqref{eq:torus-flow} $\Xone \equiv \Xone_0$ is
constant throughout the flow. Let $\lambda_i = e^{2\mi\theta_i}$ for $i=1,\,2$ be the two \sym points. If $|\theta_2-\theta_1|$ is bounded away form zero, then $1-\Xtwo^2$ is bounded away from zero.
For $\jk=\pm1$ the values of $\NU$ at the \sym points
coincide.

By \eqref{eq:sym-point-deformation} we have that $\omega_2 - \omega_1 = \int_{\theta_1}^{\theta_2}d\omega$ is
constant throughout the flow. We claim that
\begin{equation*}
|d\omega | \leq 2|d\theta |\quad 
\mbox{ for }\theta\in\mathbb{R}\mbox{ and }\jq\in[-1,1].
\end{equation*}
We will need to consider the two cases $\jq\cos(2\theta)\leq 0$
and $\jq\cos(2\theta)\geq 0$ separately. In the first case we use
$\JE'\leq\frac{\pi}{2}$ \cite[13.8.(11)]{Bat2}, Proposition \ref{th:alpha_less_beta} and \eqref{eq:nu} to obtain
$$
    |d\omega| = \left|
    \frac{\JE'-\jq\JK'\cos(2\theta)}{\pi\NU}\right|\,|d\theta |
    \leq\left|
    \frac{\JE'(1-\tfrac{2\jq}{1+\jq^2}cos(2\theta))}
                       {\pi\NU}\right|\,|d\theta|
    \leq\left|\frac{2\NU}{1+\jq^2}\right|\,|d\theta|
    \leq2|d\theta| \spaceperiod
$$
When $\jq\cos(2\theta)\leq 0$, then again by $\JE'\leq\frac{\pi}{2}$,
Proposition \ref{th:alpha_less_beta}, and
$2|\NU | \geq \sqrt{1+\jq^2} \geq 1$ \eqref{eq:nu} we get
$$
    |d\omega| = \left|
    \frac{\JE'-\jq\JK'\cos(2\theta)}{\pi\NU}\right|\,|d\theta |
    \leq\left|
    \frac{4\JE'}{\pi\sqrt{1+\jq^2}}\right|\,|d\theta|
    \leq2|d\theta| \spaceperiod
$$
Thus $|\theta_2-\theta_1|\geq\half|\omega_2-\omega_1|$.
Due to \eqref{eq:omega} and Proposition \ref{th:alpha_less_beta}, the differential $d\OMEGA$ has no root between $\theta_1$ and $\theta_2$. Therefore $|\theta_2-\theta_1|$ and $1-\jh^2$ are bounded away from zero.

Due to Proposition \ref{th:alpha_less_beta}, the function $\jq^2$ is strictly increasing for $0<|\jq|<1$. Therefore the maximal integral curve passes from $\jq=0$ to $\jq=\pm1$. The vector field is at the left end point of order $\Order(\jq)$ and at the right end point not zero. Therefore the maximal interval of definition has the form $(-\infty,\,t_{\max})$.

Since $\jh$ is strictly monotonic, and $1-\jh^2$ is bounded away from
zero, the mean curvature \eqref{eq:H_and_h} is a diffeomorphism from $(-\infty,\,t_{\max})$ onto a finite interval.
\end{proof}

\subsection{Regularity of spectral curve} A corollary of
Theorem~\ref{th:no_equi_bub} below is that equivariant \cmc tori
cannot be dressed to tori by simple factors \cite{KilSS,TerU}.
While there exist large families of \cmc cylinders which can be
dressed to cylinders with bubbletons \cite{SteW:bub}, it is still an open question raised by Bobenko \cite{Bob:tor}, whether there are \cmc tori with bubbletons. A \emph{double point} is a point on the spectral curve at which both logarithms $\mu_1$ and $\mu_2$ are
unimodular, but the spectral curve is not branched.
\begin{theorem}
\label{th:no_equi_bub} The spectral curve of an equivariant
{\sc{cmc}} torus has no double points in $\Cstar\setminus\bbS^1$.
\end{theorem}
\begin{proof}
At a double point both logarithms $\mu_1$ and $\mu_2$ are
unimodular. Therefore it suffices to prove that the set of all
$\lambda\in\Cstar$, where $\mu_1$ and $\mu_2$ are unimodular is the set $\bbS^1\cup\{\jq,\,\jq^{-1}\}$. This set coincides with the
subset of $(\lambda,\,\nu)\in\curve$, such that $\nu$ and $\omega$ are real. Due to \eqref{eq:nu} and \eqref{eq:omega} we have for
$\jq\in(0,1]$: $\nu\in\bbR$ if and only if $\lambda\in\bbS^1\cup[\jq,\,\jq^{-1}]\cup\bbR_-$,
and $\lambda\in\bbR$ and $\omega\in\bbR$ if and only if $\lambda\in[0,\,\jq]\cup[\jq^{-1},\,\infty)$.
For $\jq\in[-1,\,0)$: $\nu\in\bbR$ if and only if
$\lambda\in\bbS^1\cup[\jq^{-1},\,\jq]\cup\bbR_+$; $\lambda\in\bbR$ and $\omega\in\bbR$ if and only if $\lambda\in(-\infty,\,\jq^{-1}]\cup[\jq,\,0]$.
With $\jq=\pm1$ we cover the genus zero case.
\end{proof}
%


\section{Equivariant extended frames}
\label{sec:frame}

To obtain a more geometric description of the above deformation we compute the frames of the surfaces. We do this in a slightly broader context by studying equivariant maps  $\bbC^2 \to \matSL{2}{\bbC}$ with complex mean curvature as in~\cite{DorKP}. This generalized
framework for complex equivariant surfaces is then applicable to not
only \cmc immersions into 3-dimensional space forms , but also
pseudospherical surfaces in $\mathbb{R}^3$, \cmc surfaces in
Minkowski space, and other integrable surfaces in many others spaces by
performing an appropriate reduction.

Let $\Sigma\subset\bbC^2$ be an open and simply-connected domain with complex
coordinates $(z,\,w)$, and define $\ast$ on $1$-forms on $\Sigma$ by $\ast dz = \mi dz$, $\ast dw
= -\mi dw$ extended complex linearly.

Let $\matsl{2}{\bbC}=\Diagonal\oplus\OffDiagonal$ be a Cartan
decomposition. For a $1$-form $\alpha$ on $\Sigma$, we will write
$\alpha = \alpha_\Diagonal + \alpha_\OffDiagonal$ where with
$\alpha_\Diagonal\in\Diagonal$ and
$\alpha_\OffDiagonal\in\OffDiagonal$.
If we equip the matrix Lie algebra $\matsl{2}{\bbC}$ with a multiple of the Killing form $\DOTC{X}{Y} \coloneq -\half\trace(XY)$, then we may choose a basis $e_0,\,e_1,\,e_2$ for $\matsl{2}{\bbC}$ satisfying
$e_0\in\Diagonal$, $\DOTC{e_0}{e_0}=1$ and $[e_0,\,e_1] = e_2$, $[e_1,\,e_2] = e_0$ and $[e_2,\,e_0] = e_1$,
and set
$$
    \epsilon_1 = \half(e_1-\mi e_2) \spacecomma\interspace
    \epsilon_2 = \half(e_1+\mi e_2)\,.
$$
Let $f:\Sigma \to \matSL{2}{\bbC}$ be a conformal immersion such that for the smooth function $v:\Sigma\to\bbC^\ast$ we have $\DOTC{f_z}{f_z} = 0 = \DOTC{f_w}{f_w}$ and $2\DOTC{f_z}{f_w} = v^2$. The remaining invariants of $f$ are smooth functions $H,\,Q,\,R:\Sigma\to\bbC$ such that $2\DOTC{f_{zz}}{N} = Q$, $2\DOTC{f_{zw}}{N} = v^2H$ and $2\DOTC{f_{ww}}{N} = R$, and the normal map is $N = -\mi v^{-2}[f_z,\,f_w]$. Further, there exists a unique pair of maps $F,\,G:\Sigma\to\matSL{2}{\bbC}$ which frame $f$ in the sense that
\begin{equation} \label{eq:framing}
    f = FG^{-1}\spacecomma\interspace f_z =
    vF\epsilon_1G^{-1}\spacecomma\interspace f_w =
    vF\epsilon_2G^{-1}\spacecomma\interspace N = Fe_0G^{-1}\spaceperiod
\end{equation}
Then $\alpha_{+1}  \coloneq  F^{-1}dF$ and $\alpha_{-1} \coloneq
G^{-1}dG$, are smooth $1$-forms on $\Sigma$ with values in
$\matsl{2}{\bbC}$ given by
\begin{equation} \label{eq:maurer-cartan-general}
\begin{aligned}
    \alpha_{\sigma} &= A_{\sigma} dz +
    B_{\sigma} dw \quad \mbox{where } \sigma\in\{\pm 1\} \quad \mbox{and}\\
    2\mi A_{\sigma} &\coloneq  v^{-1}v_z e_0 + v (H+\mi\sigma) \ep_1 -
    v^{-1}Q\ep_2
    \spacecomma\\
    2\mi B_{\sigma} &\coloneq -v^{-1}v_w e_0 + v^{-1}R \ep_1 -
    v(H-\mi\sigma)\ep_2 \spaceperiod
\end{aligned}
\end{equation}
Integrability $2 d\alpha_\sigma+[\alpha_\sigma\wedge\alpha_\sigma]=0$ is equivalent to the Gauss-Codazzi equations
\begin{subequations} \label{eq:gauss-codazzi}
\begin{gather}
\label{eq:gauss}
    {(\log v)}_{zw} - v^2(H^2+\sigma^{-2}) + v^{-2}QR=0\spacecomma\\
    \label{eq:codazzi1} v^2 H_w = R_z\spacecomma\interspace v^2
    H_z=Q_w\spaceperiod
\end{gather}
\end{subequations}
Conversely, given smooth functions $v,\,H,\,Q,\,R:\Sigma\to\bbC$
with $v$ non-vanishing, satisfying the integrability
equations~\eqref{eq:gauss-codazzi}, a conformal immersion $f$
with these invariants can be recovered by integrating \eqref{eq:maurer-cartan-general} to obtain maps
$F,\,G:\Sigma\to\matSL{2}{\bbC}$ which are unique up to transforms
$(F,\,G)\mapsto(AF,\,BG)$, $A,\,B\in\matSL{2}{\bbC}$. By the form of
$\alpha_\sigma$, the maps $F$ and $G$ frame $f \coloneq FG^{-1}$ in
the sense of~\eqref{eq:framing}, and $f$ has the specified
invariants. The map $f$ is unique up to $f\mapsto
AFB^{-1}$, $A,\,B\in\matSL{2}{\bbC}$.
%
%
\subsection{Equivariance}
Let us first introduce new coordinates $(x,\,y)$ on $\Sigma$ such that
\begin{equation} \label{eq:xyzw}
    z=x + \mi\,y\, ,\qquad w = x -\mi\,y\spaceperiod
\end{equation}
We say a smooth map $F=F(x,\,y):\bbC^2\to\matSL{2}{\bbC}$ is
$\bbC$-equivariant, or simply equivariant if it is of the form
$F(x,\,y) = \exp(x\,A)\,P(y)$ for some $A\in\matsl{2}{\bbC}$ and smooth map $P:\bbC\to\matSL{2}{\bbC}$.
What characterizes an equivariant map is that
$F^{-1}dF = P^{-1}A\,P\,dx + P^{-1}P'\,dy$ is $x$-independent~\cite{BurK}.
We first compute the logarithmic derivative of an equivariant map, and then integrate this to obtain the equivariant map itself.
\begin{theorem} \label{thm:equivariant-frame}
{\rm{(1)}} The $x$-independent solution
$\alpha = \alpha_\mathfrak{k} + \alpha_\mathfrak{p}$ to
\begin{equation} \label{eq:harmonic-map-equations}
\begin{split}
    d\alpha_\mathfrak{k}&+
        \half[\alpha_\mathfrak{p}\wedge\alpha_\mathfrak{p}]
        = 0 \spacecomma\\
    d\alpha_\mathfrak{p}&+
        [\alpha_\mathfrak{k}\wedge\alpha_\mathfrak{p}]
        = 0 \spacecomma\\
    d(\ast\alpha_\mathfrak{p})&+
        [\alpha_\mathfrak{k}\wedge(\ast\alpha_\mathfrak{p})]
        = 0
\end{split}
\end{equation}
are $\alpha = g^{-1}\Omega\,g + g^{-1}dg$, where
$g$ is a smooth $x$-independent map from $\Sigma$ into the Lie group
of $\Diagonal$, and $\Omega = \Omega_1 \,dz + \Omega_2 \,dw$, where
\begin{equation} \label{eq:equivariant-harmonic-maurer-cartan}
\begin{aligned}
    2\mi\Omega_1 &\coloneq -\tfrac{\mi}{2} \Metric^{-1}\Metric' e_0
    + a_2 \Metric \epsilon_1 - b_1 \Metric^{-1} \epsilon_2 \spacecomma\\
    2\mi\Omega_2 &\coloneq -\tfrac{\mi}{2} \Metric^{-1}\Metric' e_0
    + b_2 \Metric^{-1}\epsilon_1 - a_1 \Metric\epsilon_2 \spacecomma
\end{aligned}
\end{equation}
with $a_1,\,a_2,\,b_1,\,b_2,\,\Vconst\in\bbC$, and
$\Metric=\Metric(y):\bbC\to\bbC$ is defined by
\begin{equation} \label{eq:metric}
\begin{aligned}
    &{(\Metric^{-1}\Metric')}^2 + (a_1 \Metric + b_1 \Metric^{-1})(a_2
    \Metric + b_2 \Metric^{-1}) = 4\Vconst^2 \spacecomma\\
    &{(\Metric^{-1}\Metric')}' + a_1 a_2 \Metric^2 - b_1 b_2
    \Metric^{-2} = 0 \spaceperiod
\end{aligned}
\end{equation}
{\rm{(2)}} A solution to $dF=F\Omega$ is given by
\begin{equation}
\begin{gathered}
\label{eq:frame-parts}
\FEQ(x,\,y)  \coloneq
 \exp\left((x\,\Vconst + \half\,\AngleZero) e_0\right)
 \exp\left( \half \AngleOne e_1 \right)
 \exp\left( \half \AngleTwo e_0 \right) \spacecomma \\ \mbox{ {\rm{with}} }
\interspace \AngleZero  \coloneq 2\mi\Vconst \int_0^y
(J_1(t)-J_2(t))dt\spacecomma \\ \AngleOne  \coloneq
\arccos(-\half\Vconst^{-1}\Metric^{-1}\Metric') \spacecomma \interspace
\AngleTwo  \coloneq  \tfrac{\mi}{2} \log (X_1^{-1}X_2) \spacecomma \\
X_1  \coloneq  a_1 \Metric + b_1 \Metric^{-1} \spacecomma \interspace
X_2  \coloneq  a_2 \Metric + b_2 \Metric^{-1} \spacecomma
J_1 \coloneq b_1\Metric^{-1}X_1^{-1} \spacecomma \interspace
J_2 \coloneq b_2\Metric^{-1}X_2^{-1} \spaceperiod
\end{gathered}
\end{equation}
\end{theorem}
\begin{remark} The second equation in~\eqref{eq:metric} is the derivative of the first, and eliminates spurious constant solutions which would appear if only the first equation were present.
\end{remark}
\begin{proof}
(1) Write $\alpha_\Diagonal = ae_0\, dx + be_0\, dy$ for some
functions $a=a(y)$ and $b=b(y)$, and let $g = \exp(-(\int^y b(t)dt)
e_0)$. Then $dg\, g^{-1} = -b e_0 dy$, so the form $g^{-1}\alpha g
+g^{-1}dg = a e_0 dx + g^{-1}\alpha_\OffDiagonal g$ is a multiple of
$dx = \tfrac{1}{2}(dz+dw)$. Then $\alpha_\Diagonal =
\tfrac{\mi}{4}f(dz+dw)e_0$ for some function $f=f(y)$. Let $v =
\exp(\int^y f(t)dt)$, so $\alpha_\Diagonal = \tfrac{\mi}{4}v^{-1}v'
e_0(dz+dw)$, where prime is differentiation with respect to $y$.
Then there exist functions $a_1,\,b_1,\,a_2,\,b_2$ of $y$ such that
$2\alpha_\OffDiagonal =
 (a_2 v \epsilon_1 - b_1 v^{-1}\epsilon_2)\,dz
+ (b_2 v \epsilon_1 - a_1 v^{-1}\epsilon_2)\,dw$.
By a calculation, ~\eqref{eq:harmonic-map-equations} is equivalent
to the second equation of~\eqref{eq:metric}, along with
$a_1'=b_1'=a_2'=b_2'=0$. Hence $\alpha$ is of the required form, and
concludes the proof of (1).

To prove (2), we have $\Omega = \Omega_x dx + \Omega_y dy$
in~\eqref{eq:equivariant-harmonic-maurer-cartan}, where $2\mi\Omega_x
\coloneq -\mi\Metric^{-1}\Metric' e_0 +(a_2 \Metric+b_2
\Metric^{-1})\ep_1 -(a_1 \Metric + b_1 \Metric^{-1})\ep_2$ and
$2\Omega_y \coloneq (a_2 \Metric - b_2 \Metric^{-1})\ep_1 + (a_1
\Metric-b_1 \Metric^{-1})\ep_2$. To compute $F^{-1}dF$, we will
first show that $P(x) \coloneq \exp\left( \half \AngleZero e_0
\right) \exp\left( \half \AngleOne e_1 \right) \exp\left( \half
\AngleTwo e_0 \right)$ satisfies
\begin{subequations}
\label{eq:finalframe-P}
\begin{align}
\label{eq:finalframe-P1a-X}
    \Vconst P^{-1}e_0P &= \Omega_x
\spacecomma
\\
\label{eq:finalframe-P2a-X}
    P^{-1}P' &= \Omega_y \spaceperiod
\end{align}
\end{subequations}
We have $P^{-1}e_0P =
  \cos\AngleOne e_0 + \mi e^{-\mi\AngleTwo} \sin\AngleOne\ep_1
  - \mi e^{\mi\AngleTwo}\sin\AngleOne \ep_2$ and
$\cos\AngleOne = -\half\Vconst^{-1}\Metric^{-1}\Metric'$,
$\sin\AngleOne = -\half\Vconst^{-1}X_1^{\frac{1}{2}}X_2^{\frac{1}{2}}$,
$e^{\mi\AngleTwo} = X_1^{\frac{1}{2}}X_2^{-\frac{1}{2}}$ and thus
$2\mi\Vconst P^{-1}e_0P  =\mi\Metric^{-1}\Metric'e_0 - X_2 \ep_1 +
X_1 \ep_2 = -2(\Omega_1-\Omega_2) = 2\mi\Omega_x$,
proving~\eqref{eq:finalframe-P1a-X}.

Now $2P^{-1} P' = (\AngleZero'\cos\AngleOne + \AngleTwo')e_0
  +e^{-\mi\AngleTwo}(\mi\AngleZero'\sin\AngleOne + \AngleOne')\ep_1
  -e^{\mi\AngleTwo}(\mi\AngleZero'\sin\AngleOne - \AngleOne')\ep_2$.
With $\cc_5 \coloneq a_1 b_2 - a_2 b_1$ we have
\begin{equation} \label{eq:angle-deriv}
    \AngleZero' = -\frac{2\mi\Vconst \cc_5}{X_1X_2}
    \spacecomma\quad \AngleOne' = \frac{a_1 a_2 v^2 - b_1 b_2
    v^{-2}}{\sqrt{X_1}\sqrt{X_2}} \spacecomma\quad \AngleTwo' =
    -\frac{\mi \cc_5 v'}{v X_1 X_2} \spaceperiod
\end{equation}
Hence $2P^{-1} P' = (a_2 \Metric - b_2 \Metric^{-1})\ep_1 + (a_1
\Metric - b_1 \Metric^{-1})\ep_2 = 2\Omega_y$,
proving~\eqref{eq:finalframe-P2a-X}.

Thus $F^{-1}dF = F^{-1} F_x dx +
F^{-1} F_y dy = \Vconst P^{-1}e_0 P dx + P^{-1} P' dy =
\Omega_x dx + \Omega_y dy = \Omega$.
\end{proof}
\begin{example} \label{sec:vacuum}
The vacuum is the case in which the function $\Metric$
in~\eqref{eq:metric} is constant. By the second equation
in~\eqref{eq:metric}, $\Metric \equiv v_0$ with $v_0^4 \coloneq (b_1
b_2)/(a_1 a_2)$. The potential for the vacuum is
$\Omega_{\mathrm{V}} \coloneq \Omega_z dz + \Omega_w dw$, where
$2\mi\Omega_z \coloneq a_2 v_0\epsilon_1 - b_1 v_0^{-1} \epsilon_2$
and $2\mi\Omega_w \coloneq b_2 v_0^{-1}\epsilon_1 - a_1
v_0\epsilon_2$. Since $[\Omega_z,\,\Omega_w]=0$, the extended frame
of the vacuum is $F = \exp(\Omega_z z + \Omega_w w)$ with
eigenvalues $\exp\bigl(\pm\tfrac{\mi}{2}(rz+sw)\bigr)$ where
$r,\,s\in\bbC$ are determined by $r^2=a_2b_1$, $s^2=a_1b_2$, $rs =
a_1a_2v_0^2$. This differs from the frame in
Theorem~\ref{thm:equivariant-frame} by left multiplication by a
$z$-and $w$-independent element of $\matSL{2}{\bbC}$.
\end{example}

%
\begin{figure}[t]
\centering
\includegraphics[width=8.15cm]{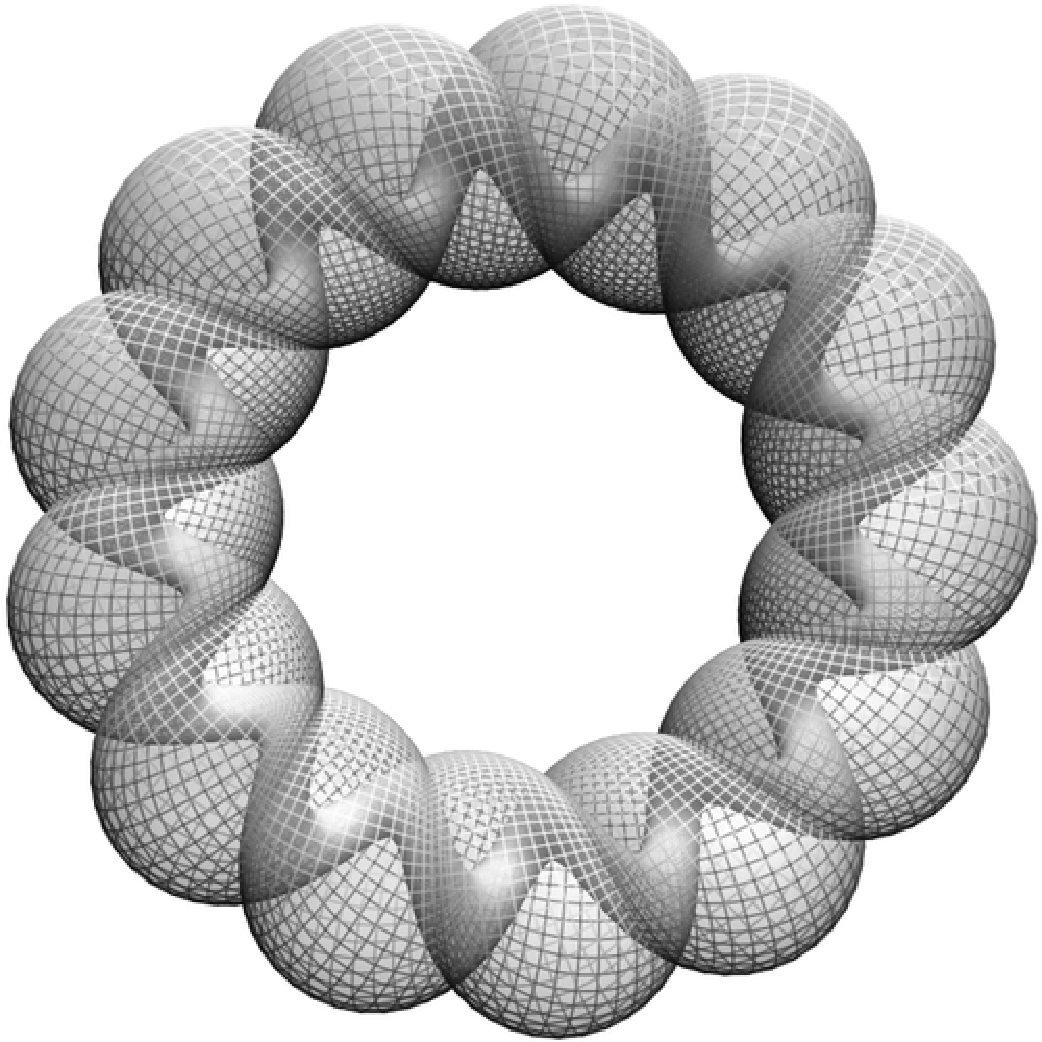}
\includegraphics[width=8.15cm]{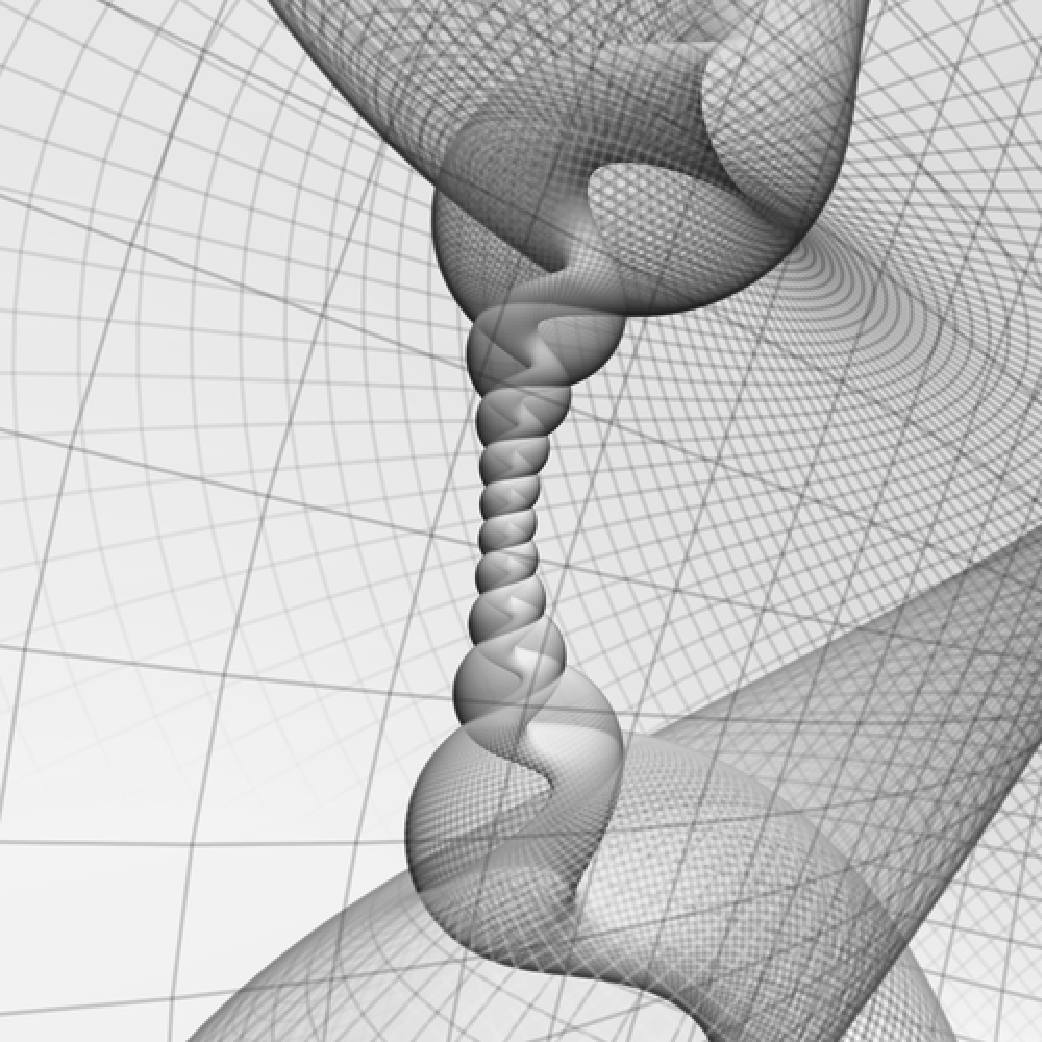}
\caption{ \label{fig:13lobe} Two views of an equivariant \cmc
$(2,\,1,\,13)$ torus in $\bbS^3$. }
\end{figure}
To describe equivariant \cmc immersions into $\bbS^3 = \matSU{2}{}$ we
specialize the above formulas.
We first make the reduction in \eqref{eq:xyzw} that $w = \bar{z}$,
which is equivalent to $x,\,y \in \bbR$. Given an extended frame
$F_\lambda :\bbR^2 \to \mathrm{SU}_2$ and two distinct \sym points
$\lambda_1,\,\lambda_2 \in \bbS^1$, then
\begin{equation} \label{eq:sym_s3}
    f \coloneq F_{\lambda_1}F^{-1}_{\lambda_2}
\end{equation}
is a conformal immersion $f:\bbR^2 \to \bbS^3$ with
constant mean curvature $H$ given in \eqref{eq:cmc-I-II}.
For the translation $\tau_\gamma: \bbC \to \bbC,\,p\mapsto p + \gamma$
we write $\tau_\gamma^* f = f \circ \tau_\gamma$.
If $F_\lambda^{-1}dF_\lambda$ is periodic with period $\gamma$,
then we define the monodromy $M_\lambda$ of $F_\lambda$ with
respect to $\gamma$ as
\begin{equation} \label{eq:monodromy}
    M_\lambda(\gamma) = \tau_\gamma^*F_\lambda \, F^{-1}_\lambda\,.
\end{equation}
The closing condition $\tau_\gamma^* f = f$ with respect to a
translation is equivalent to
\begin{equation} \label{eq:closing-s3}
    M_{\lambda_1}(\gamma)= M_{\lambda_2}(\gamma) = \id \OR
    M_{\lambda_1}(\gamma)= M_{\lambda_2}(\gamma) = -\id\spaceperiod
\end{equation}
If $\mu_1$ and $\mu_2 = \mu_1^{-1}$ denote the eigenvalues of the monodromy, then \eqref{eq:closing-s3} reads $\mu_j(\lambda_k) = \pm 1$, or equivalently that there exist four integers $p_{jk} \in \bbZ$ such that for $j,\,k=1,\,2$ we have
\begin{equation}\label{eq:pjk}
  \ln \mu_j(\lambda_k) = \mi\pi\,p_{jk} \quad \mbox{ and } \quad
   p_{j1}-p_{j2}\in 2\bbZ \spaceperiod
\end{equation}
The torus is embedded if and only if the \emph{winding numbers} $p_{jk}$ all have absolute value equal to one.


\subsection{Flat cmc tori in $\bbS^3$.}

Integrating $\Omega$~\eqref{eq:equivariant-harmonic-maurer-cartan} with
$(a_1,\,b_1,\,a_2,\,b_2)  \coloneq (\lambda,\,1,\,\lambda^{-1},\,1),\,
v \equiv 1$ yields an extended frame $F_\lambda$
of a flat \cmc surface in $\bbS^3$.
By Theorem~\ref{thm:equivariant-frame} and Example \ref{sec:vacuum},
the extended frame of any flat \cmc surface (up to isometry and
conformal change of coordinates) is
\begin{equation} \label{eq:flat-frame}
    F_\lambda = \exp \bigl(
        \pi\,\mi \left( (z \lambda^{-1} +\ol{z})\epsilon_1 - (z+\ol{z}\lambda) \epsilon_2 \right) \bigr)\spaceperiod
\end{equation}
Then also (up to isometry and conformal change of coordinates)
any flat \cmc immersion $f:\bbR^2 \to \bbS^3$ is of the form
$f = F_{\lambda_1}F_{\lambda_2}^{-1}$ with extended frame
$F_\lambda$ as in \eqref{eq:flat-frame} and two distinct \sym points
$\lambda_1,\,\lambda_2 \in \bbS^1$.
The eigenvalues of $F_\lambda$ are $\mu^{\pm 1}$ with
\begin{equation} \label{eq:eigenvalues}
    \mu(z,\,\lambda) = \exp( \pi\,\mi (z\lambda^{-\frac{1}{2}}+
    \ol{z}\lambda^{\frac{1}{2}}))\,.
\end{equation}
Define
\begin{equation}
    \DOTR{x}{y} \coloneq \tfrac{1}{2}\,(x\,\ol y + y\,\ol x ) = \Re(x\ol y) \spaceperiod
\end{equation}
Then by equation \eqref{eq:pjk} the immersion $f$
factors through the lattice $\Gamma = \gamma_1\,\bbZ + \gamma_2\,\bbZ$
if and only if
\begin{equation} \label{eq:closing-dot}
    \DOTR{\gamma_j}{\lambda_k^{1/2}}\in\bbZ \AND
    \DOTR{\gamma_j}{\lambda_1^{1/2} - \lambda_2^{1/2}}\in 2\bbZ\,.
\end{equation}
The dual of a lattice $\Gamma$ in $\bbC$ is the lattice
$\Gamma^\ast = \{\kappa\in\bbC \mid \DOTR{\kappa}{\gamma}\in\bbZ
\text{ for all }\gamma\in\Gamma\}$.
\begin{proposition} \label{prop:flat-torus}
\mbox{}

{\rm{(i)}} A flat \cmc immersion $f = F_{\lambda_1}F^{-1}_{\lambda_2}$
with extended frame $F_\lambda$~\eqref{eq:flat-frame} is closed with
respect to a lattice $\Gamma\subset\bbC$ if and only if
$\,\Gamma \subseteq \Lambda^\ast$, where
$\Lambda \coloneq \kappa_1\bbZ+\kappa_2\bbZ$ and
$$
    \kappa_1\coloneq \half(\sql{1}+\sql{2}) \AND \kappa_2\coloneq
\half(\sql{1}-\sql{2})\,.
$$
{\rm{(ii)}} The torus is rectangular and embedded
if and only if $\,\Gamma=\Lambda^\ast$.

{\rm{(iii)}} Every flat \cmc torus is isogenic to a rectangular
embedded flat \cmc torus.
\end{proposition}
\begin{proof} (i) Since $\lambda_k^{1/2}=\kappa_1\pm\kappa_2$
the condition \eqref{eq:closing-dot} is equivalent to
$\,\Gamma \subseteq \Lambda^\ast$.

(ii) The torus is embedded if and only if
$\DOTR{\gamma_j}{\lambda_k^{1/2}}\in \{ \pm 1\}$,
and rectangular if and only if $\gamma_1/\gamma_2 \in \mi\,\bbR$.
The corresponding periods $\gamma_1$ and $\gamma_2$ are dual to
$\kappa_1$ and $\kappa_2$ and generate $\Lambda^\ast$.
Clearly $\Lambda$, and consequently $\Lambda^\ast$ are rectangular,
since $\kappa_1/\kappa_2 \in \mi\,\bbR$.

(iii) From (i) we know that a lattice $\Gamma$ of any torus is a sublattice
of $\Lambda^\ast$, which by (ii) is the lattice of the embedded rectangular
torus. Hence there is an isogeny taking $\Lambda^\ast$ to $\Gamma$.
\end{proof}
\begin{proposition} \label{th:clifford}
The lattice of an embedded flat \cmc torus is square if and only if it the mean curvature is zero. Swapping the \sym points does not affect the period lattice of a flat \cmc torus.
\end{proposition}
\begin{proof}
Solving the four equations \eqref{eq:pjk} for the periods gives
\begin{equation} \label{eq:periods_general}
    \gamma_1 = \frac{\lambda_1^{1/2}\lambda_2\,p_{11} -
    \lambda_1\lambda_2^{1/2}\,p_{12}}
    {\lambda_2 - \lambda_1}\,,\quad
    \gamma_2 = \frac{\lambda_1^{1/2}\lambda_2\,p_{21} -
    \lambda_1\lambda_2^{1/2}\,p_{22}}
    {\lambda_2 - \lambda_1}\,.
\end{equation}
In particular, setting $p_{11} = p_{12} = p_{21} = -p_{22} = 1$, a computation shows that the generators of a lattice of an embedded flat \cmc torus satisfy
$$
\frac{|\gamma_1|}{|\gamma_2|} =  \left| \pm\sqrt{1+H^2} + H \right|\,.
$$
Thus $|\gamma_1|=|\gamma_2|$ if and only if the mean curvature is zero. Setting $\lambda_1 = \mi,\,\lambda_2 = -\mi$ we obtain the generators $\gamma_1 = 1/\sqrt{2}$ and $\gamma_2 = \mi /\sqrt{2}$ of the square lattice of the Clifford torus.

Swapping the \sym points $\lambda_1 \leftrightarrow \lambda_2$, results in the integers $p_{jk}$ swapping second indices: $p_{11} \leftrightarrow p_{12},\,p_{21} \leftrightarrow p_{22}$. This does not change the periods $\gamma_1,\,\gamma_2$ in \eqref{eq:periods_general}.
\end{proof}

%
%
\subsection{Spectral genus one}
Consider the $1$-form $\Omega = \Omega_1 \,dz + \Omega_2 \,d\bar{z}$, where $\Omega_j$ are given in equation~\eqref{eq:equivariant-harmonic-maurer-cartan} with $\NU$ as in \eqref{eq:nu} for some $0<|\Bpoint|\leq 1$ and $(a_1,\,b_1,\,a_2,\,b_2)\coloneq (\lambda,\,-\Bpoint,\,\lambda^{-1},\,-\olBpoint)$.
Integrating the $1$-form $\Omega$ with these choices then yields an $\bbR$-equivariant extended frame $F_\lambda(x,\,y) = \exp(x\,A_\lambda)\,P_\lambda(y)$ for smooth maps $\lambda \mapsto A_\lambda,\,\bbS^1 \to \matsu{2}{}$ and $(\lambda,\,y) \mapsto P_\lambda(y),\,\bbS^1 \times \bbR \to \matSU{2}{}$. Now let $\lambda_1,\,\lambda_2 \in \mathbb{S}^1$ be two distinct \sym points. The immersion $f:\bbC \to \matSU{2}{}$ in \eqref{eq:sym_s3} is then a conformal equivariant {\sc{cmc}} immersion with mean curvature \eqref{eq:cmc-I-II}.

Equivariance reduces the Gauss equation~\eqref{eq:metric} to $(\Metric')^2 + (\Metric^2 - 1)(\Metric^2 - \Bpoint^2) = 0$. A solution to this equation is the \emph{square root of the conformal factor} $\Metric:\bbR\to\bbR_+$ given by $\Metric (y) =  \jacobiDN(y,\,1-\Bpoint^2)$ where $\jacobiDN(y,\,m)$ is the Jacobi elliptic function. All other
solutions are of the form $\Metric(y+y_0)$, $y_0\in\bbR$. The
function $\Metric$ is even and has no zeros on $\bbR$. The range of
$\Metric$ is $[\min(1,\,\abs{\Bpoint}),\,\max(1,\,\abs{\Bpoint})]$.
The function satisfies $\Metric(0)=1$ and $\Metric'(0)=0$, and
limits to $\lim_{\abs{\Bpoint}\to 1}\Metric(y)=1$ and
$\lim_{\Bpoint\to 0}\Metric(y)=\sech y$. The period of $\Metric$
depends on the parameter $0<|\jq|\leq1$ and is equal to $2\JK'(\Bpoint)$. Now a straightforward calculation shows that the fundamental forms of such an equivariant conformal \cmc immersion are
\begin{subequations} \label{eq:cmc-equi-I-II}
\begin{gather}
    -\frac{ {(\lambda_2-\lambda_1)}^2 }{4\lambda_1\lambda_2}\,
     \Metric^2\,(dx^2 + dy^2) \spacecomma\\
    (v^2 H + \Real(Q))\,dx^2 -\Imag(Q)\,dx\, dy + (v^2 H - \Real(Q))\,dy^2
    \spacecomma
\end{gather}
\end{subequations}
with $H$ as in \eqref{eq:cmc-I-II} and Hopf differential $Q \,dz^2$ with $Q\coloneq \tfrac{\mi}{4} \jq(\lambda_2^{-1}-\lambda_1^{-1})$.
\begin{proposition}
A period of an equivariant extended frame is of the form
\begin{equation} \label{eq:period_gamma}
    \gamma =  x\,\pi + 2\mi\, p\,\JK'
\end{equation}
where $x\in\bbR$, $p\in\bbZ$, and $\Period$ is the period of
$\Metric$. The monodromy \eqref{eq:monodromy} with respect to such a period is
\begin{equation} \label{eq:equi-monodromy}
    M_\lambda (\gamma) = \exp(\pi(\,x\,\NU + p\,\OMEGA\,)\,e_0 )
    \spaceperiod
\end{equation}
\end{proposition}
\begin{proof}
The imaginary part of a frame period has to be a period of the square root of the conformal factor $\Metric:\bbR\to\bbR_+,\,y \mapsto   \jacobiDN(y,\,1-\Bpoint^2)$. Since $\Metric$ has period $2\JK'(\Bpoint)$, the imaginary part of $\gamma$ in \eqref{eq:period_gamma} has to be of the form $2\, p\,\JK'$ for some $p \in \bbZ$. From \eqref{eq:frame-parts} the extended frame of an equivariant \cmc torus is of the form $\FEQ(x,\,y)  \coloneq \exp\left((x\,\Vconst + \half\,\AngleZero) e_0\right) \exp\left( \half \AngleOne e_1 \right) \exp\left( \half \AngleTwo e_0 \right)$. The middle and right factor are periodic in $y$ and do not depend on $x$. Hence both these factors have trivial monodromy if $y \in 2\,\JK'\bbZ$, and the monodromy with respect to a translation by $\gamma =  x\,\pi + 2\mi\, p\,\JK'$ is
$$
    \FEQ(x\pi,\,2n\JK') = \exp( (\pi\,x\,\NU + \half\,\AngleZero(2p\JK'))\,e_0 )
$$
Clearly $\AngleZero(2p\JK')=p\AngleZero(2\JK')$ so it suffices to show that $\AngleZero(2\JK',\,\jq,\,\lambda) = 2\pi \omega(\jq,\,\lambda)$ to conclude the proof. Let $\lambda=e^{i\theta}$, then a calculation yields
\begin{align}   \label{eq:dJ1}
    d(\NU J_1) &=
    -\frac{\tfrac{d}{d\theta}(-\jq(\lambda+\lambda^{-1}))}
    {4\NU\jq(\lambda^{-1}-\lambda)}
    \left(
    \frac{-\jq\lambda^{-1}}{J_1} - \frac{1}{2} \frac{d^2}{dt^2}\log J_1(t)
    \right)\,d\theta
    \spacecomma
\\
\label{eq:dJ2}
    d(\NU J_2) &=\frac{\tfrac{d}{d\theta}(-\jq(\lambda+\lambda^{-1}))}
    {4\NU\jq(\lambda^{-1}-\lambda)}
    \left(
    \frac{-\jq\lambda}{J_2} - \frac{1}{2} \frac{d^2}{dt^2}\log J_2(t)
    \right)\,d\theta \spaceperiod
\end{align}
Since $J_1$ and $J_2$ are functions of $\Metric$,
they are also periodic with period $\Period$,
so
\begin{equation*} 
    \int_0^{\Period} \frac{d^2}{dt^2}\log J_k dt
    =
    \left.\frac{\tfrac{d}{dt}J_k}{J_k}\right|_{0}^{\Period} = 0\spaceperiod
\end{equation*}
Subtracting~\eqref{eq:dJ2} from~\eqref{eq:dJ1}
and integrating over the interval $[0,\,2\JK']$ gives
\begin{equation*}
    d\AngleZero =
    \frac{1}{4\NU}
    \left( 4\JE' -  \Period\jq\,(\lambda+\lambda^{-1})
    \right)\,d\theta  = 2\pi\,d\OMEGA \spaceperiod
\end{equation*}
Further, since
\begin{equation*}
    \AngleZero(2\JK',\,\jq,\,\lambda) = 2\mi \NU \int_0^{2\JK'} \frac{\jq}{v(\lambda^{-1}v - \jq v^{-1})} -  \frac{\jq}{v(\lambda v - \jq v^{-1})} \,dt
\end{equation*}
and $\NU(\lambda^{-1}) = \NU(\lambda)$ it follows that $\AngleZero(2\JK',\,\jq,\,\lambda^{-1}) = - \AngleZero(2\JK',\,\jq,\,\lambda)$. Thus $\AngleZero$ shares the properties of $\OMEGA$ which determine it uniquely. Hence $\AngleZero(2\JK',\,\jq,\,\lambda) = 2\pi \omega(\jq,\,\lambda)$.
\end{proof}
 By \eqref{eq:equi-monodromy} we have $\tau_{\gamma} f = M_{\lambda_1}(\gamma)\,f\,M_{\lambda_2}^{-1}(\gamma)
= \exp(\pi( x\,\NU_1 + n\,\OMEGA_1)\,e_0 )\,f\,
\exp( -\pi( x\,\NU_2 + n\, \OMEGA_2)\,e_0 )$,
so translation by a period $\gamma$ \eqref{eq:period_gamma}
induces an ambient isometry. The \emph{equivariant action} is the action of the 1-parameter group
$\GK$ of isometries of $\bbS^3$ defined by
\begin{equation} \label{eq:equi_action}
    \GK = \{ g_x \in \mathrm{Iso}(\mathbb{S}^3) \mid g_x(p) =
    \exp(x\,\Vconst_1\,e_0)\,
    p \, \exp(-x\,\Vconst_2\,e_0)\,, x \in \bbR \}\,.
\end{equation}
Since $\NU_1\not=0\not=\NU_2$ the commutator of the equivariant
action \eqref{eq:equi_action}
\begin{equation}\label{eq:commut}
    \hat{\GK} = \{ g \in \mathrm{Iso}(\mathbb{S}^3) \mid gk = kg \mbox{
    for all } k\in\GK\}
\end{equation}
in the group $\mathrm{Iso}(\bbS^3)$ of orientation preserving
isometries of $\bbS^3$ is a two-dimensional torus. With the exception of two geodesics the orbits of $\hat{\GK}$ are two-dimensional embedded tori. These geodesics, which we call the \emph{axes} of the surface, are linked, and are situated so that every geodesic 2-sphere through one is orthogonal to the other. Every orbit of the equivariant action~\eqref{eq:equi_action}, with the exception of the two axes, is a $(m,\,n)$-torus knot in the corresponding orbit of $\hat{\GK}$, with
\begin{equation} \label{eq:pq_nu}
    \frac{m}{n}=\frac{\NU_1-\NU_2}{\NU_1+\NU_2}\,.
\end{equation}
If we identify
$$
    \bbS^3=\matSU{2}{} = \left\{ \left. \begin{pmatrix} a & b\\
    -\bar{b} & \bar{a} \end{pmatrix} \,\, \right| \,\, |a|^2+|b|^2 =1 \right\} \quad \mbox{and choose}
    \quad e_0 = \begin{pmatrix} \mi & 0 \\ 0 & -\mi \end{pmatrix} \spacecomma
$$
then the equivariant action extends to an action on $(a,b)\in\bbC^2$ given by $R(s,\,t)(a,\,b)=(e^{\mi s}a,\,e^{\mi t}b)$, called the \emph{extended action}. In particular, the translation $\tau_\gamma$ by a period $\gamma$ \eqref{eq:period_gamma} induces the ambient isometry $R(s,\,t)$ with
$s= x\,(\NU_1-\NU_2) + p\,(\OMEGA_1-\OMEGA_2)$ and
$t= x\,(\NU_1+\NU_2) + p\,(\OMEGA_1+\OMEGA_2)$.
\begin{proof}[Proof of Proposition \ref{prop:torus-of-revolution}]
For a flat \cmc torus this can always be achieved since we have not used the freedom of the M\"obius transformation.
A spectral genus one torus is a surface of revolution, if and only if the equivariant action is the rotation around a geodesic, or
equivalently if \eqref{eq:equi_action} fixes point wise one geodesic
of $\bbS^3$. The generator of the extended action has eigenvalues
$\mi(\Vconst_1\pm\Vconst_2)$. Thus there exists a zero eigenvalue, if and only if $\Vconst_2=\pm\Vconst_1$, which is equivalent to $\lambda_2 = \lambda_1^{-1}$.
\end{proof}
\begin{proposition} \label{prop:closing-conditions}
A spectral genus one \cmc surface in $\bbS^3$ is closed along two
independent periods if and only if there exists a
$T\in\bbZ^3\setminus\{0\}$ with $T\cdot X = 0 = T\cdot Y$ for
$X\coloneq \left( 0,\,\NU_1 ,\,\NU_2 \right)$ and $Y\coloneq \left(1,\,\OMEGA_1,\,\OMEGA_2 \right)$. 
%
\end{proposition}
\begin{proof}
Suppose we have two $\bbR$-independent periods
$\gamma_j = x_j\pi + 2\mi p_{j0}\JK' \in \bbC^\times$ for some
$x_j\in\bbR$ and $p_{j0}\in\bbZ$, $j=1,\,2$. Let $M_\lambda(\gamma_j)$ be the respective frame monodromies with eigenvalues $\mu_j^{\pm 1}$.
Then there exist four further integers $p_{jk} \in \bbZ$ as in \eqref{eq:pjk} for $j,\,k=1,\,2$. Hence $p_{jk} =  x_j \, \NU_k + p_{j0}\,\OMEGA_k$, and we write this system as
\begin{equation} \label{eq:closing2}
\begin{pmatrix}
    x_1 & p_{10} \\ x_2 & p_{20}
\end{pmatrix}
\begin{pmatrix}
    X\\Y
\end{pmatrix}
=
\begin{pmatrix}
    p_{10} & p_{11} & p_{12}\\
    p_{20} & p_{21} & p_{22}
\end{pmatrix}
\coloneq
\begin{pmatrix}
    P\\Q
\end{pmatrix}
\spaceperiod
\end{equation}
Hence $T\coloneq P\times Q$ is in
$\bbZ^3\setminus\{0\}$ and satisfies $T\cdot X=0$ and $T\cdot Y=0$.

Conversely suppose that there exists $T\in \bbZ^3\setminus\{0\}$
satisfying $T\cdot X=0$ and $T\cdot Y=0$. Let $\{P,\,Q\}$ be a basis
for the lattice $\Lambda=\{P\in\bbZ^3\suchthat T\cdot P=0\}$. Then
there exist $x_1,\,x_2\in\bbR$ and $p_{10},\,p_{20}\in\bbZ$ such
that~\eqref{eq:closing2} holds. Then $\gamma_j =  x_j\pi + 2\mi\, p_{j0}\JK'\in\Cstar$ generate a lattice with respect to which the surface is doubly periodic.
\end{proof}
%
%
\begin{remark}\rm{
The closing conditions in Proposition~\ref{prop:closing-conditions} can be used to describe the intersection of the zero sets of two functions on the parameter space $(\jq,\,\jk,\,\jh)$. The curve
forming the intersection of two level sets then integrates to the vector field~\eqref{eq:torus-flow}. For $X,\,Y$ as in Proposition~\ref{prop:closing-conditions}, there exists
$s=(s_0,\,s_1,\,s_2)\in\bbZ^3$ such that $s\cdot X=0$ and $s\cdot
Y=0$. The closing conditions are thus $F=0,\,G=0$, where
$F\coloneq s\cdot X$ and $G\coloneq s\cdot Y$. If we set
$\lambda_k= e^{2\mi\theta_k}$, then the system of
implicit flow equations is
\begin{equation*}
    \begin{pmatrix}
    \frac{\del F}{\del\Bpoint} & \frac{\del F}{\del\theta_1} &
    \frac{\del F}{\del\theta_2}\\
    \frac{\del G}{\del\Bpoint} & \frac{\del G}{\del\theta_1} &
    \frac{\del G}{\del\theta_2}
    \end{pmatrix}
    \begin{pmatrix}
    \dot\Bpoint \\ \dot \theta_1 \\ \dot \theta_2
    \end{pmatrix}
    = 0\spacecomma
\end{equation*}
of which we next compute the matrix on the right hand side. Up
to scale, $(\dot \Bpoint,\, \dot \theta_1,\, \dot \theta_2)$ is
a cross product of its rows. Hence
\begin{equation*}
    \begin{pmatrix}
    \frac{\del F}{\del \Bpoint} & \frac{\del F}{\del \theta_1} &
    \frac{\del F}{\del \theta_2}\\
    \frac{\del G}{\del \Bpoint} & \frac{\del G}{\del \theta_1} &
    \frac{\del G}{\del \theta_2}
    \end{pmatrix}
    =
    \begin{pmatrix}
    s_1\frac{\del\NU_1}{\del \Bpoint} + s_2\frac{\del\NU_2}{\del
    \Bpoint} & s_1\frac{\del\NU_1}{\del\theta_1} &
    s_2\frac{\del\NU_2}{\del\theta_2}
    \\
    s_1\frac{\del\OMEGA_1}{\del \Bpoint} + s_2\frac{\del\OMEGA_2}{\del
    \Bpoint} & s_1\frac{\del\OMEGA_1}{\del\theta_1} &
    s_2\frac{\del\OMEGA_2}{\del\theta_2}
    \end{pmatrix}\spaceperiod
\end{equation*}
Since $G=s_1\NU_1+s_2\NU_2=0$, this matrix is a scalar multiple of
\begin{equation*}
    \begin{pmatrix}
    V_1\\V_2
    \end{pmatrix}
    \coloneq
    \begin{pmatrix}
    \NU_2\frac{\del\NU_1}{\del \Bpoint} -\NU_1\frac{\del\NU_2}{\del
    \Bpoint} & \NU_2\frac{\del\NU_1}{\del\theta_1} &
    -\NU_1\frac{\del\NU_2}{\del\theta_2}
    \\
    \NU_2\frac{\del\OMEGA_1}{\del \Bpoint}
    -\NU_1\frac{\del\OMEGA_2}{\del \Bpoint} &
    \NU_2\frac{\del\OMEGA_1}{\del\theta_1} &
    -\NU_1\frac{\del\OMEGA_2}{\del\theta_2}
    \end{pmatrix}\spaceperiod
\end{equation*}
Hence $(\dot \Bpoint,\,\dot\theta_1,\,\dot\theta_2)$ is a scalar
multiple of $V_1\cross V_2$. The derivatives of $\NU$ and $\OMEGA$
with respect to $\Bpoint$ and $\theta$, where
$2\mi\,\theta \coloneq \ln\lambda$ were computed in \eqref{eq:omega}, \eqref{eq:nu dot} and \eqref{eq:omega'}. A calculation yields that $V_1\cross V_2$ is a scalar multiple of
\begin{equation*}
    \begin{pmatrix}
    2\Bpoint(\JE'\cos(\theta_1+\theta_2)-
            \Bpoint\JK'\cos(\theta_1-\theta_2) ) \\
    -\frac{(1+\jq^2)\JE'-2\JK'}{1-\jq^2}
      \sin\bigl(\theta_1+\theta_2\bigr)
    +\frac{2\JE'-(1+\jq^2)\JK'}{1-\jq^2}
      \sin\bigl(\theta_1-\theta_2\bigr) \\
    -\frac{(1+\jq^2)\JE'-2\JK'}{1-\jq^2}
      \sin\bigl(\theta_1+\theta_2\bigr)
    -\frac{2\JE'-(1+\jq^2)\JK'}{1-\jq^2}
      \sin\bigl(\theta_1-\theta_2\bigr)
    \end{pmatrix}
    \spaceperiod
\end{equation*}
Changing variables from $(\Bpoint,\,\theta_1,\,\theta_2)$ to
$(\Xzero,\,\Xone,\,\Xtwo)$, and rescaling,
yields~\eqref{eq:torus-flow}.}
\end{remark}
\begin{figure}[t]
\centering
\includegraphics[width=5.35cm]{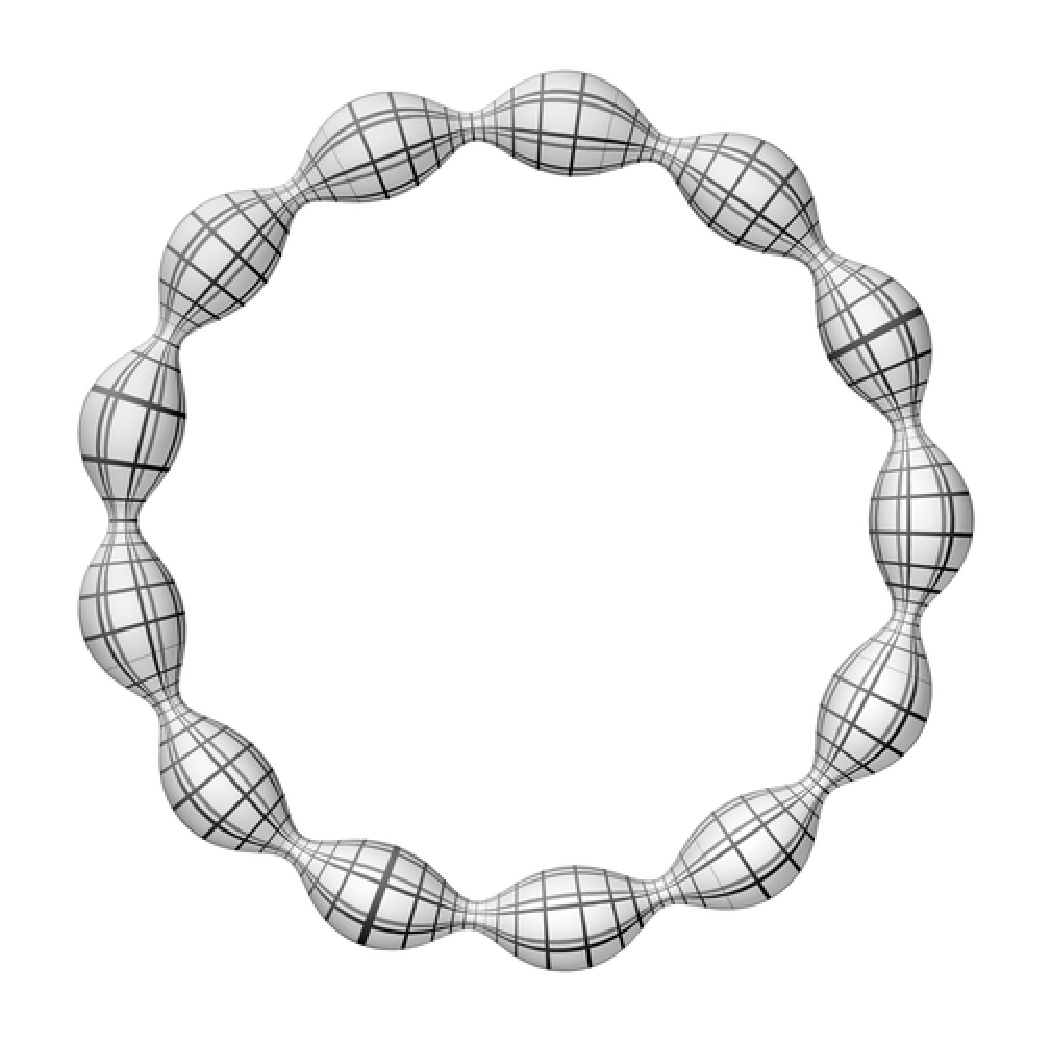}
\includegraphics[width=5.35cm]{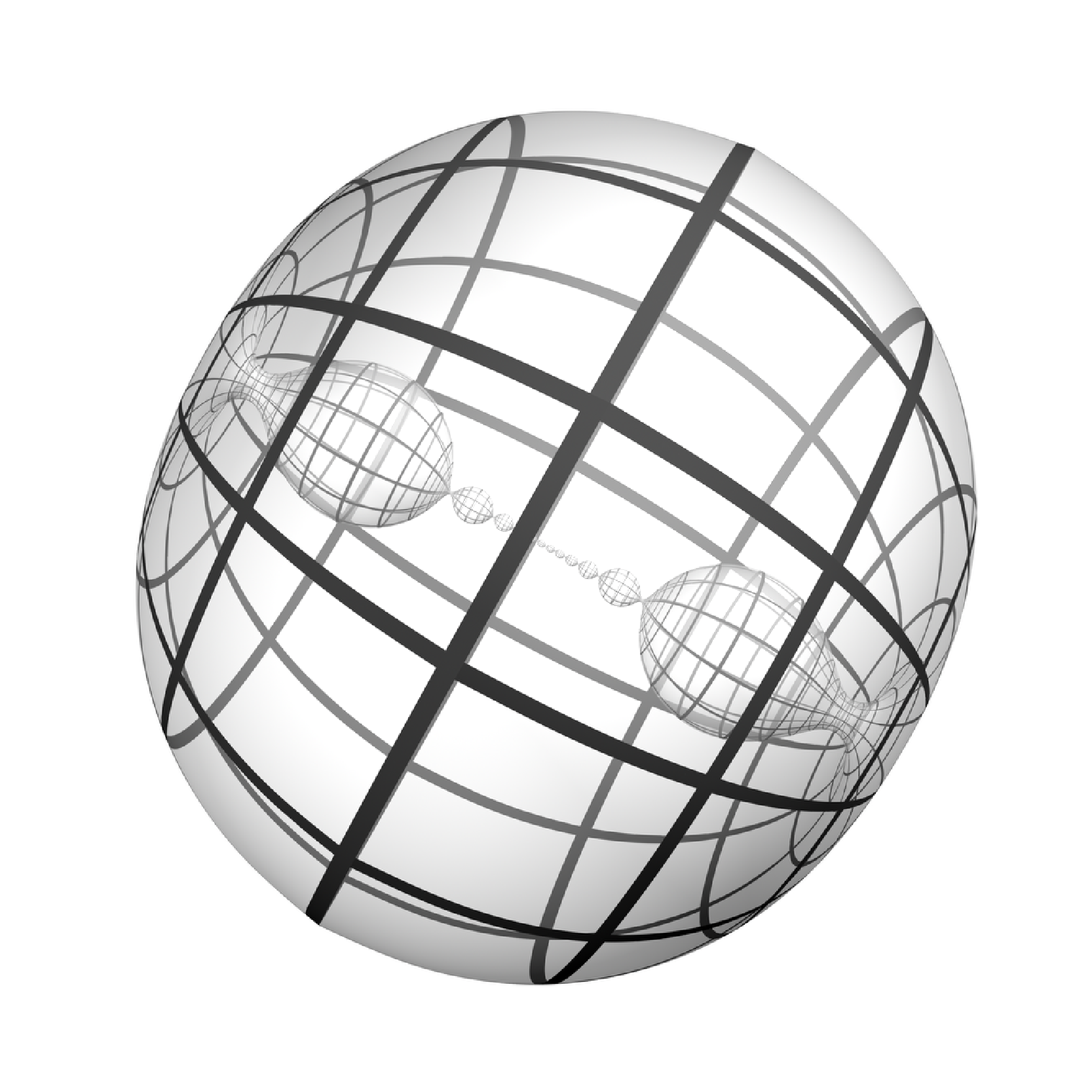}
\includegraphics[width=5.35cm]{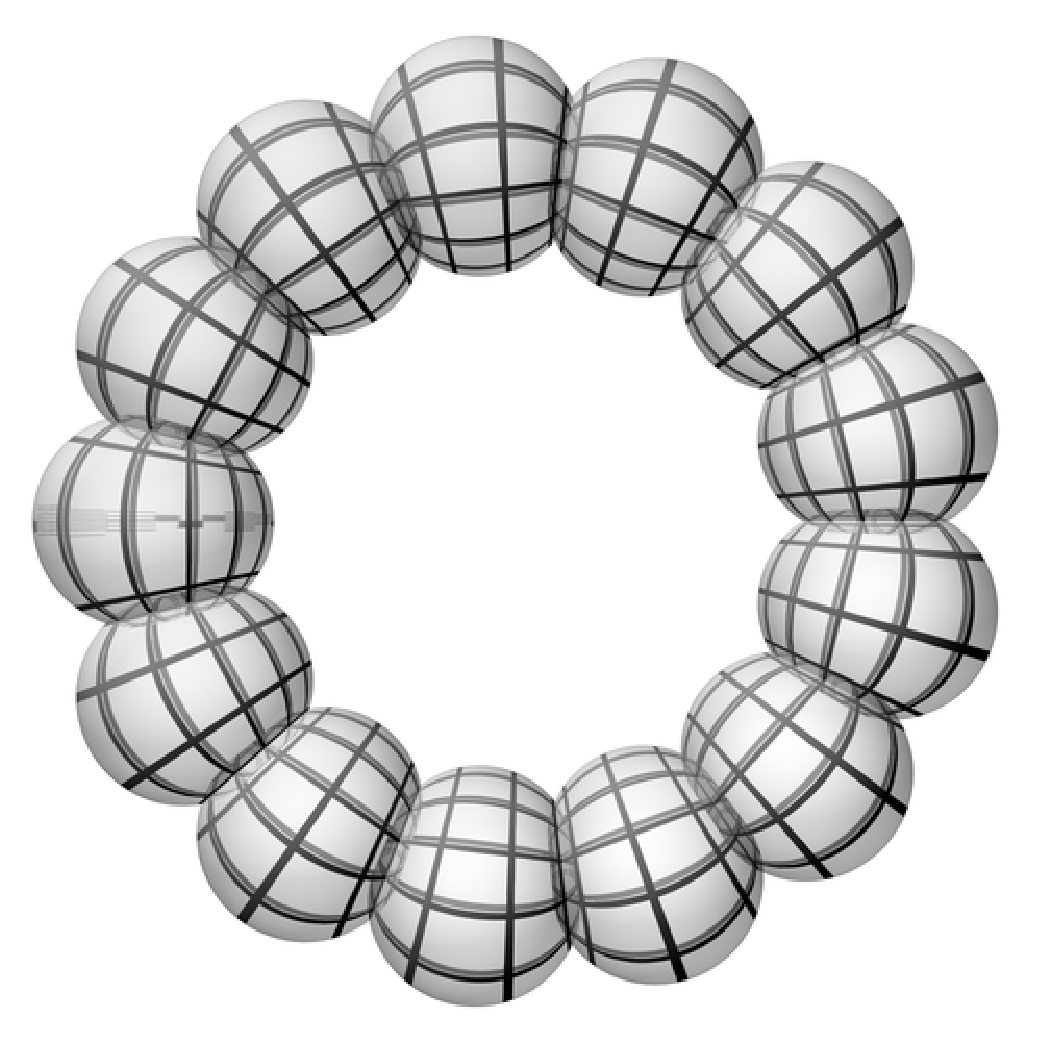}
\caption{ \label{fig:thirteenlobe} A sampling of $(k,\,13)$ \cmc
tori of revolution in $\bbS^3$. }
\end{figure}


\section{The moduli space of equivariant \cmc tori in $\mathbb{S}^3$}

We next determine those flat \cmc tori which allow a bifurcation into genus $g=1$ \cmc tori. These are precisely those flat \cmc tori whose spectral curves have a double point on the unit circle.
We will show in Theorem~\ref{thm:moduli-space} that the spectral genus one tori lie in flow families which begin at a flat \cmc torus with a double point on $\bbS^1$. These flat \cmc tori hence serve as initial conditions for the spectral genus $g=1$ flow~\eqref{eq:torus-flow}. The purpose of this section is to classify the flat \cmc tori in $\bbS^3$ with a double point on $\bbS^1$.

We will show that three integers determine a primitive flat \cmc torus with a double point on $\bbS^1$ up to the $\bbS^1$ action and complex conjugation on the spectral curve. This is done by parameterizing the configuration space of three marked points on $\bbS^1$ as follows: Let
$\Delta \subset \bbT^3 \coloneq \bbS^1 \times \bbS^1 \times \bbS^1$
denote the union of all triples of unimodular numbers in which two of the three entries coincide. Let $\calT= (\bbT^3 \setminus\Delta)/\thicksim$, where we identify triples obtained by the three transformations:
\begin{itemize}
\item[(A)] Rotation
    $(\lambda_0,\,\lambda_1,\,\lambda_2)\mapsto\delta(\lambda_0,\,\lambda_1,\,\lambda_2)$
    with $\delta\in\bbS^1$.
    \item[(B)] Inversion
    $(\lambda_0,\,\lambda_1,\,\lambda_2)\mapsto
    (\lambda_0^{-1},\,\lambda_1^{-1},\,\lambda_2^{-1})$.
    \item[(C)] Swapping of the \sym points $(\lambda_0,\,\lambda_1,\,\lambda_2)\mapsto(\lambda_0,\,\lambda_2,\,\lambda_1)$.
\end{itemize}
Now we define a mapping
\begin{equation}\label{eq:phi}
\begin{split}
    \phi:\,(\lambda_0,\,\lambda_1,\,\lambda_2)&\mapsto
    \left(\frac{|s_0|}{\max\{|s_1+s_2|,\,|s_1-s_2|\}},\,
    \frac{\min\{|s_1+s_2|,\,|s_1-s_2|\}}{\max\{|s_1+s_2|,\,|s_1-s_2|\}}
    \right)\quad \mbox{ where }\\
    (s_0,\,s_1,\,s_2)&=
    \mi(\lambda_0^{1/2},\,\lambda_1^{1/2},\,\lambda_2^{1/2})\times
    (\lambda_0^{-1/2},\,\lambda_1^{-1/2},\,\lambda_2^{-1/2})
    \end{split}
\end{equation}
\begin{proposition} \label{thm:tau}
The map~\eqref{eq:phi} is bijective from $\calT$ onto $\calR=\{(t_0,\,t_1)\in [0,\,1)^2 \mid t_1<t_0\}$.
\end{proposition}
\begin{proof}
For $(\lambda_0,\lambda_1,\lambda_2)\in\bbT^3$ choose square roots
and set $m := (\lambda_0^{1/2},\,\lambda_1^{1/2},\,\lambda_2^{1/2})$. Then
$s = \mi m\cross \ol m$ and $\ol s = -\mi \ol m\cross m =\mi
m\cross\ol m = s$, so $s$ is real. If $s=0$, then $\ol
m$ is contained in the $\bbS^1$ orbit of $m$, and
$\lambda_0=\lambda_1=\lambda_2$ so
$(\lambda_0,\lambda_1,\lambda_2)\in\Delta$. This shows $s\ne 0$,
and $s\in\bbR\bbP^2$.

For the map $\maptau: \bbR\bbP^3 \to \bbR$ defined by
\begin{align*} \label{eq:tau}
    \maptau(s_0,\,s_1,\,s_2)& =
    -(s_0+s_1+s_2)(-s_0+s_1+s_2)(s_0-s_1+s_2)(s_0+s_1-s_2)\\
&=(\lambda_0^{1/2}\lambda_1^{-1/2}-\lambda_1^{1/2}\lambda_0^{-1/2})^2
(\lambda_0^{1/2}\lambda_2^{-1/2}-\lambda_2^{1/2}\lambda_0^{-1/2})^2
(\lambda_1^{1/2}\lambda_2^{-1/2}-\lambda_2^{1/2}\lambda_1^{-1/2})^2
\end{align*}
we have $\maptau(s) \leq 0$ since each of the three squared factors is non-positive. Furthermore, $\maptau(s)=0$ is equivalent to
$(\lambda_0,\,\lambda_1,\,\lambda_2)\in\Delta$. Hence we have
$\maptau(s)<0$. Define
\begin{itemize}
\item[(A')] Rotation
    $(\lambda_0^{1/2},\,\lambda_1^{1/2},\,\lambda_2^{1/2})\mapsto \delta(\lambda_0^{1/2},\,\lambda_1^{1/2},\,\lambda_2^{1/2})$ with $\delta\in\bbS^1$.
    \item[(B')] Inversion
    $(\lambda_0^{1/2},\,\lambda_1^{1/2},\,\lambda_2^{1/2})\mapsto
    (\lambda_0^{-1/2},\,\lambda_1^{-1/2},\,\lambda_2^{-1/2})$.
    \item[(C')] Swapping of the \sym points
    $(\lambda_0^{1/2},\,\lambda_1^{1/2},\,\lambda_2^{1/2})\mapsto
    (\lambda_0^{1/2},\,\lambda_2^{1/2},\,\lambda_1^{1/2})$.
\item[(D')] Changing the signs of the entries of $(\lambda_0^{1/2},\,\lambda_1^{1/2},\,\lambda_2^{1/2})$ independently.
\end{itemize}
The transformations (A')-(B') do not change $s\in\bbR\bbP^2$, while
(C') negates signs and swaps $s_1$ and $s_2$, while (D') negates two signs in $s$ for each sign change. If $s_j =0$ for at least one $j\in\{0,\,1,\,2\}$ then
$\maptau(s_0,\,s_1,\,s_2) \geq 0$. Hence if $s\in\bbR\bbP^2$, then
$\maptau(s)<0$ implies $s_j\ne 0$ for all $j\in\{0,\,1,\,2\}$. Then
$s=\mi m \times \ol m$ if and only if $s \cdot m = 0$ and $s \cdot \ol m = 0$. These are equivalent to
\begin{equation*}
\lambda_1 - 2A_1 \lambda_1^{1/2}\lambda_0^{1/2} + \lambda_0 =
0\spacecomma\interspace \lambda_2 - 2A_2 \lambda_2^{1/2}\lambda_0^{1/2} + \lambda_0 = 0\spacecomma\interspace
s_0 \lambda_0^{1/2} + s_1 \lambda_1^{1/2} + s_2 \lambda_2^{1/2}
= 0\spacecomma
\end{equation*}
where
\begin{equation*}
    A_1 = \frac{-s_0^2-s_1^2+s_2^2}{2s_0s_1}
    = 1+\frac{\maptau(s)}{{(2s_0s_1)}^2}
    \AND A_2 = \frac{-s_0^2+s_1^2-s_2^2}{2s_0s_1}
    = 1+\frac{\maptau(s)}{{(2s_0s_2)}^2}
\spaceperiod
\end{equation*}
Thus $s\in\bbR\bbP^2$ with $\maptau(s)<0$ determines the following
elements of $\bbT^3$ uniquely up to (A') and (B'):
\begin{equation}   \label{eq:t-vector}
    (\lambda_0^{1/2},\,\lambda_1^{1/2},\,\lambda_2^{1/2}) =
    \left( 1,\, %
    \frac{-s_0^2-s_1^2+s_2^2 \pm \sqrt{\maptau(s)}}{2s_0s_1},\ %
    \frac{-s_0^2+s_1^2-s_2^2 \mp \sqrt{\maptau(s)}}{2s_0s_2}
    \right)\spaceperiod
\end{equation}
Furthermore the numbers
\begin{equation}\label{eq:integers-flat}
    (\TurnZero,\,\TurnOne,\,\TurnTwo)=\tfrac{1}{2}\,
    (|s_0|,\,\min\{|s_1+s_2|,|s_1-s_2|\},\,
    \max\{|s_1+s_2|,\,|s_2-s_2|\})
\end{equation}
determine $(s_0,\,s_1,\,s_2)$ up to the transformations $C'$ and $D'$. Due to $\maptau(s) = (s_0^2-(s_1-s_2)^2)(s_0^2-(s_1+s_2)^2) =
4(\TurnZero^2-\TurnOne^2)(\TurnZero^2-\TurnTwo^2)$ the condition $\maptau(s)<0$ is equivalent to $0\leq\TurnOne<\TurnZero<\TurnTwo$. This shows that $\phi$ is surjective from $\bbT^3\setminus\Delta$ onto $\calR$, whose pre-images are uniquely determined up to (A)-(C).
\end{proof}
The {\emph{spectral data}} of a flat \cmc torus with a double point is a triple $(\lambda_0,\,\lambda_1,\,\lambda_2)\in\calT$ of values of the spectral parameter at the double point and the two \sym points. We identify triples obtained by the transformations~(A)-(C). Hence the set of spectral data is isomorphic to $\calR \cap \bbQ^2$.

The \emph{turning number} of a plane curve is the degree
of its Gauss map; we take the turning number to be unsigned. The
\emph{total turning number} of a collection of immersed curves is
the sum of their turning numbers. For twizzled surfaces, a \emph{profile curve set} of the surface with respect to one of its axes $A$ is constructed as follows: Let $\bbS^2_A$ be a geodesic 2-sphere containing $A$. The axis $A$ divides $\bbS^2_A$ into two hemispheres. Then a profile curve set of the torus with respect to $A$ is the intersection of the surface with $\bbS^2_A$ or with one of the hemispheres of $\bbS^2_A$. By the equivariant action, all profile curve sets associated with an axis are isometric. For surfaces of revolution, there is a profile curve set with respect to the axis which is not the revolution axis.

\begin{figure}[t]
\centering
\includegraphics[width=8.15cm]{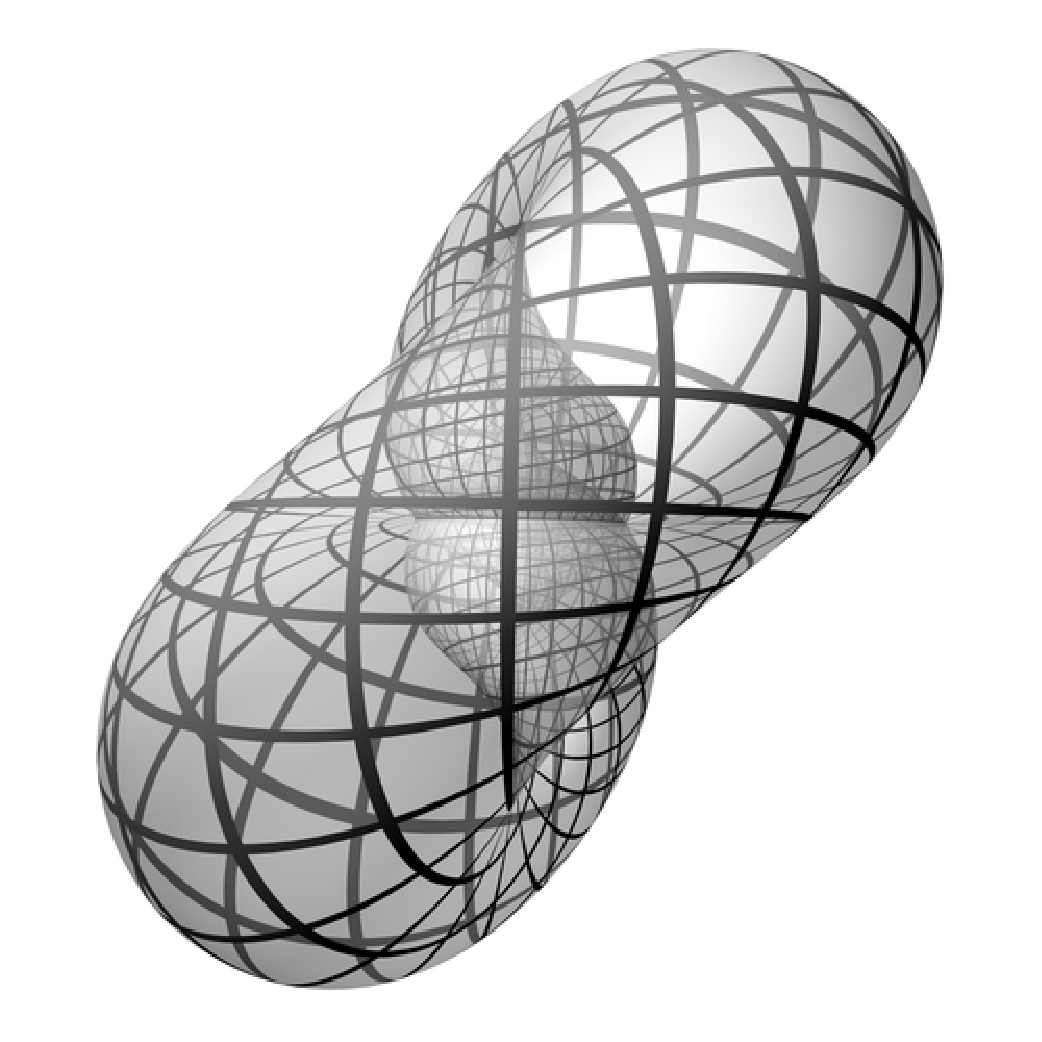}
\includegraphics[width=8.15cm]{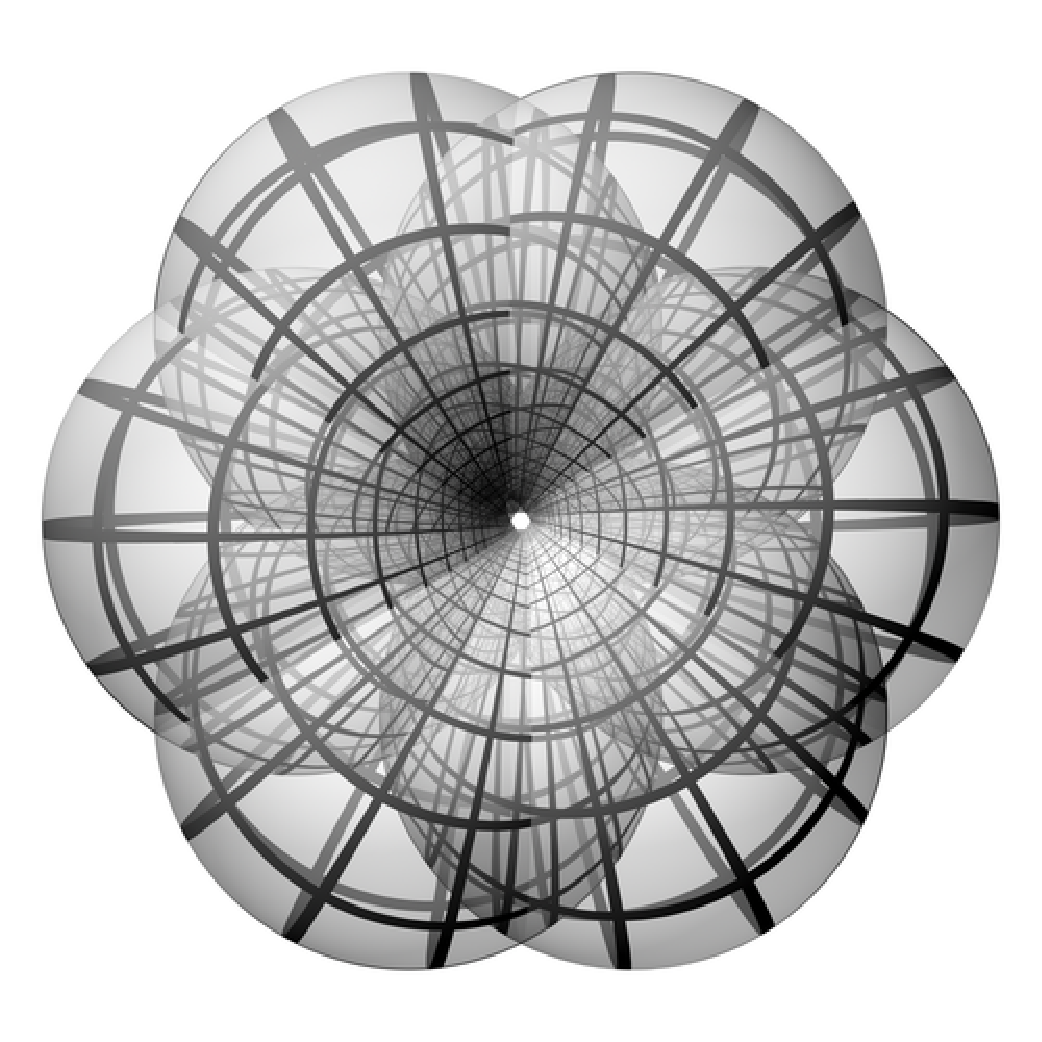}
\caption{ \label{fig:minimal}
The minimal $(2,\,1,\,3)$ and $(3,\,1,\,6)$ tori.
Stereographically projected from $\bbS^3$
to $\bbE^3$.
}
\end{figure}

\begin{theorem} \label{thm:flat-torus-integers}
\mbox{}

{\rm{(1)}} \label{item:integer-inequality}
The set of spectral data of flat \cmc tori with a double point on $\bbS^1$ up to transformations~(A)-(C) is in one-to-one correspondence with integer triples
\begin{equation}\label{eq:triple}
	(\TurnZero,\,\TurnOne,\,\TurnTwo)\in\bbZ^3 \mbox{ with }
	\gcd(\TurnZero,\,\TurnOne,\,\TurnTwo)=1 \mbox{ satisfying }
	0 \le \TurnOne < \TurnZero < \TurnTwo\,.
\end{equation}
{\rm{(2)}} \label{item:flat-torus-covering}
Let $\bbT^2$ be a $(\TurnZero,\,\TurnOne,\,\TurnTwo)$ flat \cmc torus. Then $\bbT^2$ covers its underlying flat embedded rectangular \cmc torus $\TurnZero$ times. Each profile curve set $\calC_k$ ($k=1,\,2$) of $\bbT^2$ has total turning number $\Turn_0$. The set is a union of $\gcd(\TurnZero,\,\Turn_k)$ coinciding circles, and each circle is wrapped $\Turn_0/\gcd(\TurnZero,\,\Turn_k)$ times.

{\rm{(3)}} \label{item:flat-torus-rev} The case $\TurnOne=0$ occurs if and only if the torus is rotational.
\end{theorem}
\begin{proof}
Let $m = (\lambda_0^{1/2},\,\lambda_1^{1/2},\,\lambda_2^{1/2})$ be the square roots of spectral data of a flat \cmc torus, and assume the torus has period lattice $\Gamma = \gamma_1\bbZ \times \gamma_2 \bbZ$. We can frame the torus by an extended frame \eqref{eq:flat-frame}, and then the logarithmic eigenvalues of the monodromy with respect to these periods are $\ln \mu(\gamma_j,\,\lambda_k)$ as in \eqref{eq:eigenvalues}. Then there exist six integers $p_{jk} \in \bbZ$ such that $\ln \mu(\gamma_j,\,\lambda_k) = \pi\mi p_{jk}$ for $j=1,\,2$ and $k=0,\,1,\,2$. As in Proposition~\ref{prop:flat-torus} let $\Lambda= \kappa_1\bbZ \times \kappa_2\bbZ$, and $\Lambda^\ast = \gamma^*_1\bbZ \times \gamma^*_2\bbZ$ the lattice of the underlying embedded rectangular torus with orthogonal basis $\DOTR{\gamma^*_j}{\kappa_k} = \delta_{jk}$. Since $\Gamma \subset \Lambda^\ast$, we can expand
\begin{equation} \label{eq:flat-2.1}
    \begin{pmatrix} \gamma_1 \\ \gamma_2 \end{pmatrix} = \begin{pmatrix} \DOTR{\gamma_1}{\kappa_1} & \DOTR{\gamma_1}{\kappa_2}\\
    \DOTR{\gamma_2}{\kappa_1} & \DOTR{\gamma_2}{\kappa_2} \end{pmatrix} \,\begin{pmatrix}
    \gamma^*_1 \\ \gamma^*_2 \end{pmatrix} \,.
\end{equation}
Since $p_{jk} = \DOTR{\gamma_j}{\lambda_k^{1/2}}$, the determinant of the above change of basis is
\begin{equation}  \label{eq:det_is_l0}
    \left|\,\det\left(\DOTR{\gamma_j}{\kappa_k}\right)\, \right| = \tfrac{1}{2}\left| \,p_{12}p_{21} - p_{11}p_{22} \,\right|\,.
\end{equation}
(1) Define $(\tilde{s}_0,\, \tilde{s}_1,\,\tilde{s}_2) \coloneq (p_{10},\,p_{11},\,p_{12}) \times (p_{20},\,p_{21},\,p_{22})$. By construction  $(\tilde{s}_0,\,\tilde{s}_1,\,\tilde{s}_2) \in \bbZ^3$. A computation gives $(\tilde{s}_0,\,\tilde{s}_1,\,\tilde{s}_2) = (\overline{\gamma}_1\gamma_2- \gamma_1\overline{\gamma}_2) \,m \times \ol m$. Note that
$\overline{\gamma}_1\gamma_2 - \gamma_1\overline{\gamma}_2 \in \mi
\bbR$, and is never zero, since $\gamma_1,\,\gamma_2$ are not
collinear. From \eqref{eq:closing-dot} we additionally have that $p_{11}-p_{12} \in 2\,\bbZ$ and $p_{21}-p_{22} \in 2\,\bbZ$. Hence it follows that $\tilde{s}_0,\,\tilde{s}_1 + \tilde{s}_2,\,\tilde{s}_1 - \tilde{s}_2 \in 2\,\bbZ$. There exists a unique representative
$(s_0,\,s_1,\,s_2)\in\bbZ^3$ of
$(\tilde{s}_0,\,\tilde{s}_1,\,\tilde{s}_2)\in\bbR\bbP^2$ with
$\gcd(s_0,\,s_1 + s_2,\,s_1 - s_2) =2$. The corresponding
numbers $(\TurnZero,\,\TurnOne,\,\TurnTwo)$ defined in \eqref{eq:integers-flat} then obey \eqref{eq:triple}.
We thus have a map from spectral data to integer triples obeying~\eqref{eq:triple}, and by Proposition~\ref{thm:tau} this map is bijective.

(2) The vector $s= (s_0,\,s_1,\,s_2)$ determines a lattice $\Gamma_s\subset\Lambda^\ast$ defined by
\begin{equation}\label{eq:sublattice}
    \Gamma_s=\left\{
    \tfrac{p_1+p_2}{2}\gamma_1^\ast+\tfrac{p_1-p_2}{2}\gamma_2^\ast\mid
    (p_0,\tfrac{p_1+p_2}{2},\tfrac{p_1-p_2}{2})\in\mathbb{Z}^3
    \mbox{ with }s_0p_0+s_1p_1+s_2p_2=0\right\} \spaceperiod
\end{equation}
The lattice $\Gamma_s$ contains all the periods of the torus with respect to which the logarithmic eigenvalues of the monodromy do not change their values at the double point. The lattice $\Gamma_s$ does not change if $s$ is multiplied by some integer, or the sign of $s_0$ is changed. Due to \eqref{eq:integers-flat} the integers $s$ are determined by $(\TurnZero,\,\TurnOne,\,\TurnTwo)$ up to transformations (C') and (D'). Two co-linear $s$ correspond to the same lattice $\Gamma_s$. Switching the signs of $s_0$ does not change $\Gamma_s$. Therefore the lattices $\Gamma_s$ of the flat \cmc tori with triple $(\TurnZero,\,\TurnOne,\,\TurnTwo)$ satisfying \eqref{eq:triple} are those $\Gamma_s$ with one of the following vectors $s$ given by
\begin{equation}\label{eq:fours} \begin{split}
    &(2\TurnZero,\TurnOne+\TurnTwo,\TurnOne-\TurnTwo)\,,\qquad
    (2\TurnZero,\TurnOne+\TurnTwo,\TurnTwo-\TurnOne)\,,\\
    &(2\TurnZero,\TurnOne-\TurnTwo,\TurnOne+\TurnTwo)\,,\qquad
    (2\TurnZero,\TurnTwo-\TurnOne,\TurnOne+\TurnTwo)\,.
    \end{split}
\end{equation}
The transformations (C') and (D') act on sublattices $\Upsilon\subset\Lambda^\ast$ as follows:
\begin{itemize}
\item[(C'')] $C''\Upsilon = \left\{ n_1\gamma_1^\ast+n_2\gamma_2^\ast \mid n_1\gamma_1^\ast - n_2\gamma_2^\ast\in\Upsilon \right\}$,
\item[(D'')] $D''\Upsilon = \left\{ n_1\gamma_1^\ast+n_2\gamma_2^\ast \mid n_2\gamma_1^\ast + n_1\gamma_2^\ast\in\Upsilon \right\}$.
\end{itemize}
Since $p_02\TurnZero+p_1(\TurnOne+\TurnTwo)+p_2(\Turn_1-\TurnTwo)= 2(p_0\TurnZero+\frac{p_1+p_2}{2}\TurnOne+\frac{p_1+p_2}{2}\TurnTwo)$ the lattices corresponding to the four vectors $s$ in \eqref{eq:fours} are respectively equal to
\begin{equation}\label{eq:sublattices} \begin{split}
    &\Gamma_{[\TurnZero,\,\TurnOne,\,\TurnTwo]}=
    \{n_1\gamma_1^\ast+n_2\gamma_2^\ast\mid
    n_1\TurnOne+n_2\TurnTwo\in\TurnZero\,\bbZ\,\}\,,\quad
    C''\;\Gamma_{[\TurnZero,\,\TurnOne,\,\TurnTwo]}\,,\\
    &D''\;\Gamma_{[\TurnZero,\,\TurnOne,\,\TurnTwo]}\, \quad \mbox{ and } \quad
    D''\; C''\;\Gamma_{[\TurnZero,\,\TurnOne,\,\TurnTwo]}\,.
\end{split}
\end{equation}
For flat \cmc tori the two-dimensional group $\hat{K}$~\eqref{eq:commut} acts on the torus. Since the axes do not depend on the subgroup $K\subset\hat{K}$~\eqref{eq:equi_action}, the two subgroups corresponding to the rotations of the embedded torus act on a geodesic sphere containing one of the axes. Hence the smallest periods in $\Gamma_s\cap\gamma_1^\ast\,\bbZ$ and $\Gamma_s\cap\gamma_2^\ast\,\bbZ$ represent components of the profile curve sets. These wrapping numbers are $\TurnZero/\gcd(\TurnZero,\TurnOne)$ and $\TurnZero/\gcd(\TurnZero,\TurnTwo)$. The number of components times these wrapping numbers is equal to $|\Lambda^\ast/\Gamma_s|$. But  $\TurnZero =|\Lambda^\ast/\Gamma_s|$ by \eqref{eq:det_is_l0}. Hence the total turning number is equal to $\TurnZero$. Moreover, the corresponding embedded torus is covered $\TurnZero$ times.

(3) Clearly $\TurnOne=0$ holds if and only if $s_1 = \pm s_2$. We may assume that $\lambda_0 = 1$ and then by \eqref{eq:t-vector} this holds if and only if $\lambda_2 = \lambda_1^{-1}$. By Proposition \ref{prop:torus-of-revolution} we conclude that $\TurnOne=0$ holds if and only if the torus is rotational.
\end{proof}


\begin{proposition} \label{prop:involution}
The family of non-rotational spectral genus 1 tori starting at the
$(\TurnZero,\,\TurnOne,\,\TurnTwo)$ flat \cmc torus ends at the
$(\TurnOne+\TurnTwo-\TurnZero,\,\TurnOne,\,\TurnTwo)$ flat \cmc torus.
\end{proposition}
\begin{proof}
Due to Proposition~\ref{prop:levelset}~\eqref{item:levelset2} the
difference $(\cos(2\theta_1)-\cos(2\theta_2))^2=4(1-\jk^2)(1-\jh^2)$
is uniformly bounded away from zero. After possibly swapping the
\sym points we may assume $\sign(\Xzero)\cos(2\theta_1)<\sign(\Xzero)\cos(2\theta_2)$ during the flow. We remark that $\sign(\Xzero)$ is constant
throughout the flow. In particular, $\lambda_1$ cannot pass through
$\lambda_1=\sign(\Xzero)$ and $\lambda_2$ cannot pass through
$\lambda_2=-\sign(\Xzero)$. Due to Proposition~\ref{prop:moduli-space} the function $\Xone-\sign(\Xzero)\Xtwo$ changes the sign, and the flow
passes through a root of this function. This implies that $\lambda_2$ has to pass through $\lambda_2=\sign(\jq)$. We want to describe how the final six integers in \eqref{eq:closing2} are related to
the corresponding initial six integers. Since $\OMEGA$ is multi-valued on the fixed point set $\bbS^1$ of $\varrho$, we cut this circle at $\lambda=\sign(\Xzero)$. At the end points of the flow with $\Xzero=\pm1$ this point is a double point. This is a good choice for the cut, since at the end points of the flow the \sym points cannot sit there. The multi-valued function $\OMEGA$
is single valued on $\bbS^1\setminus\{\sign(\Xzero)\}$.
The difference between the two boundary values of $\OMEGA$ on this
interval is equal to $\pm2$ due to Legendre's relation \cite[17.3.13]{AbrS}. With $p_{jk} =  x_j \, \NU_k + p_{j0}\,\OMEGA_k$ with respect to the periods $\gamma_j = x_j \pi + 2\mi p_{j0}\JK'$ as in \eqref{eq:closing2}, the final and initial integers are related by
\begin{equation*}
\left.\begin{pmatrix}
    p_{10} & p_{11} & p_{12}\\
    p_{20} & p_{21} & p_{22}
\end{pmatrix}\right|_{t=t_{\max}}=
\left.\begin{pmatrix}
    p_{10} & p_{11} & p_{12}\\
    p_{20} & p_{21} & p_{22}
\end{pmatrix}\right|_{t=t_{\min}}\pm 2
\begin{pmatrix}
    0 & 0 & p_{10}\\
    0 & 0 & p_{20}
\end{pmatrix} \spaceperiod
\end{equation*}
We remark that due to our choice these integers change their values
only, when $\lambda_2$ passes through $\lambda_2=\sign(\Xzero)$.
Now let $(\NU_1,\,\OMEGA_1)$ and $(\NU_2,\,\OMEGA_2)$ denote the values of $(\NU,\,\OMEGA)$ at the \sym points. After possibly independent
hyperelliptic involutions, we have $0<\NU_2<\NU_1$ in agreement with
our choice of the \sym points.

Note that $(s_0,\,s_1,\,s_2) = (p_{10},\,p_{11},\,p_{12}) \times (p_{20},\,p_{21},\,p_{22}) = (x_1 p_{20}-x_2 p_{10})\,(0,\,\NU_1,\,\NU_2) \times (1,\,\OMEGA_1,\,\OMEGA_2)$. Since the integers change only, when $\lambda_2$ passes through $\lambda_2=\sign(\Xzero)$, we can calculate the change of the $s_j$ in terms of the change of the values of $\NU$ and $\OMEGA$ at $\lambda_2=\sign(\Xzero)$.

At the beginning and end of the flow we have
$\left. (s_0,\,s_1,\,s_2)\right|_{t=t_{\min}} =  (0,\,\NU_1,\,\NU_2) \times (1,\,\OMEGA_1,\,\OMEGA_2)$ and $\left. (s_0,\,s_1,\,s_2) \right|_{t=t_{\max}} = (0,\,\NU_1,\,\NU_2) \times (1,\,\OMEGA_1,\,\OMEGA_2 \pm 2)$. Due to $0<\NU_2<\NU_1$ we have
\begin{equation}\label{eq:integers nu omega}
    (\TurnZero,\,\TurnOne,\,\TurnTwo) = |\,x_1 p_{20} - x_2 p_{10}\,|\,(|\NU_1\OMEGA_2-\NU_2\OMEGA_1|,\,
    \NU_1-\NU_2,\,\NU_1+\NU_2)\,.
\end{equation}
Therefore the final values are
$(|\TurnZero \pm(\TurnOne+\TurnTwo)|,\,\TurnOne,\,\TurnTwo)$ in terms of the initial values. The inequality $\TurnZero<\TurnTwo$ excludes the plus sign, and concludes the proof.
\end{proof}
\subsection{Formulae}
In a few places we will need formulae for equivariant \cmc tori and their profile curves in terms of the extended frame \eqref{eq:frame-parts}. Identifying the unit quaternions with $\bbi \coloneq \exp((\pi/2) e_0)$, $\bbj \coloneq \exp((\pi/2) e_1)$, $\bbk \coloneq \exp((\pi/2)e_2)$, and identifying $\bbi = \sqrt{-1}$, a computation gives
\begin{gather} \label{eq:twizzled-profile}
    f = \alpha_1\alpha_2^{-1}\beta_1\beta_2^{-1}
    (\gamma_1\gamma_2^{-1}c_1c_2+\gamma_1^{-1}\gamma_2 s_1s_2)
    + \alpha_1\alpha_2\beta_1\beta_2
    (\gamma_1^{-1}\gamma_2 s_1 c_2 - \gamma_1\gamma_2^{-1}c_1 s_2)\bbj \spacecomma\\
    \alpha = \exp(\mi x\Vconst) \spacecomma\interspace
    \beta = \exp( \tfrac{\mi}{2}\AngleZero) \spacecomma\interspace
    \gamma = \exp(\tfrac{\mi}{2}\AngleTwo) \spacecomma\\
    c = \cos(\half\AngleOne) \spacecomma\interspace
    s = \sin(\half\AngleOne) \spacecomma\\
    \alpha_k = \alpha(\lambda_k) \spacecomma\quad \beta_k =
    \beta(\lambda_k) \spacecomma\quad \gamma_k = \gamma(\lambda_k)
    \spacecomma\quad c_k = c(\lambda_k) \spacecomma\quad s_k =
    s(\lambda_k) \spacecomma \quad k=1,\,2 \spaceperiod
\end{gather}
With $\tau$~\eqref{eq:involtau}, we have $\tau^\ast\NU =
\NU$, $\tau^\ast\chi_0 = -\chi_0$, $\tau^\ast\chi_1 = \chi_1$,
$\tau^\ast\chi_2 = -\chi_2$. Applying these symmetries to the formula~\eqref{eq:twizzled-profile} for the immersion at $y=0$ shows the profile curve of an equivariant \cmc surface of revolution in $\mathbb{S}^3$ is $f_0 =\exp(\mi\AngleZero)(\cos\AngleTwo +
\mi\cos\AngleOne\sin\AngleTwo)+\mi\sin\AngleOne\sin\AngleTwo\bbj$. More explicitly,
\begin{gather} \label{eq:rev-profile-curve2}
    f_0 =\exp(\mi\AngleZero)(g_1+ \mi g_2) + g_{0}\bbk \spacecomma \\
    g_0 \coloneq\half\NU^{-1}\Metric\sin(2\theta_1)
    \spacecomma\interspace g_1
    \coloneq \cc^{-1}(\Metric\cos(2\theta_1) - \Metric^{-1}\jq)
    \spacecomma\interspace g_2\coloneq
    \half \cc^{-1}\NU^{-1}\Metric'\sin(2\theta_1)\spacecomma \\
    \cc^2 \coloneq \Metric^{-1}\Metric' -4\Vconst^2 =
    \Metric^2-2\Bpoint\cos(2\theta_1)+\Bpoint^2\Metric^{-2} \spaceperiod
\end{gather}

\subsection{Sphere bouquets}
\label{sec:sphere-bouquet}

%
\begin{figure}[t]
\centering
\includegraphics[height=5cm]{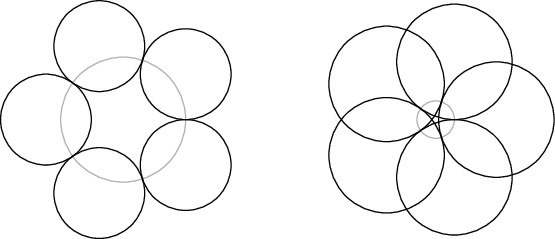}
\caption{
\label{fig:sphere}
The profile curves of the
two five-lobed sphere bouquets $(1,\,5)$ and $(2,\,5)$,
stereographically projected to $\bbE^2$.
The gray central circles are the axes of revolution.
}
\end{figure}
For $m,\,n\in\bbZ$ with $1 \le m < n$ and $\gcd(m,\,n)=1$, the
$(m,\,n)$ \emph{sphere bouquet} is constructed as follows.
With $\alpha=\exp(2\pi\mi m/n)\in\Cstar$, let $C$ be the union of the $n$ circles through $\alpha^k$ and $\alpha^{k+1}$ perpendicular to
$\bbS^1$ ($k=0,\dots,n-1$). Inverse stereographic projection of $C$
to a geodesic 2-sphere and revolution about the image of $\bbS^1$ as
axis produces the $(m,\,n)$ sphere bouquet, consisting of $n$
pairwise tangent spheres forming a necklace.
Since the radius of each circle is $r = \pi m/n$, the
mean curvature of the $(m,\,n)$ sphere bouquet is $\cot(\pi
m/n)$.
Since the $(m,\,n)$ sphere bouquet is the same as the $(n-m,\,m)$
sphere bouquet, then for $n>2$, the number of distinct sphere
bouquets with $n$ spheres is half the number of generators of
$\bbZ_n$. The $(1,\,2)$ sphere bouquet is a special case because its
two spheres coincide. In the rotational case we write $(\TurnZero,\,0,\,\TurnTwo)=(\TurnZero,\,\TurnTwo)$ for brevity.
%
%
\begin{lemma} \label{thm:sphere-bouquet}
The $(\TurnZero,\,\TurnTwo)$ family of rotational tori converges to
the $(\TurnZero,\,\TurnTwo)$ sphere bouquet as $\Bpoint\to 0$.
\end{lemma}
\begin{proof} For rotational tori we have $(\lambda_2,\,\NU_2,\,\OMEGA_2) = (\lambda_1^{-1},\,-\NU_1,\,-\OMEGA_1)$. By Proposition~\ref{prop:closing-conditions} there exists $(s_0,\,s_1,\,s_2)\in\bbZ^3$ that is perpendicular to $(0,\,\NU_1,\,\NU_2)$ and $(1,\,\OMEGA_1,\,\OMEGA_2)$. Then
$s_2=-s_1$ and $s_0+2s_1\OMEGA_1=0$. Hence $\TurnOne=0$ and $\TurnZero/\TurnTwo = \half\abs{s_0/s_1} = \abs{\OMEGA_1}$.

Define $\theta_1:\bbR \to \bbR$ by $2\mi\,\theta_1\coloneq \ln\lambda_1$. The limiting profile curve as $\jq\to 0$ can be computed from the profile curve for tori of revolution $f_0$ in~\eqref{eq:rev-profile-curve2}. Let $\jq(t)$, $\theta_1(t)$ vary with the flow parameter $t\in[t_{\min},\,t_{\max}]$, and assume without loss of generality that $\lim_{t\to t_{\min}}\jq(t)=0$.

From $\lim_{\jq\to 0}\OMEGA = 1- 2\theta/\pi$ we conclude that $\theta_0 = \lim_{t\to t_{\min}}\theta_1(t) =\tfrac{\pi}{2}(1 - \TurnZero/\TurnTwo)$.

With $\lim_{\jq\to 0}\Metric = \sech x$ and $\lim_{\jq\to 0}{2\Vconst}=1$ we have
\begin{equation*}
    \lim_{\Bpoint\to 0}g_0 = \sin(2\theta_0) \sech x
    \spacecomma\interspace
    \lim_{\Bpoint\to 0}g_1 = \cos(2\theta_0)
    \spacecomma\interspace
    \lim_{\Bpoint\to 0}g_2 = -\sin(2\theta_0) \tanh x
    \spaceperiod
\end{equation*}
Since the integrand in $\AngleZero$ goes to $0$ as $\jq\to 0$, then
$\lim_{\jq\to 0}\AngleZero = 0$. The limiting profile curve as $\jq\to 0$ is thus $\gamma_0(x) \coloneq \lim_{\jq\to 0}f_0 =
(\cos(2\theta_0)-\mi\sin(2\theta_0)\tanh x)+(\sin(2\theta_0)\sech x)\bbj$. Since $\cos(2\theta_0)-\mi\sin(2\theta_0)\tanh x$
traces out a straight line
segment in $\bbC$, then $\gamma_0(x)$ traces out a circle. The
discrete rotational symmetry of $\gamma$ implies that the limiting
set as $\jq\to 0$ is a sphere bouquet.

The limiting points are $\gamma_0(\pm\infty)=(e^{2\mp\mi\theta_0},\,0)$ on $\bbS^1$. The angle between the radii is
$4\theta_0 = 2\pi(1-\TurnZero/\TurnTwo)$.
Hence the limiting curve
is the $(\TurnTwo-\TurnZero,\,\TurnTwo)$ sphere bouquet, which is
the same as the $(\TurnZero,\,\TurnTwo)$ sphere bouquet. Note that
in the case $(\TurnZero,\,\TurnTwo)=(1,\,2)$ the circle is a
geodesic.
\end{proof}
We bring together Proposition~\ref{prop:moduli-space},
Proposition~\ref{prop:moduli-tori-of-revolution} and
Proposition~\ref{prop:involution} in the following Theorem.
\begin{theorem} \label{thm:moduli-space}
Spectral genus 1 \cmc tori lie in 1-parameter families with monotonic mean curvature. The family starting at the $(\TurnZero,\,\TurnOne,\,\TurnTwo)$ flat \cmc torus ends at the $(\TurnOne+\TurnTwo-\TurnZero,\,\TurnOne,\,\TurnTwo)$ flat \cmc torus.
\end{theorem}
\begin{proof}
By Propositions~\ref{prop:moduli-space} and \ref{prop:moduli-tori-of-revolution} the mean curvature is monotonic.
By Proposition~\ref{prop:moduli-space}, every flow starts and ends
at a flat \cmc torus with a double point on $\bbS^1$. The integers
associated to these two flat \cmc tori endpoints are computed next. As shown in Lemma~\ref{thm:sphere-bouquet}, the two flows ending at
the $(\TurnZero,\,0,\,\TurnTwo)$ and
$(\TurnTwo-\TurnZero,\,0,\,\TurnTwo)$ flat \cmc tori at $\Bpoint=1$ start at the same sphere bouquet at $\Bpoint=0$. Because only tori of revolution flow to sphere bouquets, we conclude that every sphere bouquet is the limit of these two flows and no others. While the flow is singular at $\Bpoint=0$, Proposition~\ref{prop:involution} nevertheless holds for the family constructed by gluing these two families together along the sphere bouquet.
\end{proof}


\section{Geometry}
\label{sec:torus}
\subsection{The torus knot and symmetry group}
\label{sec:symmetry}

Every orbit of the equivariant action~\eqref{eq:equi_action} with the exception of the two axes is a $(p,\,q)$-torus knot in the corresponding orbit of $\hat{\GK}$. Due to \eqref{eq:pq_nu} and \eqref{eq:integers nu omega} this implies that the orbit of a point on $(\TurnZero,\,\TurnOne,\,\TurnTwo)$ torus is generically a torus knot in the corresponding orbit of $\hat{\GK}$. If a $(\TurnZero,\,\TurnOne,\,\TurnTwo)$ torus does not meet the axes, then the linking numbers of the $\GK$ orbit of a
point on the torus and the two axes are $\left( \tfrac{\TurnOne}{\gcd(\TurnOne,\TurnTwo)},\,
    \tfrac{\TurnTwo}{\gcd(\TurnOne,\TurnTwo)}\right)$.
%

%
%
\begin{proposition} \label{thm:symmetry}
With $n\coloneq\gcd(\TurnOne,\,\TurnTwo)$, the
symmetry group of an $(\TurnZero,\,\TurnOne,\,\TurnTwo)$ cohomogeneity one \cmc torus is a semidirect product of $\bbS^1\times\bbZ_n$ and $\bbZ_2$ if it is twizzled, and a semidirect product of $\bbS^1\times\bbZ_n$ and $\bbZ_2\times\bbZ_2$ if it is a torus of revolution.
\end{proposition}
\begin{proof}
Let $T\coloneq (t_0,\,t_1,\,t_2)\in\bbZ^3$ with $\gcd(t_0,\,t_1,\,t_2)=1$ and
$t_0\ne 0$, and let $n\coloneq \gcd(t_1,\,t_2)$. We first show that
$\calZ\coloneq \{ n_0\in\bbZ \suchthat T\cdot(n_0,\,n_1,\,n_2) = 0
\text{ for some $n_1,\,n_2\in\bbZ$} \} = n\bbZ$.
If $T \cdot (n_0,\,n_1,\,n_2)=0$, then since $n\divides t_1$ and
$n\divides t_2$, then $n\divides(t_0n_0)$. Since
$\gcd(t_0,\,t_1,\,t_2)=1$, then $\gcd(t_0,\,n)=1$. Hence $n\divides
n_0$. Thus $n$ divides every element of $\calZ$. Since
$\gcd(t_1/n,\,t_2/n)=1$, by the Euclidean algorithm, there exist
$x_1,\,x_2\in\bbZ$ such that $n + t_1x_1+t_2x_2 = 0$. Hence with
$N=t_0(1,\,x_1,\,x_2)$, then $T\cdot N=0$. This shows that
$n\in\calZ$. Hence $\calZ = n\bbZ$.

There exists a basis $\gamma_1,\,\gamma_2\in\Cstar$ for the torus
lattice so that $\DOTR{\gamma_1}{\sql{0}}=0$ and $p_{20} \coloneq
\DOTR{\gamma_2}{\sql{0}}=\gcd(\TurnOne,\,\TurnTwo)$, where $\sql{0}=\mi$. Then $\gamma_1\in\bbR$ and $\gamma_2 = \pi x_2 + 2\mi p_{20}\JK'$ for some $x_2\in\bbR$.

By~\eqref{eq:cmc-equi-I-II}, the first fundamental form is preserved
if and only if $\Metric$ is preserved. The symmetry group thus
contains the three conformal automorphisms $z\mapsto z + t\gamma_1,\
t\in\bbR$, $z\mapsto z + \gamma_2/n_2$ and $z\mapsto -z$. For tori of revolution, the \sym points satisfy $\lambda_2 = \lambda_1^{-1}$ by
Proposition~\ref{prop:torus-of-revolution}. Hence the coefficient
$\half Q$ of the Hopf differential in~\eqref{eq:cmc-equi-I-II} is
real. Since the mean curvature $H$ is real, the second fundamental
form is preserved under complex conjugation. Hence in this case
there is a further anti-conformal automorphism $z\mapsto \ol{z}$.
\end{proof}

\subsection{Lobe counts}


The two lobe counts are the orders of the two orientation-preserving
cyclic subgroups of the symmetry group which fix one or the other
axis point wise.
\begin{proposition} \label{thm:profile-curve-symmetry}
The lobe counts of a twizzled $(\TurnZero,\,\TurnOne,\,\TurnTwo)$
torus are $\TurnOne$ and $\TurnTwo$, and for a rotational torus it
is $\TurnTwo$.
\end{proposition}
\begin{proof}
Let $\hat{\GK}$ be the two-dimensional torus \eqref{eq:commut}.
Let $G$ be a subgroup of the orientation preserving subgroup of the
isometry group of the torus which fixes
one axis of the equivariant action point wise. Such groups are
homeomorphic to $\mathbb{S}^1$, and thus closed. Closed subgroups of
$\mathbb{S}^1$ are either finite or all of $\mathbb{S}^1$, and since
we are not considering surfaces of revolution, the group $G$ is
finite, and thus cyclic. Let $H \subset G$ be a subgroup which fixes
every orbit of the equivariant action set wise. We compute
$\ord(G/H)$ and $\ord(H)$. From the proof of Proposition~\ref{thm:symmetry} we conclude that $n_2 = \ord(G/H) =
\gcd(\ell_1,\,\ell_2)$, since $H=G \cap \GK$. Now
$G\subset \{ (1,\,e^{is}) \mid s \in \bbR \}$ and
$\GK=\{(e^{\mi\ell_1 t},\,e^{\mi\ell_2 t}) \mid t \in \bbR\}$, and
$\ord(\{(1,\,e^{is}) \mid s \in \bbR \}\cap
\{ (e^{\mi\ell_1 t},\,e^{\mi\ell_2t}) \mid t \in \bbR \})$ is equal to $\ell_1/n$, and generated by $(1,\,e^{2\pi\mi n/\ell_1}) \in G$,
the order coincides with the order of $\GK \cap G$.
Hence $\ord(G) = \ord(G/H) \ord(G) = \ell_1$. Similarly for the other axis. This proves the claim for the twizzled case. For the rectangular case a similar argument holds, and concludes the proof.
\end{proof}

In view of Proposition~\ref{thm:profile-curve-symmetry} we call
$\TurnOne$ and $\TurnTwo$ respectively the \emph{minor and major
lobe counts}. The tori shown in Figure~\ref{fig:twizzled2} have
major and minor lobe counts $n$ and $1$ respectively. By Theorem~\ref{thm:moduli-space} and Theorem~\ref{thm:flat-torus-integers} we have
\begin{proposition}\label{th:minlobes}
The major lobe count of a non-rotational spectral genus 1 \cmc torus is at least 3, and that of a rotational spectral genus 1 \cmc torus of revolution is at least 2.
\end{proposition}

\subsection{Profile curve sets}

%
\begin{figure}[t]
\centering
{\includegraphics[width=5cm]{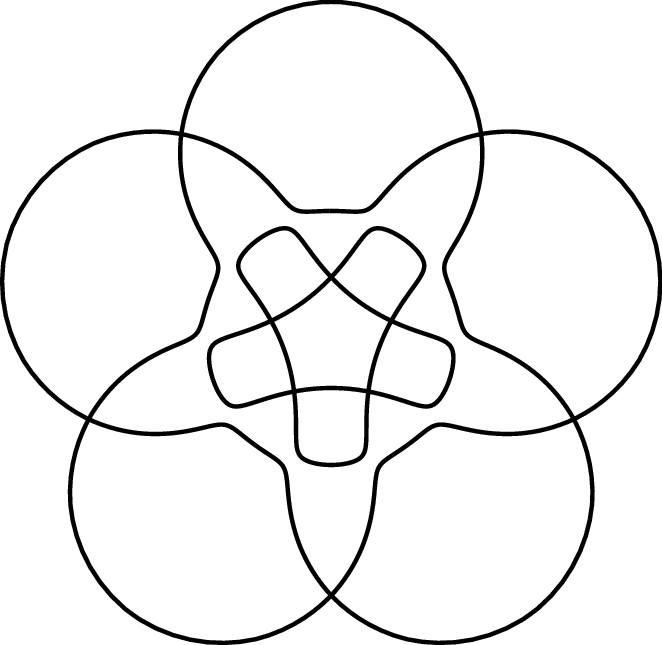}}\hspace{\PSPACE}
{\includegraphics[width=5cm]{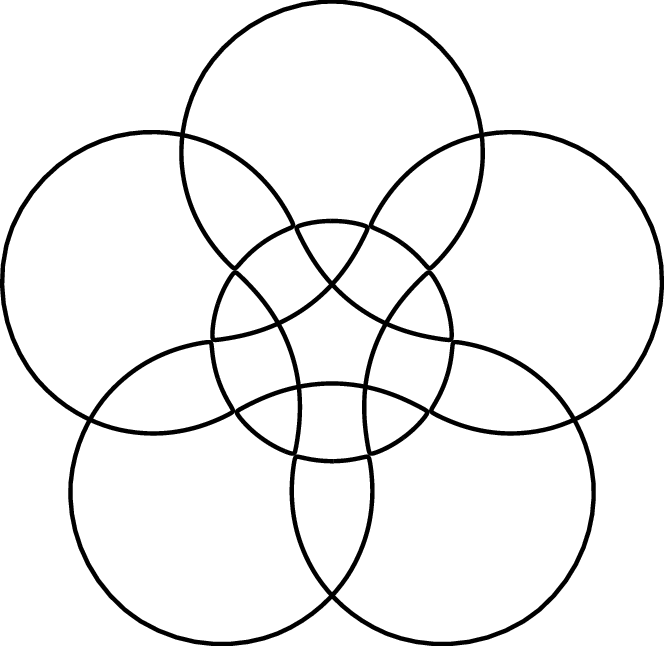}}\hspace{\PSPACE}
{\includegraphics[width=5cm]{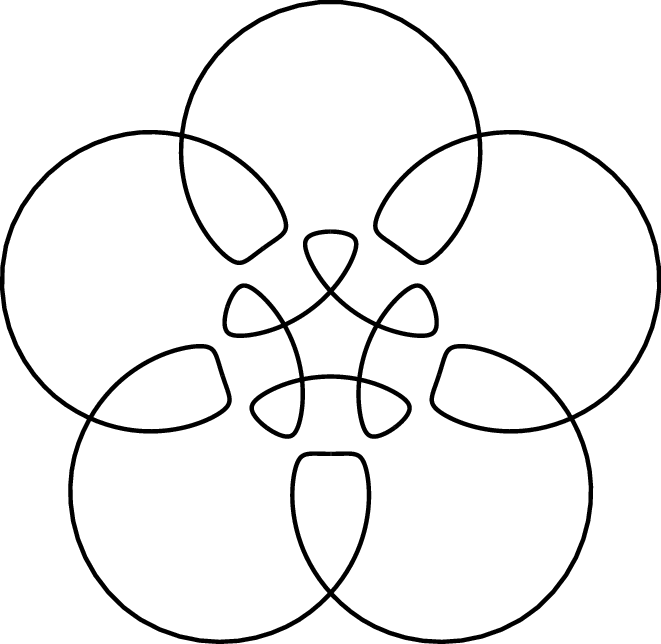}} \caption{
\label{fig:profile-twizzled} Profile curves of the
$(\TurnZero,\,\TurnOne,\,\TurnTwo) = (2,\,1,\,5)$ twizzled torus
family as the torus flows through its axis. The turning number of
the inner profile curve jumps from $\TurnZero=2$ to
$\TurnOne+\TurnTwo-\TurnZero=4$. Figure~\ref{fig:twizzled2} shows a
$5$-lobed torus in the family of which these are cross-sections. }
\end{figure}
The discrepancy between the two endpoints of the $g=1$ flow in
Proposition~\ref{prop:involution} is associated to the fact that at
two points during the flow, the corresponding torus intersects one
and then the other of its axes. At each of these two tori, one of
the torus knots degenerates to a circle. The combinatorics of the profile curve sets of equivariant tori are almost invariant during the flow: they are invariant on two disjoint intervals. When the torus intersects its axis, the connectivity and turning numbers of the profile curve set jumps as described in Lemma~\ref{thm:twizzled-profile-curve}. This phenomenon is
depicted in Figure~\ref{fig:profile-twizzled}.
\begin{lemma} \label{thm:twizzled-profile-curve}
If a profile curve set is immersed then $(\Xtwo\Bpoint -
\Xone)(\Xtwo - \Bpoint\Xone) \neq 0$. The total turning number of
each profile curve set of a non-flat twizzled
$(\TurnZero,\,\TurnOne,\,\TurnTwo)$ \cmc torus is $\TurnZero$ or
$\TurnOne+\TurnTwo-\TurnZero$.
\end{lemma}
\begin{proof}
Claim 1: If a profile curve set is not immersed, then $(\Xtwo\Bpoint - \Xone)(\Xtwo - \Bpoint\Xone) = 0$. To prove the claim, let $f$ be the immersion of the torus as in
\eqref{eq:twizzled-profile}. Writing $f=f_1+f_2\bbj$, the two
profile curve sets are defined implicitly by $\Real f_1=0$ and
$\Real f_2=0$ respectively. The profile curves are singular wherever
$\Real f_1$, ${(\Real f_1)}_x$ and ${(\Real f_1)}_y$ all vanish or
$\Real f_2$, ${(\Real f_2)}_x$ and ${(\Real f_2)}_y$ all vanish. The
function $f_k$ decouples into $f_k = \phi_k(x)\psi_k(y)$, where
$\gamma=\gamma_1\gamma_2^{-1}$ and $\phi_1 = \alpha_1\alpha_2^{-1}$ and $\psi_1 = \beta_1\beta_2^{-1} (\gamma c_1 c_2+\gamma^{-1}s_1s_2)$, and
$\phi_2 = \alpha_1\alpha_2$ and $\psi_2 = \beta_1 \beta_2
(\gamma^{-1} s_1 c_2 - \gamma c_1s_2)$. Then
\begin{equation*}
    2\Real f_k = \phi_k\psi_k + \ol\phi_k\ol\psi_k
    \spacecomma\interspace 2{(\Real f_k)}_x = {(\phi_k)}_x\psi_k +
    \ol{(\phi_k)}_x \ol\psi_k \spacecomma\interspace k=1,\,2
    \spaceperiod
\end{equation*}
For $k=1,\,2$, Since ${(\phi_k)}_x$ never vanishes, it follows that
$\Real f_k$ and ${(\Real f_k)}_x$ vanish if and only if $\psi_k$
vanishes. The additional condition that ${(\Real f_k)}_y$ vanishes
is ignored; it specifies for which values of $x$, if any, the curve
fails to be immersed. Since $\beta_1$ and $\beta_2$ are unimodular
and $c_j$ and $s_j$ are real, this occurs if $\gamma^4=1$ and either
$c_1^2c_2^2 - s_1^2 s_2^2=0$ or $s_1^2c_2^2 - c_1^2 s_2^2=0$.
We have
\begin{align*}
    2(c_1^2c_2^2 - s_1^2 s_2^2) &= \cos(\AngleOne(\lambda_2)) +
    \cos(\AngleOne(\lambda_1)) =
    \Metric^{-1}\Metric'(\half
                        (\Vconst^{-1}(\lambda_2)+\Vconst^{-1}(\lambda_1)))
    \spacecomma \\
    2(s_1^2c_2^2 - c_1^2 s_2^2) &= \cos(\AngleOne(\lambda_2)) -
    \cos(\AngleOne(\lambda_1)) =
    \Metric^{-1}\Metric'(\half(
                        \Vconst^{-1}(\lambda_2)-\Vconst^{-1}(\lambda_1)))
    \spaceperiod
\end{align*}
The zero sets of each of these expressions is the zero set of
$\Metric'$. At the zeros of $\Metric'$, $\Metric=1$ or
$\Metric=\Bpoint$.
A computation shows that the zero set of $\gamma^4-1$ is the zero
set of $\Xtwo\Metric^2-\Xone\Bpoint$. Hence if the curve is not
immersed, then either $\Metric=\jq$ and $\Xtwo\Bpoint=\Xone$, or else
$\Metric=1$ and $\Xtwo=\Xone\Bpoint$. This proves claim 1.

Let
\begin{gather*}
    \calI_1=\{\Xone-\Xzero\Xtwo<0\}
    \spacecomma\quad \widehat\calI_1=\{\Xone-\Xzero\Xtwo>0\}
    \spacecomma\quad \calI_2=\{\Xtwo-\Xzero\Xone>0\} \spacecomma\quad
    \widehat\calI_2=\{\Xtwo-\Xzero\Xone<0\} \spacecomma\\
    (\tilde\TurnZero,\,\tilde\TurnOne,\,\tilde\TurnTwo) =
    (\TurnOne+\TurnTwo-\TurnZero,\,\TurnOne,\,\TurnTwo)
    \spacecomma\interspace c_{jk} = \gcd(\Turn_j,\,\Turn_k) \spacecomma
    \interspace \hat c_{jk} = \gcd(\hat\Turn_j,\,\hat\Turn_k)
    \spaceperiod
\end{gather*}
Claim 2: on $\calI_k$ (respectively $\widehat\calI_k$), $\calC_k$ has $c_k$ (respectively $\hat c_k$) components. Each component of $\calC_k$ has turning number $\TurnZero/c_k$ (respectively $\hat\TurnZero/c_k$). The total turning number of $\calC_k$ is $\TurnZero$ on each of $\calI_k$ and $\widehat\calI_k$. By Theorem~\ref{thm:flat-torus-integers} (2), at the flat \cmc torus at the beginning (respectively end) of the flow, the profile curve sets are $\TurnZero$ (respectively $\TurnOne+\TurnTwo-\TurnZero$) wrapped circles. Since turning numbers of the immersed profile curves are homotopy invariants, the total turning number of $\calC_1$ is preserved in $\calI_1$ and $\widehat\calI_1$. Similarly the total turning number of $\calC_2$ is preserved in $\calI_2$ and $\widehat\calI_2$.
\end{proof}
Lemma~\ref{thm:twizzled-profile-curve} simplifies in the case of
tori of revolution. By Theorem~\ref{thm:flat-torus-integers}, the profile curve at the flat \cmc torus is an $\TurnZero$-wrapped circle, with turning number $\TurnZero$. Since the flow
$\Xzero\in(0,\,1]$\ induces a regular homotopy of the profile curve,
then by the Whitney-Graustein theorem, every profile curve in the
flow has turning number $\TurnZero$. Figure~\ref{fig:profile-3lobe}
illustrates the profile curves of the $3$-lobed tori of revolution.
%
\begin{figure}[t]
\centering
\frame{\includegraphics[width=3.35cm]{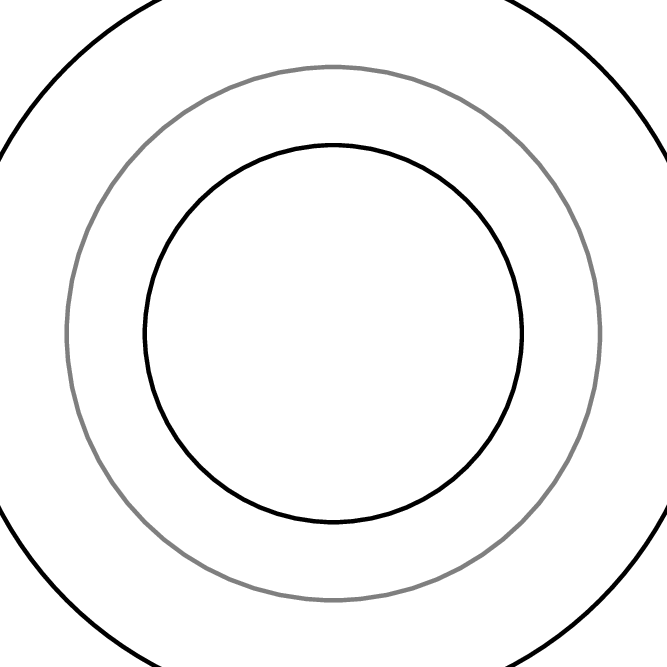}}
\frame{\includegraphics[width=3.35cm]{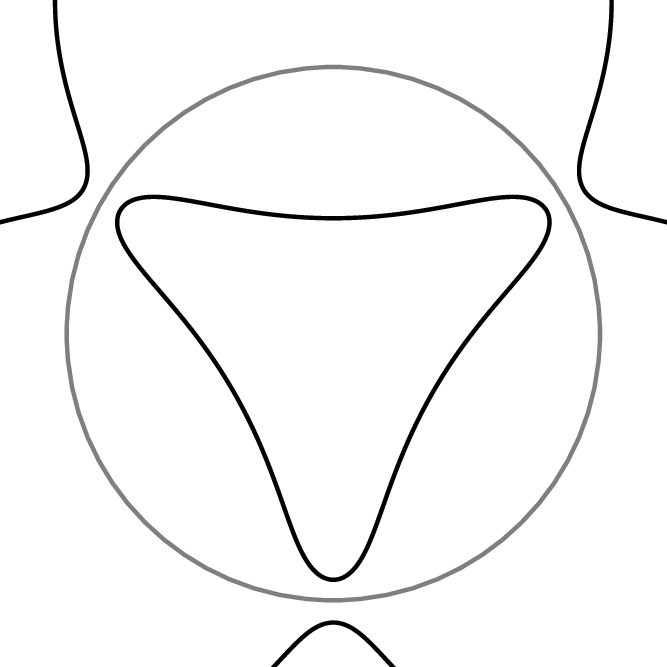}}
\frame{\includegraphics[width=3.7cm]{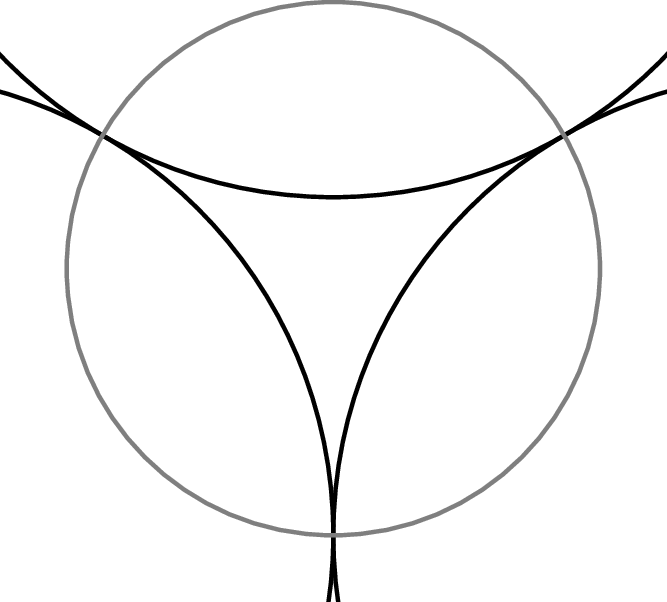}}
\frame{\includegraphics[width=3.7cm]{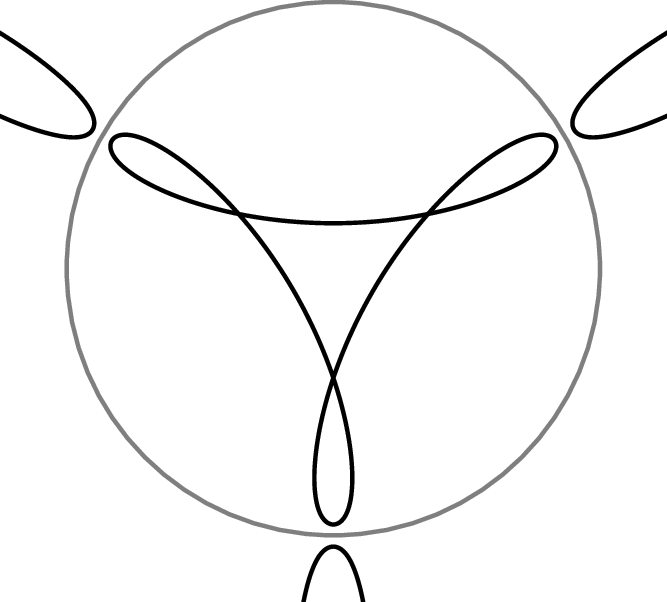}}
\frame{\includegraphics[width=3.45cm]{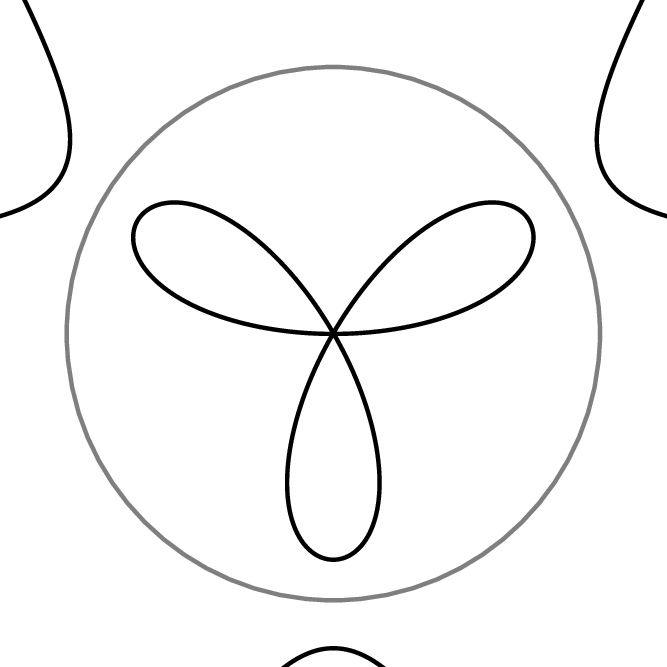}}
\frame{\includegraphics[width=3.78cm]{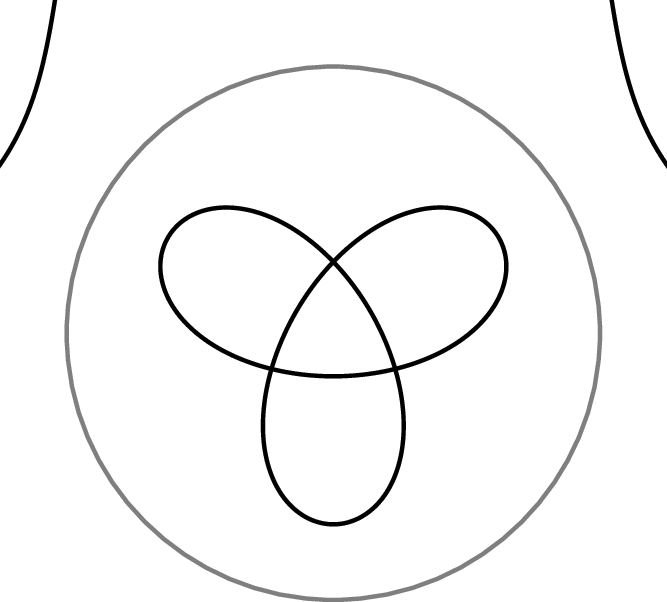}}
\frame{\includegraphics[width=3.45cm]{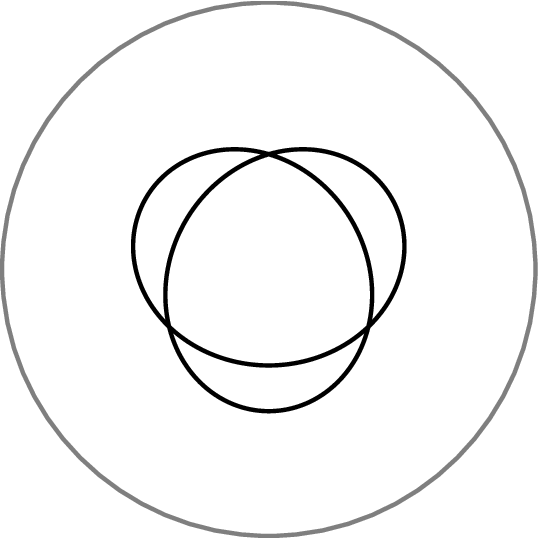}}
\frame{\includegraphics[width=3.45cm]{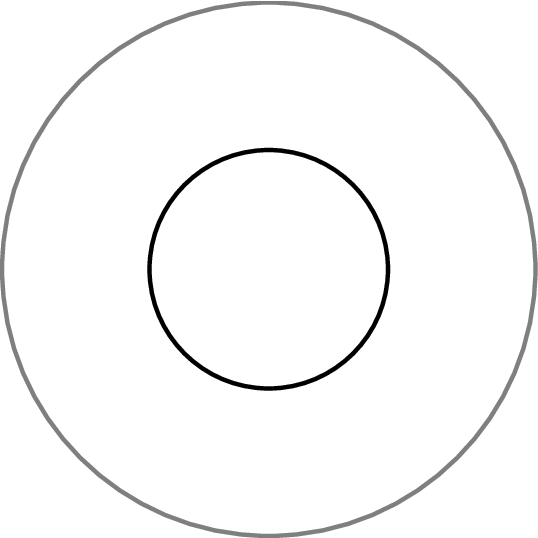}}
\caption{ \label{fig:profile-3lobe} The flow through the $3$-lobed
tori of revolution. Starting at a singly-wrapped flat \cmc torus, the inner profile curve of the embedded $(1,\,3)$ tori flows to the
$(1,\,3)$ sphere bouquet in the third frame. It continues through
the non-embedded $(2,\,3)$ tori, crosses itself in the fifth frame,
and ends at a doubly-wrapped flat \cmc torus. As the inner profile curve passes through the origin, and the outer curve through infinity, their turning numbers remain fixed but their winding numbers jump. The curves are stereographically projected to $\bbE^2$; the central gray circle is the axis of revolution. }
\end{figure}
%

\subsection{Embeddedness}

We show that twizzled \cmc tori are never embedded, and classify embedded \cmc tori of revolution. As a corollary of
Lemma~\ref{thm:twizzled-profile-curve} we have
\begin{corollary} \label{cor:twizzled-nonembedded}
A non-rotational spectral genus one \cmc torus in $\bbS^3$ is never embedded.
\end{corollary}
\begin{proof}
By Lemma~\ref{thm:twizzled-profile-curve}, the profile curve sets
of a $(\TurnZero,\,\TurnOne,\,\TurnTwo)$ twizzled \cmc torus have
total turning number $\TurnZero$ or $\TurnOne+\TurnTwo-\TurnZero$.
By Theorem~\ref{thm:flat-torus-integers}, each of these turning
numbers is strictly bigger than $1$. Hence the profile curve sets are not embedded. To show the surface is not embedded assume first that $(\Xtwo\Bpoint - \Xone)(\Xtwo - \Bpoint\Xone) \neq 0$. If $f$ where embedded, then by Lemma~\ref{thm:twizzled-profile-curve} the profile curve is immersed. By the inverse function theorem the inverse image under $f$ of the profile curve is embedded, hence if it where embedded the profile curve would be embedded, which is not true since its turning number is at least two by Lemma~\ref{thm:twizzled-profile-curve}, giving the contradiction. Because embeddedness is an open condition, during the flow $(\Xtwo\Bpoint - \Xone)(\Xtwo - \Bpoint\Xone) \neq 0$ away from isolated points, the surface is embedded also at the zeroes.
\end{proof}
%

%
\begin{theorem} \label{thm:embedded} A $(\TurnZero,\,\TurnTwo)$
torus of revolution is embedded if and only if $\TurnZero=1$.
\end{theorem}
\begin{proof}
A surface of revolution in $\bbS^3$ is embedded if and only if its
profile curve is embedded and does not meet the revolution axis. To
show the embeddedness, we show that the curvature of the orthographic projection of the profile curve $f_0 =\exp(\mi\AngleZero)(g_1 + \mi g_2) + g_{0}\bbk$ in \eqref{eq:rev-profile-curve2} is strictly positive. Write $\exp(\mi\chi_0)(g_1+\mi g_2) = r\exp(\mi\psi)$ and $s=g_0$ so the profile curve is $f_0 = re^{\mi\psi} + s\bbk$. To compute $\psi'$, note that $\psi = \AngleZero +
\arg(g_1+\mi g_2)$. The expression for $\AngleZero'$
in~\eqref{eq:angle-deriv} yields after a calculation
\begin{equation} \label{eq:theta-prime}
    \psi' = \AngleZero' +
    \frac{g_1g_2'-g_1'g_2}{g_1^2+g_2^2} =
    \frac{2\Vconst \sin(2\theta_1) 
    (\Metric^2\cos(2\theta_1) - \jq)} 
    {\Metric^2\sin^2(2\theta_1) -4\Vconst^2} \spaceperiod 
\end{equation}
with $2\mi\theta_1 = \ln\lambda_1$.
The curvature of the plane curve $re^{\mi\psi}$ is $\kappa =8
\cc^{-3}\Vconst^2 \jq\Metric^{-1} $. Note that the plane curve
$re^{\mi\psi}$ is an orthographic projection
of the hemisphere to $\bbR^2$, not stereographic; the curvature of
the stereographic projection may change sign, as seen in
Figure~\ref{fig:profile-3lobe}.

We next show that the profile curve does not meet the revolution
axis for $\Bpoint\in(0,\,1]$. Since the range of $\Metric$ is
$[\Bpoint,\,1]$, then the range of $s$ is
$[\half\Vconst^{-1}\Bpoint\sin(2\theta_1),\,
\half\Vconst^{-1}\sin(2\theta_1)]$. Hence
$s>0$, because $\Vconst>0$, $\Bpoint>0$ and $\sin(2\theta_1)>0$.
Hence $\abs{r}<1$, so the profile curve does not meet the axis of
revolution.

If the profile curve is embedded, then its turning number is $1$.
But by the discussion after
Theorem~\ref{thm:twizzled-profile-curve}, its turning number is
$\TurnZero$. Hence $\TurnZero=1$.

Conversely, assume $\TurnZero=1$, so its turning number is $1$. The
curvature $\kappa$ of the orthographic projection $re^{\mi\psi}$ of
the profile curve to $\bbC$ computed above is strictly positive, so
the profile curve is convex and hence embedded (see e.g \cite[5-7, Proposition 1]{DoC:dg1}).
\end{proof}
\begin{theorem}
An equivariant {\sc{cmc}} torus in $\bbS^3$ is Alexandrov embedded
if and only if it is a surface of revolution and singly wrapped with respect to the rotational period.
\end{theorem}
\begin{proof}
A rotational torus is Alexandrov embedded if and only if there
exists an immersion $\bbS^1 \times [0,\,1] \to \bbS^2_+$ into a
hemisphere $\bbS^2_+$ such that $\bbS^1 \times \{0\}$ is mapped to
the equator of $\bbS^2_+$ and $\bbS^1 \times \{1 \}$ is mapped onto
a profile curve of the torus. The resulting 3-manifold, obtained by
rotating the strip is then the Alexandrov embedding.

Every flat \cmc torus is a covering of an embedded flat \cmc torus. Hence the 3-manifold is always a solid torus, and thus has fundamental group $\bbZ$. The compact coverings correspond to proper subgroups of $\bbZ$. Hence flat Alexandrov embedded \cmc tori have to be singly wrapped. This condition is stable under continuous deformations which stay away from bouquets of spheres.
\end{proof}
\begin{figure}[t]
\centering
\includegraphics[width=4cm]{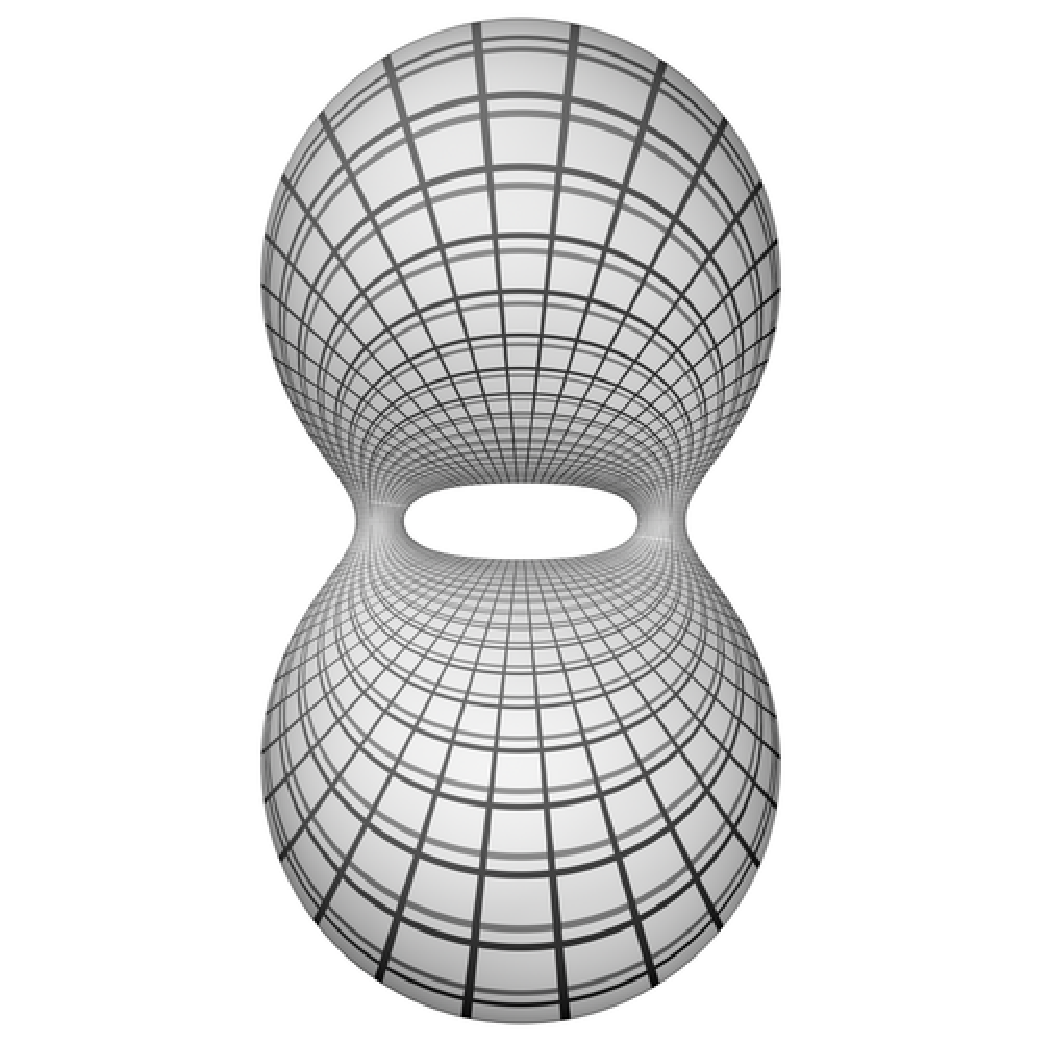}
\includegraphics[width=4cm]{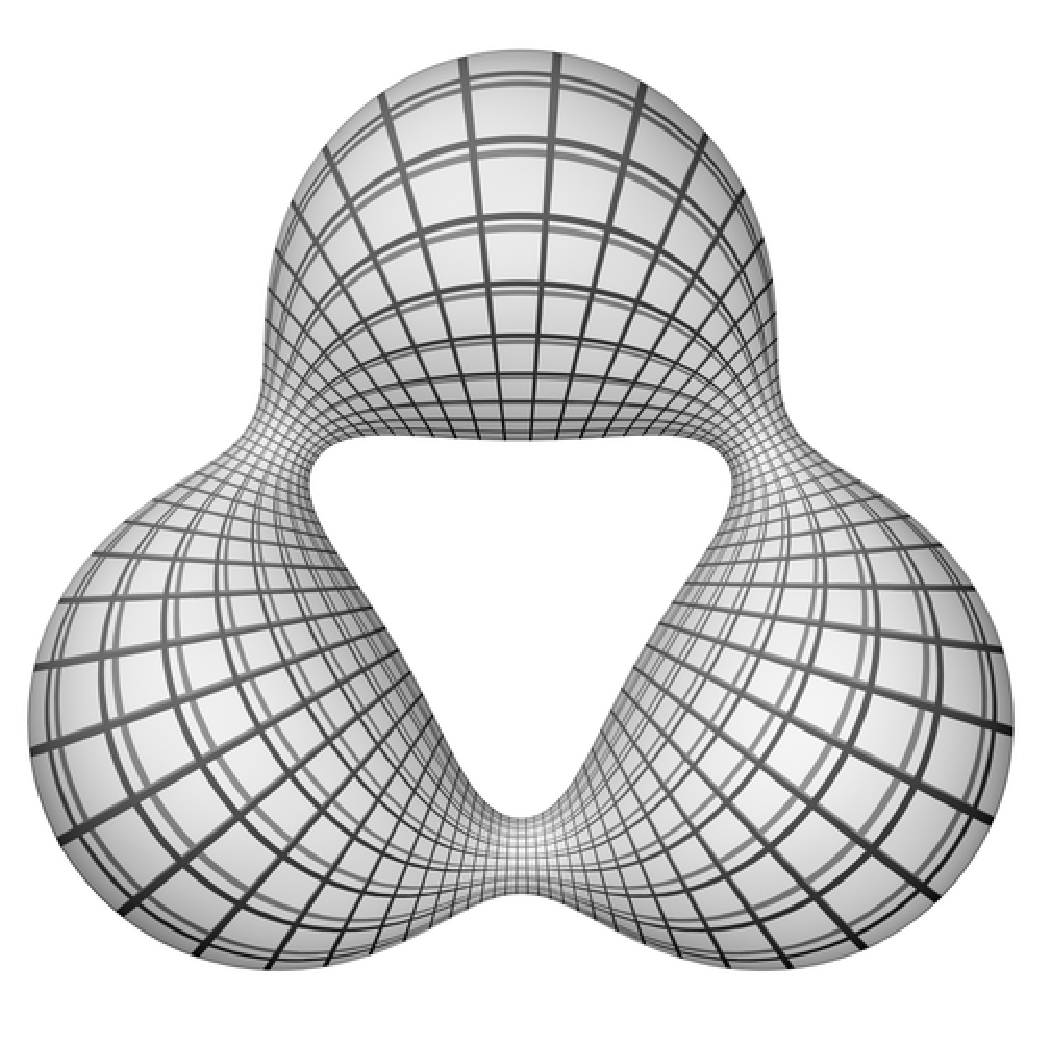}
\includegraphics[width=4cm]{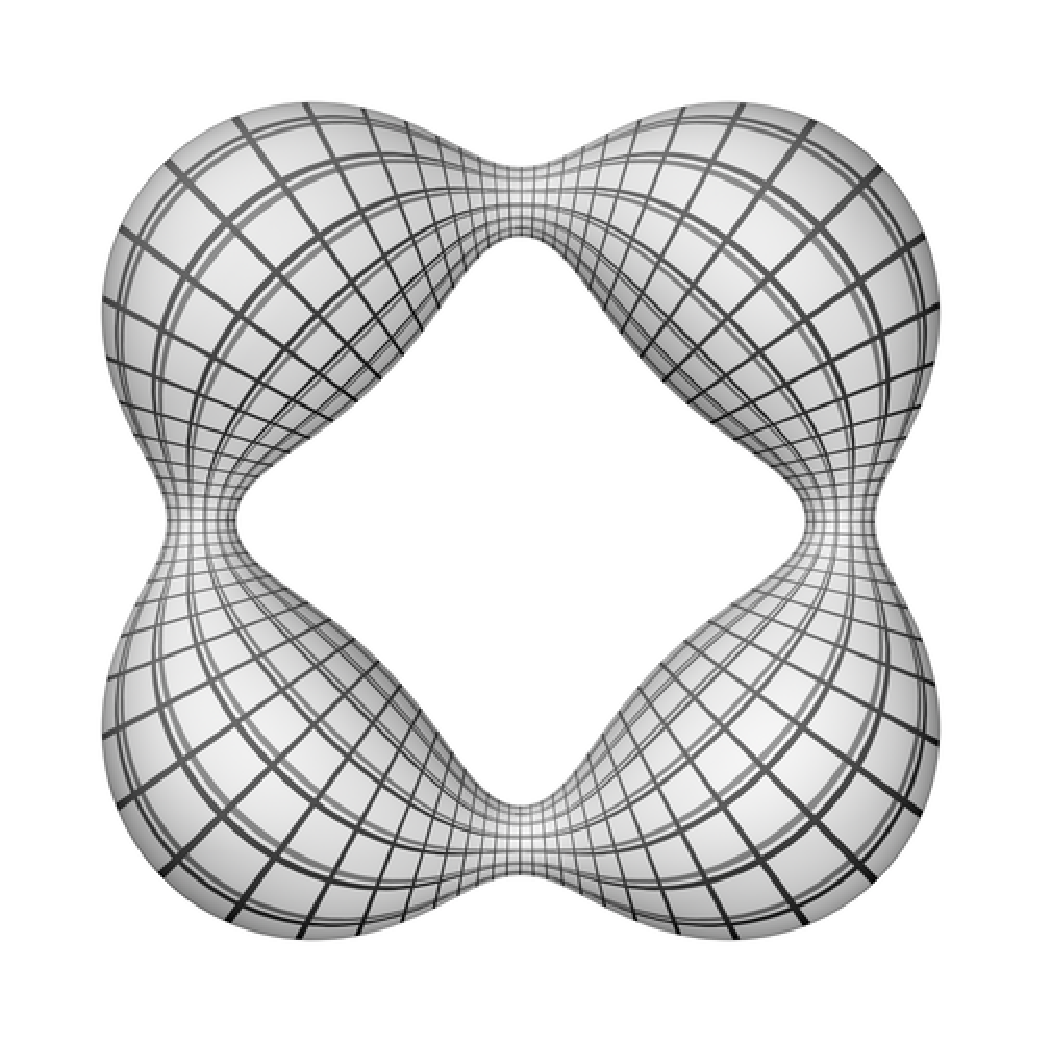}
\includegraphics[width=4cm]{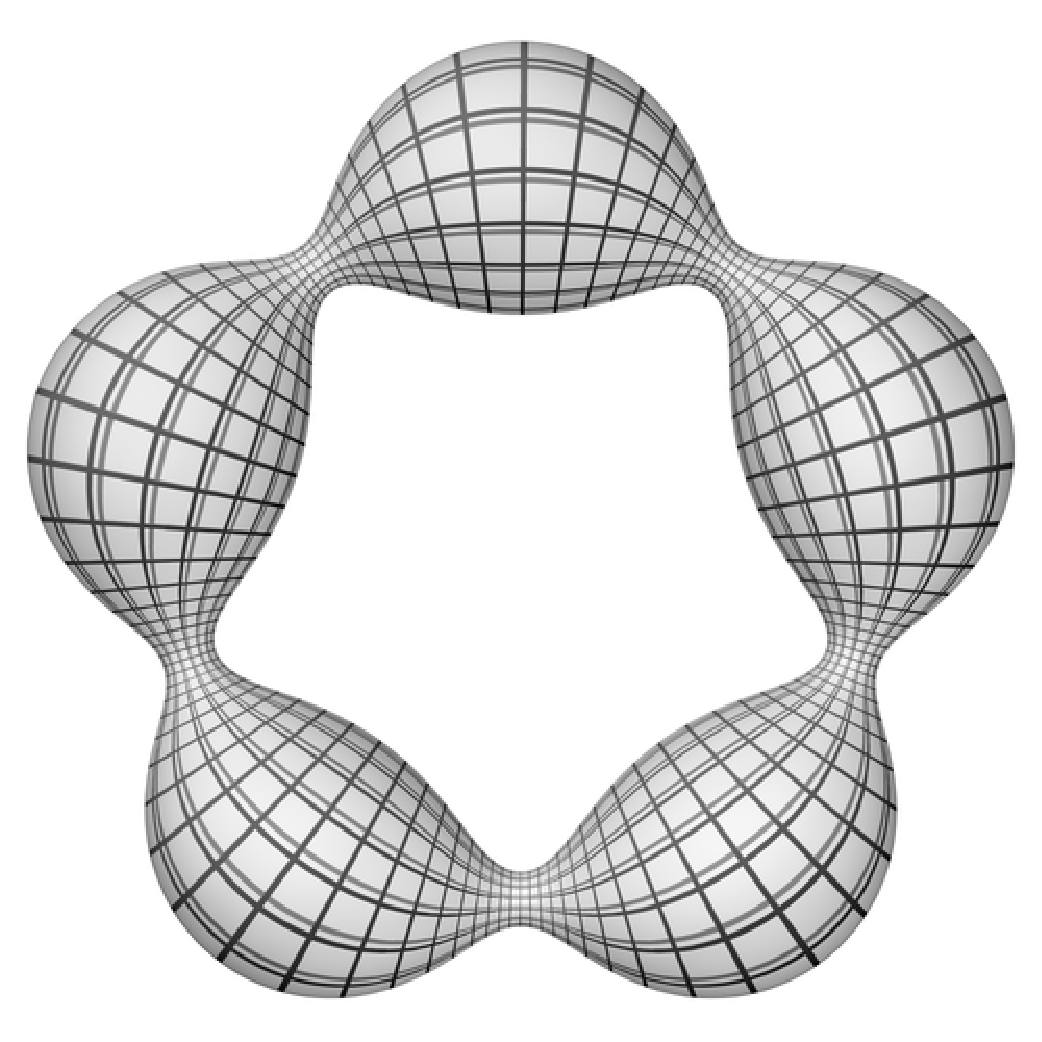}
\caption{ \label{fig:revtorus} Embedded $(1,\,n)$ \cmc tori of
revolution in $\bbS^3$, with $n=2,\,3,\,4,\,5$. }
\end{figure}
\subsection{Mean curvature and minimal tori}

There are infinitely many minimal tori in $\bbS^3$ \cite{Car:tor}. There are in fact already infinitely many minimal equivariant ones \cite{HsiL}. Theorem~\ref{thm:mean-curvature} shows the existence of infinitely many minimal twizzled tori. For example the
$(\TurnZero,\,\TurnOne,\,\TurnTwo)=(n-k,\,n,\,n+k)$ flow family with
$0<k<n$ is a fixed point of the involution of
Proposition~\ref{prop:involution}. The flow starts and ends at the
same flat \cmc torus with opposite mean curvature, and hence it contains a minimal torus. A (non-minimal) example from the
$(\TurnZero,\,\TurnOne,\,\TurnTwo)=(2,\,1,\,3)$ family is shown in
Figure~\ref{fig:twizzled2}.
%
%
\begin{lemma} \label{thm:mean-curvature}
A spectral genus 1 flow family with endpoints $(\TurnZero,\,\TurnOne,\,\TurnTwo)$ and $(\TurnOne+\TurnTwo-\TurnZero,\,\TurnOne,\,\TurnTwo)$ contains exactly one minimal torus if $(\TurnOne^2+\TurnTwo^2)^{1/2} \ge \sqrt{2} \max\{\TurnZero,\,\TurnOne+\TurnTwo-\TurnZero\}$ and no minimal tori otherwise.
\end{lemma}
\begin{proof}
Consider the flow from a flat \cmc torus to a flat torus through spectral genus 1 tori as described in Theorem~\ref{thm:moduli-space}. Since the mean curvature \eqref{eq:H_and_h} is monotonic, the flow contains a minimal torus if and only if the mean curvature of the flat \cmc tori have opposite signs, or if one of these flat \cmc tori is minimal.

By a calculation, $\TurnOne = \TurnZero\min\{X,\,Y\}$, $\TurnTwo =
\TurnZero\max\{X,\,Y\}$, where $X = \sqrt{(1+\Xone)/(1-\Xtwo)}$ and
$Y = \sqrt{(1-\Xone)/(1+\Xtwo)}$, and then
\begin{equation} \label{eq:mean-curvature-with-sign}
    H = \sign(\Xone+\Xtwo)\tfrac{\TurnOne^2+\TurnTwo^2-2\TurnZero^2}%
    {2\sqrt{(\TurnTwo^2-\TurnZero^2)(\TurnZero^2-\TurnOne^2)}}
    \spaceperiod
\end{equation}
Since $\Xone+\Xtwo$ is negative at the beginning of the flow and
positive at the end (Theorem~\ref{thm:moduli-space}), by
Proposition~\ref{prop:involution}, mean curvatures $H_0$ and $H_1$
of the flat \cmc tori at the beginning and end of the flow are
respectively
\begin{equation} \label{eq:mean-curvature-interval}
    H_0 \coloneq \tfrac{\TurnOne^2+\TurnTwo^2-2\TurnZero^2}
    {2\sqrt{(\TurnTwo^2-\TurnZero^2)(\TurnZero^2-\TurnOne^2)}}
    \spacecomma\interspace
    H_1 \coloneq -\tfrac{\TurnOne^2+\TurnTwo^2-2\hat\TurnZero^2}
    {2\sqrt{(\TurnTwo^2-\hat\TurnZero^2)(\hat\TurnZero^2-\TurnOne^2)}}
\end{equation}
where $\hat\TurnZero\coloneq \TurnOne+\TurnTwo-\TurnZero$.
The flow contains a minimal torus if and only if
$\TurnOne^2+\TurnTwo^2-2\TurnZero^2$ and
$\TurnOne^2+\TurnTwo^2-2(\TurnOne+\TurnTwo-\TurnZero)^2$
have the same sign, or either is $0$. Since the sum of these two integers
is equal to $4(\TurnTwo-\TurnZero)(\TurnZero-\TurnOne)>0$, they
are not both negative. The condition that they
have the same sign, or either is $0$, is then equivalent to the
asserted inequality.
\end{proof}
Consider the $(\TurnZero,\,\TurnTwo)$ family of tori of revolution,
and let $\Ratio \coloneq \TurnZero/\TurnTwo$. The mean curvature for
the flat \cmc tori is chosen to be positive for
$\Ratio\in(0,\,1/\sqrt{2})$ and negative for
$\Ratio\in(1/\sqrt{2},\,1)$.

%
\begin{lemma} \label{thm:rev-mean-curvature}
For spectral genus 1 \cmc tori of revolution, the mean curvature decreases monotonically from $H_0 = (1-2\Ratio^2)/(2\Ratio\sqrt{1-\Ratio^2})$
at the flat \cmc torus to $H_s = \cot\pi\Ratio$ at the sphere bouquet. This family contains exactly one minimal torus if
$\Ratio\in(\frac{1}{2},\,1/\sqrt{2}]$, and no minimal tori otherwise.
\end{lemma}
\begin{proof}
By~\eqref{eq:mean-curvature-interval}, the mean curvature of the
flat \cmc torus at the end of the flow ($\jq=\pm1$) is $H_0$ as in the assertion. By Lemma~\ref{thm:sphere-bouquet}, the $(\TurnZero,\,\TurnTwo)$ family of tori of revolution converges to the $(\TurnZero,\,\TurnTwo)$ sphere bouquet as $\jq\to 0$. The limiting sphere bouquet has mean curvature $H_s = \cot\pi\Ratio$. The mean curvature is monotonic by Theorem~\ref{thm:moduli-space}
and hence has the specified range. The family contains a minimal torus if and only if the mean curvature has different signs at the endpoints of the flow. This occurs if and only if $\Ratio\in (\frac{1}{2},\,1/\sqrt{2}]$.
\end{proof}
\begin{figure}[t]
\centering
\includegraphics[width=5.4cm]{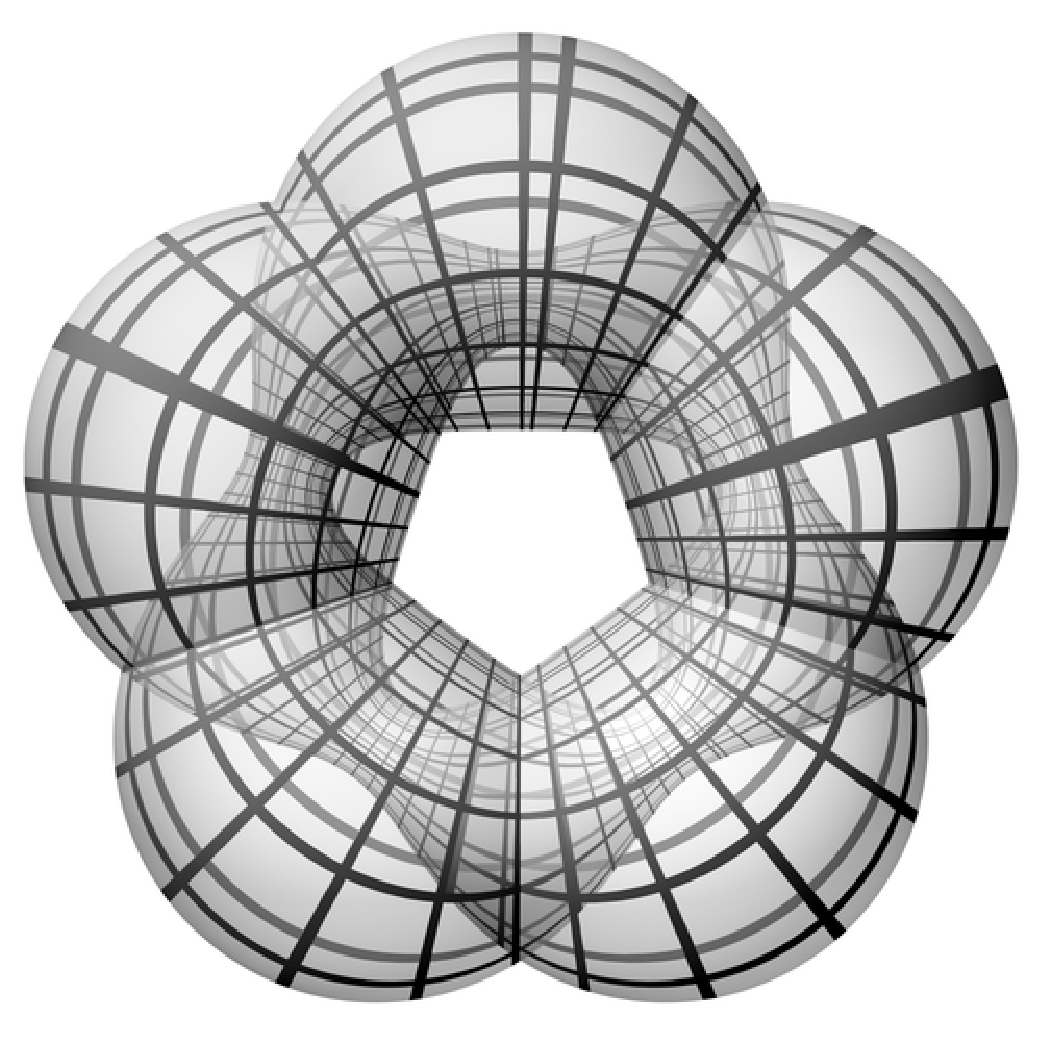}
\includegraphics[width=5.4cm]{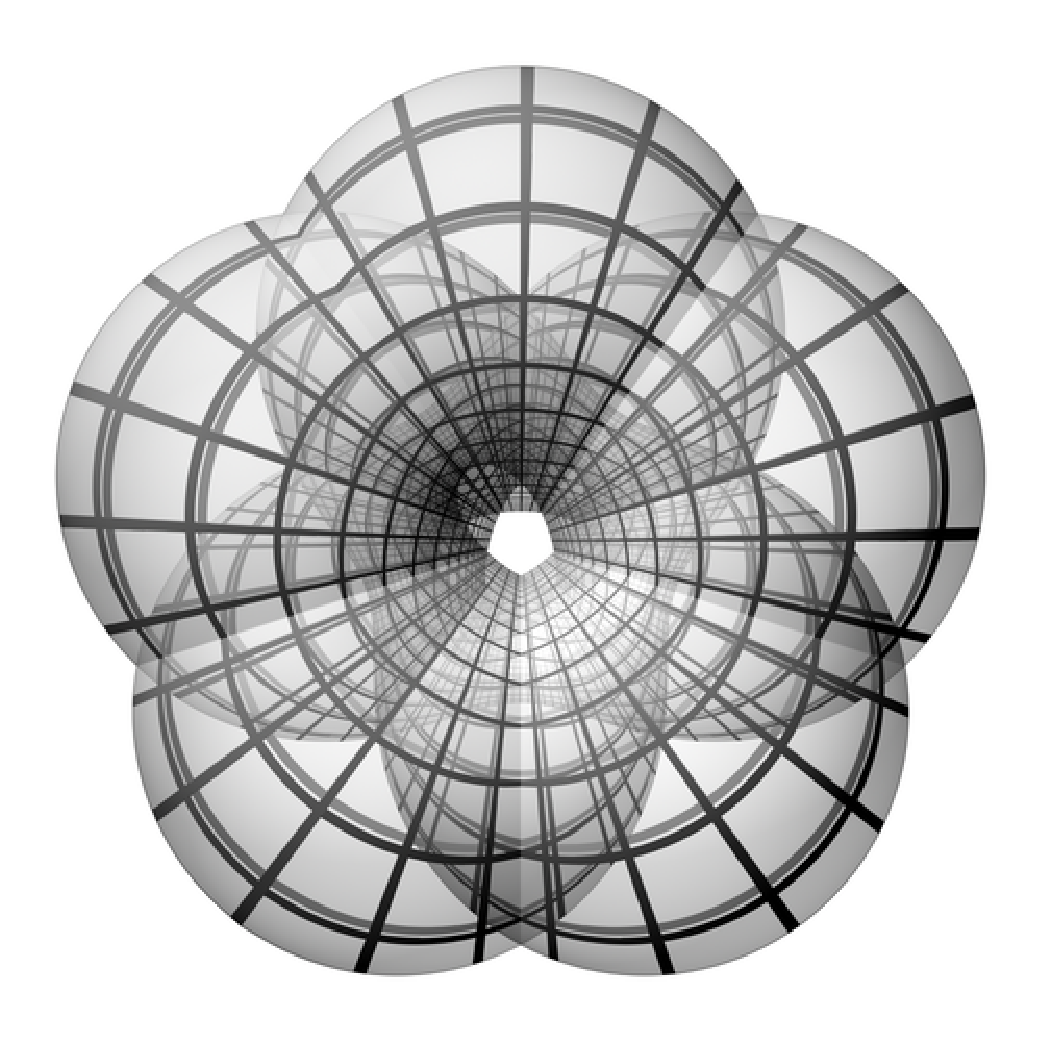}
\includegraphics[width=5.4cm]{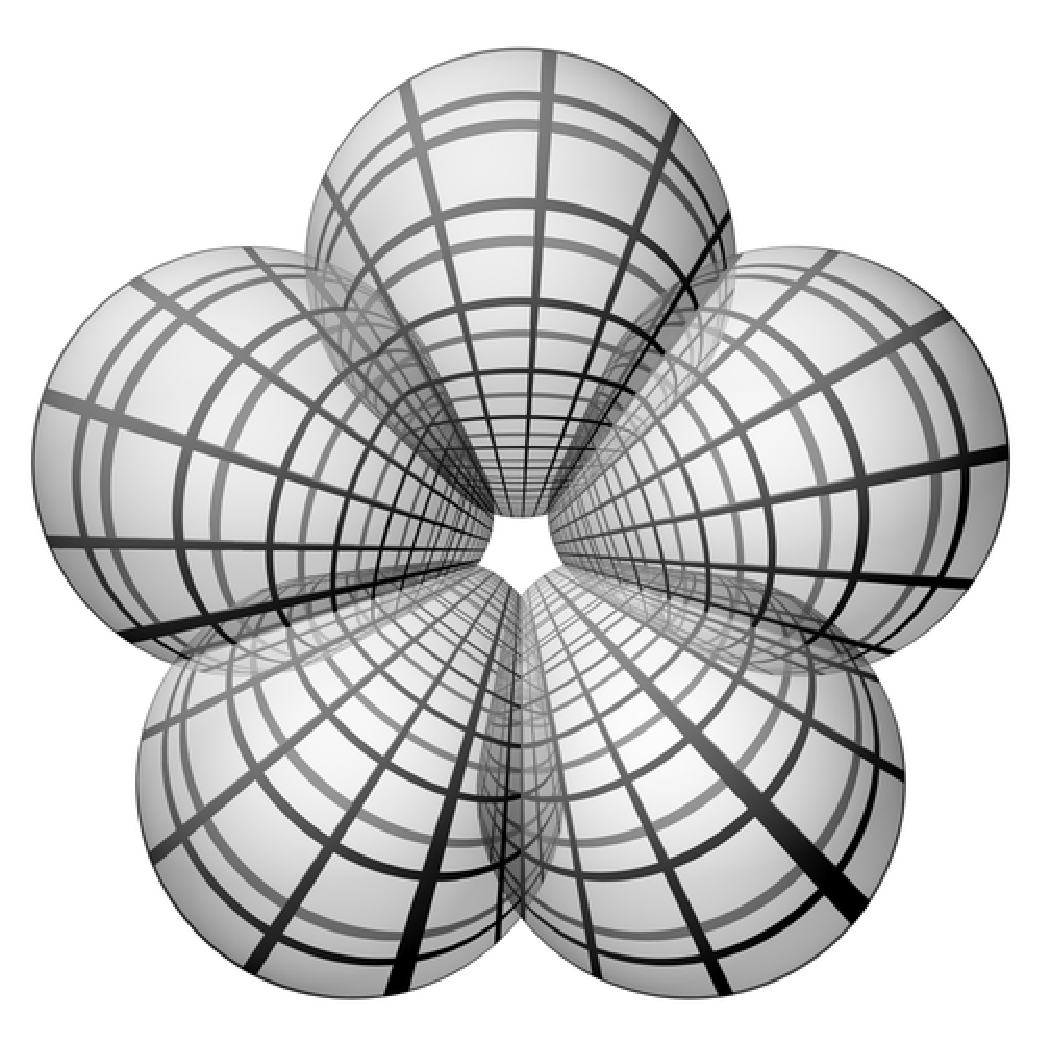}
\caption{ \label{fig:revtorusfive} Alexandrov embedded five-lobed $(k,\,5)$ \cmc tori of revolution in $\bbS^3$. The turning number of the inner profile curve is $k = 2,\,3,\,4$. }
\end{figure}
\begin{corollary}
The Clifford torus is the only minimal embedded rotational torus in the 3-sphere.
\end{corollary}
By combining the above results we have shown that amongst the infinitely many minimal equivariant tori in the 3-sphere, only one is embedded \cite{HsiL}.
\begin{theorem}
The Clifford torus is the only embedded minimal equivariant torus in $\bbS^3$.
\end{theorem}
%

\section{Connectedness of the moduli space}

Let $M$ denote the set of \cmc immersions from the oriented 2-torus $\bbT^2$ into the oriented 3-sphere $\bbS^3$. We define an equivalence relation by identifying two maps in $M$ if they differ by an orientation preserving diffeomorphism of $\bbT^2$ and an orientation preserving isometry of $\bbS^3$, and set $\mathcal{M} = M/\sim$. We denote the spectral genus zero maps by $\mathcal{M}_0 \subset \mathcal{M}$, that is
\[
    \mathcal{M}_0 = \left\{ \mbox{ equivalence classes of flat \cmc tori in } \bbS^3\,\, \right\}\,.
\]
Thus $\mathcal{M}_0$ consists of infinitely many $\bbR$-families of flat \cmc tori. Even though each of these families by Proposition~\ref{prop:flat-torus} (iii) is a finite cover of the family of the underlying embedded rectangular torus, we need the full diversity of $\mathcal{M}_0$ to bifurcate into all possible spectral genus one \cmc tori. The spectral genus one \cmc tori will be denoted by $\mathcal{M}_1 \subset \mathcal{M}$, that is
\begin{equation*}
    \mathcal{M}_1 = \left\{ \mbox{ equivalence classes of spectral genus one \cmc tori in } \bbS^3\,\, \right\} \,.
\end{equation*}
Since deformation families of rotational tori in $\mathcal{M}_1$ flow into bouquets of spheres, we take the closure of $\mathcal{M}_1$ by supplementing it with the limiting bouquets of spheres, and set
\begin{equation*}
    \overline{\mathcal{M}}_1 = \mathcal{M}_1 \cup \left\{ \mbox{ equivalence classes of sphere bouquets in } \bbS^3\,\, \right\}\,.
\end{equation*}
The aim of this section is to show that this \emph{completed} moduli space of equivariant  \cmc tori $\mathcal{M}_0 \cup \overline{\mathcal{M}}_1$ is connected (Theorem~\ref{th:connected}). We have already seen that the moduli space of equivariant \cmc tori in the 3-sphere is a graph, which consists of:
\begin{enumerate}
\item Edges of families of spectral genus zero tori;
\item Edges of families of spectral genus one tori;
\item 'Bifurcation' vertices in
  $\mathcal{M}_0$ that connect with other vertices in
  $\mathcal{M}_0$ via spectral genus one edges by
  Theorem~\ref{thm:moduli-space}.
\end{enumerate}
By Proposition~\ref{prop:flat-torus} (iii) any element in
$\mathcal{M}_0$ is isogenic to the unique (up to isomorphism)
embedding of a rectangular lattice with the same mean curvature. Hence each edge in $\mathcal{M}_0$ contains a unique minimal torus, obtained via an isogeny from the Clifford torus. If we identify two isogenies which differ only by an isomorphism of the domain, then we have a one-to-one correspondence between isomorphy classes of isogenies and co-finite sublattices of $\Lambda^\ast$. Hence we can identify the connected components of $\mathcal{M}_0$ with co-finite sublattices of $\Lambda^\ast$. We say that two such sublattices are connected, if the corresponding genus zero edges are connected in
$\mathcal{M}_0\cup\overline{\mathcal{M}}_1$.

We associated to the bifurcation vertices triples
$(\TurnZero,\,\TurnOne,\,\TurnTwo) \in \bbZ^3$ with
$\gcd(\TurnZero,\,\TurnOne,\,\TurnTwo)=1$ and
$0\leq\TurnOne < \TurnZero < \TurnTwo$. In the proof of Theorem~\ref{thm:flat-torus-integers}~(2) we showed that triples~\eqref{eq:triple} correspond to the lattices~\eqref{eq:sublattices}. The
genus one edges described in Theorem~\ref{thm:moduli-space} yield
four isomorphisms of each of these lattices onto one of those lattices corresponding to the triples
$(\TurnOne+\TurnTwo-\TurnZero,\,\TurnOne,\,\TurnTwo)$. Furthermore, the genus zero edge corresponding to a sublattice of one of the former lattices is connected by an isogenic genus one edge with the
corresponding sublattice of one of the latter lattices. But we do not use these isomorphisms of the lattices~\eqref{eq:sublattices} onto those lattices corresponding to
$(\TurnOne+\TurnTwo-\TurnZero,\,\TurnOne,\,\TurnTwo)$. We shall make use of these isomorphisms only in case of embedded tori with $\TurnZero=1$ and $\TurnOne = 0$. In this case rotation periods are preserved, and up to transformations $C'$ and $D'$ these isomorphisms are of the form
\begin{equation}\label{eq:isomorphism}
    \Lambda^\ast\to p\,\gamma_1^\ast\bbZ \oplus q\,\gamma_2^\ast\bbZ, \quad
    n_1\gamma_1^\ast+n_2\gamma_2^\ast\mapsto p\,n_1\gamma_1^\ast\oplus
    q\,n_2\gamma_2^\ast\quad\mbox{ with }p=1\mbox{ or }q=1.
\end{equation}
\begin{proposition} \label{th:lattices1}
\mbox{}

{\rm{(i)}} The edge of a lattice $\Gamma \subset \Lambda^\ast$ contains a vertex with triple $(\TurnZero,\,\TurnOne,\,\TurnTwo)$ if and only if $\Gamma$ is a sublattice of one of the lattices in \eqref{eq:sublattices}.

{\rm{(ii)}} For any $\TurnZero,\,\TurnOne,\,\TurnTwo$ we have that
$\Gamma_{[\TurnZero,\,\TurnOne,\,\TurnTwo]} = \Lambda^\ast$
if and only if $\TurnZero = 1$.

{\rm{(iii)}} The unique edge of embedded rotational tori in
$\mathcal{M}_0$ contains only the vertices with triples of the form
$(\TurnZero,\TurnOne,\TurnTwo)=(1,0,\TurnTwo)$ with
$\TurnTwo \geq 2$.
\end{proposition}
\begin{proof}
(i) In the proof of Theorem~\ref{thm:flat-torus-integers}~(2) we determined the lattices that correspond to a triple~\eqref{eq:triple}. They are given in ~\eqref{eq:sublattices}.

(ii) If $\TurnZero = 1$, then $n_1\TurnOne + n_2\TurnTwo \in
\bbZ$ holds for all $(n_1,\,n_2) \in \bbZ^2$. Hence
$\Gamma_{[1,\,\TurnOne,\,\TurnTwo]} = \Lambda^\ast$. Conversely, if
$\Gamma_{[\TurnZero,\,\TurnOne,\,\TurnTwo]} = \Lambda^\ast$, then
$n_1\TurnOne + n_2\TurnTwo \in \TurnZero\bbZ$ holds for all
$(n_1,\,n_2) \in \bbZ^2$. In particular $\TurnOne,\,\TurnTwo
\in\TurnZero \bbZ$, which implies $\TurnZero =1$ so as not to
contradict the assumption $\gcd(\TurnZero,\,\TurnOne,\,\TurnTwo) =1$.

(iii) By Proposition~\ref{thm:flat-torus-integers} we have that
$\TurnOne =0$ in the rotational case. By Theorem~\ref{thm:embedded}
the  $(\TurnZero,\,\TurnTwo)$ torus of revolution is embedded if and
only if $\TurnZero=1$.
\end{proof}
\begin{lemma}\label{rotational connected}
The moduli of rotational \cmc tori in the 3-sphere, supplemented by
bouquets of spheres is connected.
\end{lemma}
\begin{proof}
We will show that any edge of rotational tori in $\mathcal{M}_0$ is
connected to the edge of embedded rotational tori, or equivalently
that any lattice $p\,\gamma_1^\ast\bbZ + q\,\gamma_2^\ast\bbZ$ of a
rotational torus is connected to the $\Lambda^\ast$ lattice. By
Proposition~\ref{th:lattices1}~(iii) the edge of embedded rotational
tori in $\mathcal{M}_0$ contains all the vertices $(1,\,0,\,\TurnTwo)$ with $\TurnTwo \geq 2$. This edge is connected with all the edges that contain the vertices $(\TurnTwo-1,\,0,\,\TurnTwo)$ with $\TurnTwo \geq 2$. Hence the lattice $\Gamma_{[1,\,0,\,\TurnTwo]}=\Lambda^\ast$ is
connected to the lattice $\Gamma_{[\TurnTwo -1,\,0,\,\TurnTwo]}=
\gamma_1^\ast\bbZ+(\TurnTwo-1)\gamma_2^\ast\bbZ$. Furthermore, the
lattice $D\;\Gamma_{[1,\,0,\,\TurnTwo]} = \Lambda^\ast$ is
connected to the lattice $D\;\Gamma_{[\TurnTwo -1,0,\TurnTwo]}=
(\TurnTwo-1)\gamma_1^\ast\bbZ + \gamma_2^\ast\bbZ$. Sublattices
$\Gamma=p\,\gamma_1^\ast\bbZ + \gamma_2^\ast \bbZ\subset\Gamma_{[1,0,\TurnTwo]}$
are connected along genus one edges isogenic to the former genus one
edges with $p\,\gamma_1^\ast\bbZ + (\TurnTwo-1)\,\gamma_2^\ast\bbZ
\subset\Gamma_{[\TurnTwo -1,0,\TurnTwo]}$ by the isomorphism \eqref{eq:isomorphism}.
\end{proof}
In the following we shall combine deformations to pass between
bifurcation vertices in $\mathcal{M}_0$. There are three different
types of deformations:

\ding{192} The deformation through spectral genus one, possibly also passing through bouquets of spheres as described in Theorem~\ref{thm:moduli-space}. In this case we write
   \[ (\TurnZero,\,\TurnOne,\,\TurnTwo)
   \xrightarrow{\text{\ding{192}}}
(\TurnTwo + \TurnOne - \TurnZero,\,\TurnOne,\,\TurnTwo).\]

\ding{193} The deformation along an edge of flat \cmc tori,
  passing from one 'bifurcation' vertex to another one:
  Suppose $(\TurnZero,\,\TurnOne,\,\TurnTwo) \in \bbZ^3$ is such that
  $0<\TurnOne$ and $\TurnTwo<2\TurnZero$. Then from
  $n_1\TurnOne+n_2\TurnTwo \in \TurnZero \bbZ$ get that
  also $n_1(\TurnOne+\TurnZero)+n_2(\TurnTwo-\TurnZero)\in\TurnZero\bbZ$.
  In this case the transformation
  $(\TurnZero,\TurnOne,\TurnTwo)\mapsto
  (\TurnZero,\TurnTwo-\TurnZero,\TurnZero+\TurnOne)$ acts on the
  corresponding lattices as the transformation~(D''), which interchanges
  the four lattices~\eqref{eq:sublattices}. In such a case we write
    \[ (\TurnZero,\TurnOne,\TurnTwo)\xrightarrow{\text{\ding{193}}}
       (\TurnZero,\TurnTwo-\TurnZero,\TurnZero+\TurnOne).\]

\ding{194} If $\TurnOne$ is odd, then we have
\begin{equation} \label{eq:turnone_odd}
n_1\gamma_1^\ast+n_2\gamma_2^\ast\in\Gamma_{[\TurnZero,\TurnOne,\TurnTwo]}
\Longleftrightarrow
2n_1\gamma_1^\ast+n_2\gamma_2^\ast\in \Gamma_{[2\TurnZero,\TurnOne,2\TurnTwo]}.
\end{equation}
Due to Lemma~\ref{rotational connected} both genus zero edges
corresponding to the lattices $2\bbZ\gamma_1^\ast+\bbZ\gamma_2^\ast$
and $\bbZ\gamma_1^\ast+2\bbZ\gamma_2^\ast$ are connected with the edge corresponding to $\Lambda^\ast$. If we apply an isogeny to these families we obtain with \eqref{eq:isomorphism} a deformation of a genus zero edge corresponding to the lattice $\Gamma\subset\Lambda^\ast$ to another genus zero edge corresponding to the lattices $\{2n_1\gamma_1^\ast+n_2\gamma_2^\ast\mid
n_1\gamma_1^\ast+n_2\gamma_2^\ast\in\Gamma\}$ and
$\{n_1\gamma_1^\ast+2n_2\gamma_2^\ast\mid
n_1\gamma_1^\ast+n_2\gamma_2^\ast\in\Gamma\}$, respectively.
In combination with \eqref{eq:turnone_odd} we write
\begin{equation*}
(\TurnZero,\TurnOne,\TurnTwo)\xrightarrow{\text{\ding{194}}}
(2\TurnZero,\TurnOne,2\TurnTwo).
\end{equation*}
\begin{lemma} \label{th:z2lattice} For co-prime integers
  $0\leq\TurnOne<\TurnZero<\TurnTwo$ the
  lattices~\eqref{eq:sublattices} are connected with $\Lambda^\ast$.
\end{lemma}
\begin{proof}
We shall show that it is possible to successively reduce $\TurnOne$ until $\TurnOne = 0$. We can add to $\TurnTwo$ multiples of $\TurnZero$ without changing the lattices \eqref{eq:sublattices}. We pick the smallest of all
possible $\TurnTwo$ and obtain $\TurnTwo\leq2\TurnZero.$ In case of
equality, the lattices \eqref{eq:sublattices} are of the form
$p\bbZ\gamma_1^\ast+q\bbZ\gamma_2^\ast$. In this case
Lemma~\ref{rotational connected} connects $\Gamma$ with
$\Lambda^\ast$. Therefore we may assume $\TurnTwo<2\TurnZero$.

In the deformation \ding{192} we pick the smaller of the first entries $\TurnZero$ and $\TurnTwo+\TurnOne-\TurnZero$. Hence we can assume that $2\TurnZero\leq\TurnOne+\TurnTwo$,
and now we have the following sequence of deformations
\begin{equation*} \label{eq:step1}
(\TurnZero,\TurnOne,\TurnTwo)\xrightarrow{\text{\ding{193}}}
(\TurnZero,\TurnTwo-\TurnZero,\TurnZero+\TurnOne)
\xrightarrow{\text{\ding{192}}}
(\TurnTwo+\TurnOne-\TurnZero,\TurnTwo-\TurnZero,\TurnZero+\TurnOne)
\xrightarrow{\text{\ding{193}}} (\TurnTwo+\TurnOne-\TurnZero,
2\TurnZero-\TurnTwo,\TurnOne+2\TurnTwo-2\TurnZero).
\end{equation*}
If $2\TurnZero<\TurnOne+\TurnTwo$ then $2\TurnZero-\TurnTwo<\TurnOne$,
so $\TurnOne$ has decreased by this deformation.

If $2\TurnZero = \TurnOne + \TurnTwo$, then we distinguish two cases:
If $\TurnOne$ and $\TurnZero - \TurnOne$ were both even, then
$\TurnTwo$ would be even, contradicting that
$\gcd(\TurnZero,\,\TurnOne,\,\TurnTwo) =1$. Hence we just need to
consider the two cases $\TurnZero - \TurnOne$ is odd, and $\TurnOne$
is odd.

If $\TurnZero-\TurnOne$ is odd, then
\[
(\TurnZero,\TurnOne,2\TurnZero-\TurnOne)\xrightarrow{\text{\ding{193}}}
(\TurnZero,\TurnZero-\TurnOne,\TurnZero+\TurnOne)
\xrightarrow{\text{\ding{194}}}
(2\TurnZero,\TurnZero-\TurnOne,2\TurnZero+2\TurnOne)
\xrightarrow{\text{\ding{192}}}
(\TurnZero+\TurnOne,\TurnZero-\TurnOne,2\TurnZero+2\TurnOne).\]

\item If $\TurnOne$ is odd, then
\[
(\TurnZero,\TurnOne,2\TurnZero-\TurnOne)\xrightarrow{\text{\ding{194}}}
(2\TurnZero,\TurnOne,4\TurnZero-2\TurnOne)\xrightarrow{\text{\ding{192}}}
(2\TurnZero-\TurnOne,\TurnOne,4\TurnZero-2\TurnOne).\]
Obviously for $\TurnTwo=2\TurnZero$ the
lattices~\eqref{eq:sublattices} are of the form
$p\bbZ\gamma_1^\ast+q\bbZ\gamma_2^\ast$.
Hence in all cases either the lattices \eqref{eq:sublattices} are
connected with $\Lambda^\ast$ or $\TurnOne$ is reduced. By repeating the above procedure finitely many times we eventually achieve $\TurnOne=0$,
which we have already dealt with in Lemma~\ref{rotational connected}.
\end{proof}
\begin{lemma} \label{th:lattice-lemma}
For every co-finite sublattice $\Gamma \varsubsetneq \Lambda^*$ there exists a triple $(\TurnZero,\,\TurnOne,\,\TurnTwo) \in \bbZ^3$ with
$\gcd(\TurnZero,\,\TurnOne,\,\TurnTwo)=1$, $0\leq\TurnOne < \TurnZero < \TurnTwo$ and $\TurnZero >1$, such that $\Gamma$ is contained in a
lattice of \eqref{eq:sublattices}.
\end{lemma}
\begin{proof}
Since $\Gamma \neq\Lambda^\ast$, there exists
$\gamma_1\in\Lambda^\ast$ with $\gamma_1\notin\Gamma$, and $\gamma_1$ is not a multiple of another element in $\Lambda^\ast$. Let $\gamma_2\in\Lambda^\ast$ such that $\Lambda^\ast\cong\gamma_1\bbZ\oplus\gamma_2\bbZ$. Since
$\Lambda^\ast/\Gamma$ is finite, there exist smallest integers
$p,q\in\bbZ$ with $p\geq 2,\,q\geq 0$ such that $\gamma_1p\in\Gamma$
and $\gamma_1q\oplus\gamma_2\in\Gamma$.

Consider the homomorphism $g :\Lambda^\ast\to\bbZ,\,
\gamma_1l\oplus\gamma_2m\mapsto l-qm$. By definition of $p,q$ we
have that $\Gamma\subset\gamma_1p\bbZ\oplus(\gamma_1q+\gamma_2)\bbZ$,
so that $g$ maps $\Gamma$ to a sublattice of $p\bbZ$. Every such
homomorphism is of the form $a\gamma_1^\ast+b\gamma_2^\ast\in\Lambda^\ast
\mapsto\TurnOne a+\TurnTwo b$ with $\TurnOne,\TurnTwo\in\bbZ$. By
adding appropriate multiples of $\TurnZero=p$ to
$\TurnOne,\TurnTwo$ we can achieve $0\leq\TurnOne<\TurnZero<\TurnTwo$.
\end{proof}
\begin{theorem} \label{th:connected}
The completed moduli space of equivariant \cmc tori in the 3-sphere is connected.
\end{theorem}
\begin{proof}
If $\Gamma \varsubsetneq \Lambda^*$ is a co-finite
sublattice, then by Lemma~\ref{th:lattice-lemma} the corresponding
genus zero edge contains a bifurcation vertex with triple
$(\TurnZero,\,\TurnOne,\,\TurnTwo)$ and integers
$0\leq\TurnOne<\TurnZero<\TurnTwo$. By Lemma~\ref{th:z2lattice}
the corresponding lattices \eqref{eq:sublattices} are connected to the lattice $\Lambda^\ast$. An isogeny of this path connects the edge corresponding to $\Gamma$ with a genus zero edge corresponding to
$\Gamma'$ with $|\Lambda^\ast/\Gamma'|<|\Lambda^\ast/\Gamma|$.
Repeating this argument we can successively reduce the order until
$|\Lambda^\ast/\Gamma'|= 1$.
\end{proof}

\section{Stability}
We conclude the paper by computing the sign of the second variation of the Willmore energy at all spectral genus one minimal tori in $\mathbb{S}^3$, and show that this is negative. Thus we obtain the following
\begin{theorem} Spectral genus one minimal tori in $\mathbb{S}^3$ are all local maxima of the Willmore energy.
\end{theorem}
\begin{proof}
It is proven in Proposition~\ref{th:alpha_less_beta} that $\JK'$ and $\JE'$ satisfy
\[
    1\leq\frac{2\JE'}{1+\jq^2}<\JK'<\frac{\JE'}{|\jq|}\quad
    \mbox{for}\quad 0<|\jq|<1.
\]
This implies
\begin{equation} \label{eq:ineqJK}
\JE^2 - \jq^2 \JK^2 >0\,.
\end{equation}
Since the elliptic integrals $\JK'$ and $\JE'$ are at $\jq=\pm 1$ equal to $\frac{\pi}{2}$, the function $2\JE'-(1+\jq^2)\JK'$ has zeroes only at $\jq =\pm 1$. By \eqref{eq:torus-flow} we have
\[
    \dot{\jh} =
    \jq\frac{1-\jh^2}{1-\jq^2}\,(2\JE'-(1+\jq^2)\JK'))\,.
\]
The conformal factor is $2 \langle f_z,\,f_{\bar{z}} \rangle = v^2 = \mathrm{dn}^2 (x,\,1-q^2)$, and thus
\[
  \int_0^\mathrm{K(m)} \mathrm{dn}^2(t,\,m)\, dt = 2\,\mathrm{E}(m) \,.
\]
The Willmore energy with respect to the simple generators of the lattice computes to
\[
  \calW =  \int (H^2 + 1)\,dA = \frac{4 \pi \sqrt{2}\,\JE }
  {\sqrt{\jq^2 - 2 \jq \left(\jk \jh-\sqrt{1-\jk^2}\sqrt{1-\jh^2}\right) + 1}}\,.
\]
The first derivative of the Willmore Energy is
\[
  \dot{\calW} = \frac{4 \pi \sqrt{2} \jh \jq \left(\JE^2 - \jq^2
   \JK^2\right)}{\sqrt{\jq^2 - 2 \jk \jh \jq + 2 \jq \sqrt{1-\jk^2} \sqrt{1-\jh^2} + 1}} \,.
\]
Thus $\dot{\calW} = 0$ if and only if $H=0$. The second derivative at these extrema computes to
\[
  \left.\ddot{\calW}\right|_{H=0} = \frac{4 \pi \sqrt{2} \jq^2}{\sqrt{\jq^2 + 2\jq \sqrt{1-\jk^2} +1}}
  \left(\jq \JK - \JE \right)\,\left(\jq \JK + \JE \right) \,\left(\jq^2 \JK + \JK - 2 \JE \right)\,.
\]
The product of the first two factors is negative by inequality \eqref{eq:ineqJK}, while for the third factor
we have
\begin{align*}
  &\left.\left(\jq^2 \JK + \JK - 2 \JE \right)\right|_{\jq^2 =1} =0 \,,\\
  &\tfrac{d}{d\jq} \left(\jq^2 \JK + \JK - 2 \JE \right) = \frac{\JE - \jq^2 \JK}{2\jq^2} \geq 0 \,,
\end{align*}
and thus
\begin{equation}
  \left.\ddot{\calW}\right|_{H=0} < 0  \mbox{ for } 0 < |\jq | < 1\,.
\end{equation}
Hence with respect to the deformation
through {\sc{cmc}} tori, the Willmore energy has local maxima at all the minimal spectral genus one tori in $\bbS^3$,
and therefore minimal tori of spectral genus one are unstable extrema.
\end{proof}
%


\bibliographystyle{amsplain}

\def\cydot{\leavevmode\raise.4ex\hbox{.}} \def\cprime{$'$}
\providecommand{\bysame}{\leavevmode\hbox to3em{\hrulefill}\thinspace}
\providecommand{\MR}{\relax\ifhmode\unskip\space\fi MR }
\providecommand{\MRhref}[2]{%
  \href{http://www.ams.org/mathscinet-getitem?mr=#1}{#2}
}
\providecommand{\href}[2]{#2}

\end{document}